\documentclass[a4paper, 10pt, twoside, notitlepage]{amsart}

\usepackage[utf8]{inputenc}
\usepackage{color}
\usepackage{amsmath} 
\usepackage{amssymb} 
\usepackage{amsthm}
\usepackage{geometry}
\usepackage{graphicx}
\usepackage{esint}
\usepackage[colorlinks=true,linkcolor=blue]{hyperref}

\theoremstyle{plain}
\newtheorem{thm}{Theorem}
\newtheorem{prop}{Proposition}[section]
\newtheorem{lem}[prop]{Lemma}
\newtheorem{cor}[prop]{Corollary}
\newtheorem{assume}[prop]{Assumption}
\newtheorem{defi}[prop]{Definition}
\newtheorem{rmk}[prop]{Remark}

\newcommand {\R} {\mathbb{R}} 
 \newcommand {\N} {\mathbb{N}}
\newcommand {\C} {\mathbb{C}} 
\newcommand {\p} {\partial}

\newcommand {\D} {\Delta}

\DeclareMathOperator {\supp} {supp}

\DeclareMathOperator{\di}{div}

\DeclareMathOperator {\dist} {dist}

\DeclareMathOperator{\inte} {int}

\DeclareMathOperator{\F} {\mathcal{F}}

\title{On some partial data Calder\'on type problems with mixed boundary conditions}
\author{Giovanni Covi}
\address{Department of Mathematics and Statistics, University of Jyväskylä, Seminaarinkatu 15, 40014 Jyväskylä, Finland}
\email{giovanni.g.covi@jyu.fi}
\author{Angkana Rüland}
\address{Institut für Angewandte Mathematik, Ruprecht-Karls-Universität Heidelberg, Im Neunheimer Feld 205, 69120 Heidelberg, Germany}
\email{Angkana.Rueland@uni-heidelberg.de}

\begin{document}

\begin{abstract}
In this article we consider the simultaneous recovery of bulk and boundary potentials in (degenerate) elliptic equations modelling (degenerate) conducting media with inaccessible boundaries. This connects local and nonlocal Calder\'on type problems. We prove two main results on these type of problems: On the one hand, we derive simultaneous bulk and boundary Runge approximation results. Building on these, we deduce uniqueness for localized bulk and boundary potentials.
On the other hand, we construct a family of CGO solutions associated with the corresponding equations. These allow us to deduce uniqueness results for arbitrary bounded, not necessarily localized bulk and boundary potentials. The CGO solutions are constructed by duality to a new Carleman estimate.
\end{abstract}

\maketitle

\section{Introduction}
There has been a substantial amount of work on nonlocal inverse problems in the last years (see for instance the survey articles \cite{S17, R18} and the references cited below). These nonlocal equations arise naturally in many problems from applications including, for instance, finance \cite{AB88, Sch03, Lev04}, ecology \cite{RR09, Hum10, MV18}, image processing \cite{GO08}, turbulent fluid mechanics \cite{Con06}, quantum mechanics \cite{Las00, Las18} and elasticity \cite{S89} as well as many other fields \cite{GL97, MK00, Eri02, DGLZ12, AMRT10, DZ10, DG13, Ros15, BV16}. In this article, we provide yet another point of view on these non-local inverse problems by adopting a local ``Caffarelli-Silvestre perspective''. The resulting equations and the associated inverse problems are of interest in their own right, modelling for instance situations in which there are unknown, not-directly measurable \emph{fluxes} or \emph{potentials on the boundary} of an electric device in addition to \emph{electric and/or magnetic potentials in the interior} of it. Moreover, we also include situations in which the conducting property of the (electric) medium may deteriorate or improve towards the boundary.
In this setting of unknown and not directly accessible \emph{boundary and bulk potentials} at possibly \emph{degenerate} conductivities, we are interested in the reconstruction of both of these boundary and bulk potentials which are coupled through possibly degenerate, linear elliptic equations.

\subsection{A model setting}
As a model case, we consider the following problem set-up with non-degenerate conductivities:
Let $\Omega \subset \R^n$ be an open, bounded, $C^2$-regular (or smooth) domain, modelling the conducting body. Assume that $\Sigma_1, \Sigma_2 \subset \partial \Omega$ are two disjoint, relatively open, smooth non-empty sets. Consider the following magnetic Schrödinger equation with mixed boundary conditions
\begin{align}
\label{eq:Schroedinger}
\begin{split}
-\D u - i A \cdot \nabla u - i \nabla \cdot (A u) + (|A|^2 +V) u & = 0  \mbox{ in } \Omega,\\
\p_{\nu} u + qu &= 0 \mbox{ on } \Sigma_1,\\
u & = f \mbox{ on } \Sigma_2,\\
u & = 0 \mbox{ on } \partial \Omega \setminus (\Sigma_1 \cup \Sigma_2),
\end{split}
\end{align}
where, for simplicity, the coefficients are supposed to satisfy the conditions that
\begin{align}
\label{eq:bds}
V \in L^{\infty}(\Omega, \R), \ A\in L^{\infty}(\Omega, \R^n), \ q \in L^{\infty}(\partial \Omega,\R),
\end{align}
and that 
\begin{align}
\label{eq:A1A2}
\nu \cdot A = 0\mbox{ on } \partial \Omega.
\end{align}
In analogy to the setting of the Schrödinger version of the partial data Calder\'on problem we seek to recover the potentials $A,V$ and $q$ from boundary measurements encoded in the (partial) Dirichlet-to-Neumann map
\begin{align*}
\Lambda_{A,V,q}: \widetilde{H}^{\frac{1}{2}}(\Sigma_2) \mapsto H^{- \frac{1}{2}}(\Sigma_2), \ f|_{\Sigma_2} \mapsto \p_{\nu} u|_{\Sigma_2}.
\end{align*}
In a formally correct way this will be defined by means of the bilinear form 
\begin{align}
\label{eq:Bform12}
B_{A,V,q}(u,v):= \int\limits_{\Omega} \nabla u \cdot \overline{\nabla v} + i \overline{v} A \cdot \nabla u  -i  A u \cdot \overline{\nabla v} + V u \overline{v} dx  + \int\limits_{\Sigma_1} q u \overline{v} d \mathcal{H}^{n-1} \mbox{ for } u,v \in H^1(\Omega, \C),
\end{align}
in Definition \ref{defi:DtN} in Section \ref{sec:well_posed1} below. Here $\overline{u}$ denotes the complex conjugate of $u$.

We remark that in contrast to the ``usual'' partial data, magnetic Schrödinger version of the Calder\'on problem, in \eqref{eq:Schroedinger} the first boundary condition yields a new ingredient: Besides the partial data character of the problem which is encoded in the measurement of data on $\Sigma_2$ only, we now also consider a setting in which a part of the domain, $\Sigma_1$, is modelled as inaccessible and on which we also seek to recover an unknown boundary flux/potential. This is closely related to the so-called inverse Robin problem which arises, for instance, in corrosion detection (see \cite{I97} and the references below). We thus combine a Calder\'on with a Robin inverse problem, studying a setting in which in addition to the bulk potentials in the interior of the domain $\Omega$ also unknown boundary potentials and mixed-type boundary conditions are present.

In this framework it is our objective to investigate the following questions:
\begin{itemize}
\item[(Q1)] Let us assume that $A, V, q$ and $\Lambda_{A, V,q}$ are as above. Can we then simultaneously recover the boundary potential $q$, the magnetic potential $A$ and the bulk potential $V$, if the bulk (gradient) potentials $A$ and $V$ are supported in a set $\Omega_1 \Subset \Omega$ which is open and bounded? 
\item[(Q2)] Is this recovery still possible -- at least for $V$ and $q$ -- if the bulk potentials are not compactly supported in $\Omega$? In particular, is this possible, if there is no longer some safety distance between $\Omega_1$ and the boundary parts given by $\Sigma_1$ and $\Sigma_2$?
\end{itemize}

Let us comment on these questions: Both of these are \emph{partial} data problems with the objective of reconstructing unknown potentials \emph{simultaneously} on the boundary and in the bulk (see \cite{KS14} for a survey on the known partial data results). As explained in the sequel, the effect of the boundary and bulk potentials however is expected to differ quite substantially in the context of the inverse problem.

On the one hand, the magnetic and scalar potentials $A$ and $V$ are \emph{local, interior} potentials. The dimension counting heuristics on the recovery of these follow from the ones for the classical Calder\'on problem: One seeks to recover unknown objects of $n$ degrees of freedom from the (partial) Dirichlet-to-Neumann map, an operator which encodes $2(n-1)$ degrees of freedom. Building on the seminal result \cite{SU87}, a canonical tool to address the associated uniqueness question for the ``local'' potentials $A, V$ are complex geometric optics (CGO) solutions. It is further well-known that the presence of the magnetic potential creates additional difficulties due to the resulting gauge invariances. In spite of this, both in the full and the partial data settings, CGO solutions have been constructed starting with the works \cite{NSU95, S93}, see also \cite{C14, CT16}. These however do not cover our mixed-data set-up in which additional unknowns are present on the boundary. 

On the other hand, the heuristics on the recovery of the boundary potential give hope for substantially stronger \emph{boundary} uniqueness results: Indeed, recalling from the argument above that the Dirichlet-to-Neumann operator formally contains $2(n-1)$ degrees of freedom, we note that the recovery of $q$ which is a function of $n-1$ degrees of freedom is always overdetermined. Hence, in analogy to \cite{GRSU18}, even single measurement results for the uniqueness of the boundary data can be expected (see \cite{CJ99, ADPR03} for results of this type for the Robin inverse problem). We view this as a ``non-local'' reconstruction problem at the boundary; a connection to the fractional Calder\'on problem is explained below.

In dealing with the questions (Q1) and (Q2) we thus combine ideas from ``local'' and ``non-local'' inverse problems. Here in our analysis of the question (Q1) the softer ``non-local'' effects dominate, while in our approach towards the problem (Q2), the ``local'' interior effects prevail. In particular, we thus
\begin{itemize}
\item address question (Q1) using simultaneous Runge approximation results in the bulk and on the boundary (see Sections \ref{sec:Q11}-\ref{sec:Runges}),
\item deal with question (Q2) by constructing suitable CGO solutions (see Sections \ref{sec:Carl}-\ref{sec:CGO}).
\end{itemize}

Indeed, in \eqref{eq:Schroedinger} we view the boundary data on $\Sigma_1$ as a local formulation \`a la Caffarelli-Silvestre \cite{CS07} of a Schrödinger equation for the half-Laplacian on $\Sigma_1$. Then, using the fact that in question (Q1) the local interior potentials $A,V$ are only supported in a compact subset of $\Omega$ which has some safety distance to $\Sigma_1, \Sigma_2$, this indicates that the problem can be reduced to a full data type problem by means of Runge approximation results. In order to deal with the interior potentials, we recall the Runge approximation ideas developed in \cite{AU04} and quantified in \cite{RS19}. These allow one to approximate full data CGO solutions in $\Omega_1$ by partial data solutions in the whole domain $\Omega$. Compared to \cite{AU04} in our setting of \eqref{eq:Schroedinger}, we have to deal with the additional challenge that also on the boundary of $\Omega$ an unknown potential is present. However, due to the disjointness of the domains $\Sigma_1$ and $\Sigma_2$ and motivated by the interpretation of the equation on $\Sigma_1$ as a fractional Schrödinger equation, it is possible to prove corresponding simultaneous density results \emph{both in the bulk and on the boundary} (see Proposition \ref{prop:s_Runge}).

In contrast to the setting of the question (Q1), the question (Q2) is dominated by ``local'' effects. Since now $V$ may be supported in the whole domain $\Omega$ and may in particular be supported up to the sets $\Sigma_1, \Sigma_2$, the Runge approximation techniques are no longer applicable in $\Omega$. In order to nevertheless address the uniqueness question, we thus construct CGO solutions. Here we can however not directly make use of the known full/partial data CGO solutions from the magnetic Schrödinger problem, due to the presence of the additional boundary condition on $\Sigma_1$ in \eqref{eq:Schroedinger}. A related difficulty had earlier been addressed in \cite{C14,C15} in the context of partial data problems. However with respect to the setting in \cite{C15} our equation on the boundary imposes an additional challenge in that the potential $q$ is assumed to be \emph{unknown} and the problem is of \emph{mixed-data} type. Thus, aiming at uniqueness results by means of CGO solutions, we construct a new family of CGO solutions which takes into account \emph{both} the unknown bulk and boundary potentials. This relies on new Carleman estimates for a Caffarelli-Silvestre type extension problem (see Proposition \ref{prop:Carl} and Corollary \ref{cor:Carl}).

\subsection{A family of (degenerate) boundary-bulk partial data Schrödinger problems}

Before discussing our main results, let us present a  variation of the problem outlined above in which we also study operators whose conductivities or potentials depend on the distance to the boundary. More precisely, for $s\in (0,1)$ and for the potentials $A, V, q$ satisfying the conditions in \eqref{eq:A1A2} and \eqref{eq:bds}, we consider the following equation
\begin{align}
\label{eq:frac_Schr}
\begin{split}
-\nabla \cdot d^{1-2s} \nabla u - i A d^{1-2s} \cdot \nabla u - i \nabla \cdot (d^{1-2s} A u) + d^{1-2s}(|A|^2+ V) u & = 0 \mbox{ in } \Omega,\\
\lim\limits_{d(x) \rightarrow 0} d^{1-2s} \p_{\nu} u + qu &= 0 \mbox{ on } \Sigma_1, \\
u & = f \mbox{ on } \Sigma_2, \\
u & = 0 \mbox{ on } \partial \Omega \setminus (\Sigma_1 \cup \Sigma_2).
\end{split}
\end{align}
Here $d: \Omega \rightarrow [0,\infty)$ denotes a smooth function which is equal to the distance to the boundary in a neighbourhood of the boundary. If not otherwise explained, all the functions and in particular $u,v$ in the sequel will be complex-valued.
As in the case $s= \frac{1}{2}$ we define an associated (partial) Dirichlet-to-Neumann map as 
\begin{align*}
\Lambda_{s, A, V, q}: \widetilde{H}^{s}(\Sigma_2) \rightarrow H^{-s}(\Sigma_2), \ f|_{\Sigma_2} \mapsto \lim\limits_{d(x) \rightarrow 0} d(x)^{1-2s} \p_{\nu} u|_{\Sigma_2}.
\end{align*}
Again, in a formally precise way it is defined by means of the bilinear form 
\begin{align}
\label{eq:Bforms}
\begin{split}
B_{s,A,V,q}(u,v) & = \int\limits_{\Omega} d^{1-2s} \nabla u \cdot \overline{\nabla v} - d^{1-2s} i \overline{v} A \cdot \nabla u + d^{1-2s} i A u \cdot \overline{\nabla v} + d^{1-2s} (V+|A|^2) u \overline{v} dx\\
& \quad  + \int\limits_{\Sigma_1} q u \overline{v} d\mathcal{H}^{n-1} \mbox{ for } u, v \in H^1(\Omega, d^{1-2s}).
\end{split}
\end{align}
For the equation \eqref{eq:frac_Schr} and the Dirichlet-to-Neumann map \eqref{eq:Bforms} (and a slight variant of it, see \eqref{eq:CGO_model} below) we seek to investigate the analogues of the questions (Q1) and (Q2) for $s \in(0,1)$, i.e. the reconstruction of the scalar, magnetic and boundary potentials from the generalized Dirichlet-to-Neumann map in the cases that the interior potentials are either supported away from the boundary or reach up to the boundary.

These questions share the same type of local and nonlocal features as explained above. However, the relation to the fractional Laplacian may become more transparent. To illustrate this, we recall the Caffarelli-Silvestre extension \cite{CS07} which allows one to compute the fractional Laplacian through a problem of the type \eqref{eq:frac_Schr} in the unbounded domain $\R^{n+1}_+$. To this end, given a function $u\in H^s(\R^n)$ one considers the degenerate elliptic problem
\begin{align}
\label{eq:CS_ext}
\begin{split}
\nabla \cdot x_{n+1}^{1-2s} \nabla \tilde{u} &= 0 \mbox{ in } \R^{n+1}_+,\\
\tilde{u} & = u \mbox{ on } \R^n \times \{0\}.
\end{split}
\end{align}
The fractional Laplacian then turns into the generalized Dirichlet-to-Neumann operator associated with this equation; $(-\D)^s u := c_s \lim\limits_{x_{n+1}\rightarrow 0} x_{n+1}^{1-2s} \p_{n+1} \tilde{u}(x)$. The idea of realizing the fractional Laplacian as a (degenerate) Dirichlet-to-Neumann operator of a local, degenerate elliptic equation has been further extended to rather general variable coefficient settings, see for instance \cite{ST10, CS16}. In this sense, we view the equation \eqref{eq:frac_Schr} and also \eqref{eq:Schroedinger} as a localized proxy for the inverse problem of recovering the potentials $\tilde{A}$, $\tilde{V}$ and $\tilde{q}$ in the fractional Schrödinger equation
\begin{align}
\label{eq:nonlocal}
\begin{split}
(-  (\nabla + i \tilde{A})^2 + \tilde{V})^s u + \tilde{q} u & = 0 \mbox{ in } \tilde{\Omega} \subset \R^{n-1},\\
u & = f \mbox{ on } \tilde{W} \subset \R^{n-1} \setminus \overline{\tilde{\Omega}}, 
\end{split}
\end{align}
from an associated Dirichlet-to-Neumann map. We note that in \eqref{eq:frac_Schr} the set $\Omega \subset \R^n$ plays the role of the extended space $\R^{n+1}_+$ in \eqref{eq:CS_ext}. As a word of caution we however remark that, following the classical formulation of the Caffarelli-Silvestre extension \eqref{eq:CS_ext} as an equation in $n+1$ dimensions, the formulation of the problem \eqref{eq:CS_ext} is shifted by one dimension with respect to our setting in \eqref{eq:frac_Schr}. In contrast to the Caffarelli-Silvestre extension problem associated with \eqref{eq:nonlocal}, \eqref{eq:frac_Schr} has the advantage that we can work in a bounded domain $\Omega$. This allows us to circumvent the discussion of various issues which arise in the inverse problem for the full Caffarelli-Silvestre extension of \eqref{eq:nonlocal}. We emphasize that just as \eqref{eq:frac_Schr} the problem \eqref{eq:nonlocal} has a natural gauge invariance. In particular it represents yet another nonlocal model with gauge invariances besides the ones which had been introduced and analysed in \cite{BGU18,C20a,CLR18, L20}.

\subsection{Main results}

As one of the main results of this article we provide a complete answer (at $L^{\infty}$ regularity) for the uniqueness question in (Q1) in the case $s= \frac{1}{2}$. 

\begin{thm}
\label{thm:Q11/2}
Let $\Omega \subset \mathbb R^n$, $n\geq 3$, be an open, bounded and $C^2$-regular domain. Assume $\Omega_1\Subset \Omega$ is an open, bounded set with $\Omega \setminus \Omega_1$ simply connected and that $\Sigma_1, \Sigma_2 \subset \p\Omega$ are two disjoint, relatively open sets. If the potentials $q_1, q_2 \in L^\infty(\Sigma_1)$, $\,A_1,A_2 \in C^1(\Omega_1, \R^{n})$ and $V_1, V_2 \in L^{\infty}(\Omega_1)$ in the equation \eqref{eq:Schroedinger} are such that $$\Lambda_1:= \Lambda_{A_1, V_1, q_1}=\Lambda_{A_2,V_2,q_2}=:\Lambda_2\;,$$ 
\noindent then $q_1=q_2$, $\;V_1 = V_2$ and $dA_1 = dA_2$.
\end{thm}

This relies on simultaneous approximation results for the bulk and boundary measurements. For instance, restricting first to the case in which $A=0$ and considering the sets
\begin{align*}
S_{V,q}&:=\{u \in L^2(\Omega): \ u \mbox{ is a weak solution to } \eqref{eq:Schroedinger} \mbox{ in } \Omega\},\\
\tilde{S}_{V,q}&:=\{u \in H^1(\Omega_1): \ u \mbox{ is a weak solution to } \eqref{eq:Schroedinger} \mbox{ in } \Omega\} \subset L^2(\Omega_1),
\end{align*}
we prove the following simultaneous boundary and bulk approximation result.

\begin{lem}
\label{lem:sim_dense}
Assume that the conditions from Section \ref{sec:not_sets} hold for $\Omega,\Omega_1$ and $\Sigma_1, \Sigma_2$.
Let $V \in L^{\infty}(\Omega)$, $q \in L^{\infty}(\partial \Omega)$. Then the set 
\begin{align*}
\mathcal{R}_{bb}:=\{(u|_{\Sigma_1}, u|_{\Omega_1})  : \ u|_{\Sigma_1} = P f|_{\Sigma_1} \mbox{ and } u|_{\Omega_1} = P f|_{\Omega_1} \mbox{ with } f\in C_c^{\infty}(\Sigma_2)\} \subset L^2(\Sigma_1) \times L^2(\Omega_1)
\end{align*}
is dense in $L^2(\Sigma_1) \times \tilde{S}_{V,q}$ with the $L^2(\Sigma_1)\times L^2(\Omega_1)$ topology. Here $P$ denotes the Poisson operator from Definition \ref{defi:Poisson}.
\end{lem}

We remark that substantial generalizations are possible for these type of approximation results. These involve both approximations in stronger topologies and more general Schrödinger type operators. We refer to Lemma \ref{Lem:Runge_with_magnetic_potentials} and the discussion in Sections \ref{sec:Q11} and \ref{sec:Runges} for more on this.

Similar approximation results also hold in the setting of the problem \eqref{eq:frac_Schr}, see for instance Proposition \ref{prop:s_Runge}. Furthermore, an analogous uniqueness result as in Theorem \ref{thm:Q11/2} can also proved in this situation, see Theorem \ref{thm:Q1s}. In spite of the degenerate character of the equation \eqref{eq:frac_Schr} this is reduced to the construction of CGO solutions to a non-degenerate Schrödinger type problem and an application of the Runge approximation result.

We next turn to a variant of the problem \eqref{eq:frac_Schr} and investigate the question (Q2) for this model. Here we follow the usual notation from the Caffarelli-Silvestre extension which was also already used in \eqref{eq:CS_ext} and assume that $\Omega \subset \R^{n+1}$ is an open set. We emphasise that we thus increase the dimension of the problem under consideration by one with respect to our discussion of the question (Q1).
Here, in order to simplify the geometric setting which in partial data problems is not uncommon, we assume that
$\overline{\Sigma_1}:= \overline{\Omega}\cap \{x_{n+1} =0\}$ and that $\Sigma_2 = \partial \Omega \setminus \overline{\Sigma}_1$. In contrast of considering \eqref{eq:frac_Schr} we study a slight variation of it. For $q\in L^{\infty}(\Sigma_1) $, $V \in L^{\infty}(\Omega)$ we investigate solutions to
\begin{align}
\label{eq:CGO_model}
\begin{split}
\nabla \cdot x_{n+1}^{1-2s} \nabla u  + V x_{n+1}^{1-2s} u & = 0 \mbox{ in } \Omega,\\
u & = f \mbox{ on } \Sigma_2,\\
\lim\limits_{x_{n+1} \rightarrow 0} x_{n+1}^{1-2s}\p_{n+1} u + q u & = 0 \mbox{ on } \Sigma_1. 
\end{split}
\end{align}
Here, for instance, $f\in C_c^{\infty}(\Sigma_2)$. Using the same ideas as in Section \ref{sec:well_posed1}, it can be shown that this problem is well-posed if zero is not an  eigenvalue with respect to our mixed data setting. Thus, an associated Dirichlet-to-Neumann map can be (formally) defined as the map,
\begin{align*}
\Lambda_{V,q}:  f \mapsto \lim\limits_{x_{n+1}\rightarrow \partial \Omega} x_{n+1}^{1-2s} \p_{\nu} u|_{\Sigma_2}.
\end{align*}
We refer to Section \ref{sec:CGO} for a more detailed discussion of the Dirichlet-to-Neumann map associated with \eqref{eq:CGO_model} and the function spaces it acts on.
Now no longer imposing conditions on the support of $V$, we seek to recover both $V$ and $q$. Since this implies that Runge approximation methods are no longer applicable in the interior of $\Omega$, we instead rely on a new Carleman inequality for the equation \eqref{eq:CGO_model} (see Proposition \ref{prop:Carl} and Corollary \ref{cor:Carl}) and by duality construct CGO solutions from it:

 \begin{prop}\label{CGO_construction}
 Let $\Omega\subset \mathbb R^{n+1}_+$, $n\geq 3$, be an open, bounded smooth domain. Assume that $\Sigma_1 = \partial\Omega \cap (\mathbb R^n \times \{0\})$ is a relatively open, non-empty subset of the boundary, and that $\Sigma_2 = \p\Omega\setminus \overline{\Sigma_1}$. Let $s\in [1/2,1)$ and let $V\in L^{\infty}(\Omega)$ and $q \in L^{\infty}(\Sigma_1)$.
 Then there exists a non-trivial solution $u\in H^1(\Omega,x_{n+1}^{1-2s})$ of the problem
 
 \begin{align} 
 \label{eq:weighted_problem} 
 \begin{split}
\nabla \cdot x_{n+1}^{1-2s} \nabla u + x_{n+1}^{1-2s} V u & = 0 \mbox{ in } \Omega,\\
\lim\limits_{x_{n+1}\rightarrow 0} x_{n+1}^{1-2s} \p_{n+1} u+ q u & = 0 \mbox{ on } \Sigma_1,
\end{split}\end{align}
 
 \noindent of the form $u(x)=e^{\xi' \cdot x'}(e^{ik'\cdot x' + i k_{n+1} x_{n+1}^{2s}}+r(x))$, where $k\in\R^{n+1}$, $\xi'\in \mathbb C^n$ is such that $\xi'\cdot\xi' =0$, $k\cdot \xi'=0$, and
 \begin{itemize}
\item if $s=1/2$, then $\|r\|_{L^2(\Omega)} = O(|\xi'|^{-\frac{1}{2}})$, $\|r\|_{H^1(\Omega)} = O(|\xi'|^{\frac{1}{2}})$ and $\|r\|_{L^2(\Sigma_1)} = O(1)$; 
 \item if $s>1/2$, then $\|r\|_{L^2(\Omega, x_{n+1}^{1-2s})} = O(|\xi'|^{-s})$, $ \|r\|_{H^1(\Omega, x_{n+1}^{1-2s})} = O(|\xi'|^{1-s})$ and $\|r\|_{L^2(\Sigma_1)} = O(|\xi'|^{1-2s})$.
 \end{itemize}
 \end{prop}
 
\begin{rmk}
 \label{rmk:s1/2}
 We remark that by inspection of the proof given in Section \ref{sec:CGO} below, one observes that for $s=\frac{1}{2}$ one only needs to assume that $n\geq 2$ and may work with $\xi'\in \R^{n+1}$ instead of $\xi' \in \R^n$.
 \end{rmk}
 
\begin{rmk}
\label{rmk:unbdd}
Instead of considering CGOs of the form 
\begin{align*}
u(x)=e^{\xi' \cdot x'}(e^{ik'\cdot x' + i k_{n+1} x_{n+1}^{2s}}+r(x)),
\end{align*}
by the same arguments we can also construct CGOs of the form
\begin{align*}
u(x)=e^{\xi' \cdot x'}(e^{ik'\cdot x' - k_{n+1} x_{n+1}^{2s}}+r(x))
\end{align*}
for $k_{n+1}>0$ which thus have some decay behaviour in the $x_{n+1}$-direction in the amplitude.
\end{rmk}

We emphasize that the CGOs here contain new ingredients compared to the classical CGOs in that the amplitude contains the normal contribution $k_{n+1} x_{n+1}^{2s}$ instead of a linear phase. Also, in order to avoid dealing with the non-degeneracy of the equation, with respect to the classical CGOs, we loose one dimension in the case $s \in (\frac{1}{2},1)$, having to restrict ourselves to $n \geq 3$ (and thus $n+1\geq 4$).

Relying on this new family of CGO solutions for $s\in (\frac{1}{2},1)$, we give a complete answer to the question (Q2) for $n\geq 3$: 

\begin{thm}
\label{thm:unique}
Let $\Omega \subset \mathbb R^{n+1}_+$, $n\geq 3$, be an open, bounded and smooth domain. Assume that $\Sigma_1 := \p\Omega \cap \{x_{n+1}=0\}$ and $\Sigma_2 \subset \p\Omega\setminus\Sigma_1$ are two relatively open, non-empty subsets of the boundary such that $\overline{\Sigma_1\cup\Sigma_2}=\p\Omega$. Let $s\in (1/2,1)$. If the potentials $q_1, q_2 \in L^{\infty}(\Sigma_1)$ and $V_1, V_2 \in L^{\infty}(\Omega)$ relative to problem \eqref{eq:CGO_model} are such that $$\Lambda_1:= \Lambda_{s,V_1, q_1}=\Lambda_{s,V_2,q_2}=:\Lambda_2\;,$$ 
\noindent then $q_1=q_2$ and $V_1 = V_2$.
\end{thm}

This provides the first uniqueness result combining both \emph{local} and \emph{nonlocal} features of the described form in Calder\'on type problems. We hope that these ideas could also be of interest in the study of \eqref{eq:nonlocal}.

In the case $s= \frac{1}{2}$, the lack of sufficiently strong boundary decay estimates only allows us to recover $V$ given $q$ (see Proposition \ref{cor:unique} and Remark \ref{rmk:Lopatinskii}). We seek to study improvements of this in future work.

\subsection{Connection to the literature}
\label{sec:lit}
The study of nonlocal fractional Calder\'on type problems has been very active in the last years: After its formulation and the study of its uniqueness properties in \cite{GSU16}, optimal stability and uniqueness in scaling critical spaces have been addressed in \cite{RS17, RS17d}. In \cite{GRSU18} single measurement reconstruction results have been proved, see also \cite{HL19, HL20} for full-data reconstruction results by monotonicity methods. Further, variable coefficient versions were studied in \cite{GLX17,C20} and magnetic potentials were introduced in \cite{BGU18,C20a,CLR18,L20}. We refer to the articles \cite{LLR19,Lin20,L20a,CMR20,RS18,GFR19} for further variants of related nonlocal problems. Reviews for the fractional Calder\'on problem with additional literature can be found in \cite{S17, R18}. 

In all these works, a striking flexibility property of nonlocal equations is used, see also \cite{DSV14,DSV16,RS17a,Rue17a,GFR20}: As a consequence of the antilocality of the fractional Laplacian (see \cite{V93}), one obtains that the set of solutions to a given fractional Schrödinger problem with scaling-critical or subcritical potential in $\Omega$ is already dense in $L^2(\Omega)$. This allows one to prove uniqueness and reconstruction results by means of Runge approximation properties. These often lead to substantially stronger results for the nonlocal inverse problems than the known ones (e.g. partial data, low regularity) for the classical local case. 
Apart from the intrinsic interest in the described effects of anti- and nonlocality, these nonlocal inverse problems are also of relevance in various applications and in order to obtain an improved understanding of the classical, local Calder\'on problem.

By virtue of the Caffarelli-Silvestre extension, the described fractional Schrödinger inverse problems are also closely connected to (degenerate) versions of the Robin inverse problem as proposed and formulated for instance in \cite{KS95,KSV96,SVX98,BCC08}. These problems arise in the indirect detection of corrosion through electrostatic measurements and in thermal imaging techniques. Mathematically, under sufficiently strong regularity conditions on the potentials and the measurement sets, these can be addressed using ideas on unique continuation, see for instance \cite{CFJL03,S07,AS06, C04, BBL16,BCH11,HM19} for uniqueness, stability and reconstruction results on the Robin inverse problem. In contrast to our setting which combines unknown potentials on the boundary and in the bulk, the literature on the inverse Robin boundary problem however typically does \emph{not} consider a combination of these two challenges. Typically, in works on the inverse Robin problem, a setting complementary to the classical Calder\'on problem is studied, where it is assumed that the bulk properties of the material are known, while reconstruction at inaccessible boundaries is explored.

The classical, local Calder\'on problem is a prototypical and well-studied elliptic inverse problem. It had originally been formulated and studied in its linearized version by Calder\'on, see \cite{C80}. For $n\geq 3$ the uniqueness question for the full, nonlinear problem had been solved in the seminal work \cite{SU87} by introducing CGO solutions. For recent, low regularity contributions on uniqueness, we refer to \cite{CR16, HT13, H15}. Also stability \cite{A88}, reconstruction \cite{N88} and partial data \cite{KS14} problems have been addressed. 
We refer to \cite{U09} for a more detailed survey on the results for the Calder\'on problem.

In this article, we seek to combine both effects, local and nonlocal, with the objective of connecting these and providing new perspectives on them. Studying boundary and bulk potentials simultaneously, we thus combine both the local (bulk) effects and the nonlocal (boundary) effects of the two classes of inverse problems described above. By studying the questions (Q1) and (Q2) outlined above, we illustrate that either effect can dominate. Combining the two settings we investigate an interesting model problem in its own right and hope to to derive ideas and results connecting the local and nonlocal realms.

\subsection{Organization of the remainder of the article}
The remainder of the article is organized as follows. In the next section, we introduce our notation and recall some results on weighted Sobolev spaces. Next, we discuss the well-posedness of problems \eqref{eq:Schroedinger} and \eqref{eq:frac_Schr}. Building on this, in Sections \ref{sec:Q11} and \ref{sec:Runges} we address the question (Q1). Here we also provide the proofs of Theorem \ref{thm:Q11/2} and Lemma \ref{lem:sim_dense}. In Section \ref{sec:Carl} we prove a new Carleman estimate for the generalized Caffarelli-Silvestre extension in \eqref{eq:CGO_model}. Arguing by duality, we derive the existence of CGO solutions for these in Section \ref{sec:CGO} and thus present the proof of Proposition \ref{CGO_construction}. Building on this, we
provide the proof of Theorem \ref{thm:unique} there. Last but not least, we provide a proof of the density result of Proposition \ref{prop:dense} in the appendix.

\section{Notation and Auxiliary Results}
\label{sec:notation}

\subsection{Function spaces}
\label{sec:notation_spaces}
In the following we will make use a number of function spaces. Unless explicitly stated, all function spaces consist of complex valued functions. 

\subsubsection{Weighted Sobolev spaces}
We will fix $s\in (0,1)$ and assume that $\Omega \subset \R^n$ is an open, bounded, $C^2$-regular domain. We let $d: \Omega \rightarrow [0,\infty)$ denote a $C^1$-regular function which close to the boundary $\partial \Omega$ measures the distance to $\partial \Omega$ and is extended to $\Omega$ in a $C^1$-regular way. Then we set:
\begin{align*}
L^2(\Omega, d^{1-2s}) &:= \{u: \Omega \rightarrow \C \mbox{ measureable} : \ \|d^{\frac{1-2s}{2}} u\|_{L^2(\Omega)} <\infty \},\\
H^1(\Omega, d^{1-2s}) &:= \{u: \Omega \rightarrow \C \mbox{ measureable} : \ \|d^{\frac{1-2s}{2}} u\|_{L^2(\Omega)} + \|d^{\frac{1-2s}{2}} \nabla u\|_{L^2(\Omega)}<\infty \}.
\end{align*}

We further use the following notation for fractional Sobolev spaces:
\begin{align*}
H^s(\Omega) := \{u|_{\Omega}: \ u \in H^s(\R^n) \},
\end{align*}
and equip it with the quotient topology
\begin{align*}
\|u\|_{H^s(\Omega)}:= \inf\{\|U\|_{H^s(\R^n)}: \ U|_{\Omega} = u\}.
\end{align*}
It will also be convenient to work with functions obtained by completion of smooth functions with compact support:
\begin{align*}
\widetilde{H}^s(\Omega) := \mbox{ closure of } C^{\infty}_c(\R^n) \mbox{ in } H^s(\R^n). 
\end{align*}
We remark that in our setting of sufficiently regular domains, we have that
\begin{align*}
\widetilde{H}^s(\Omega)= H^{s}_{\overline{\Omega}}, 
\end{align*}
where $H^{s}_{\overline{\Omega}}:=\{u \in H^s(\R^n): \ \supp(u)\subset \overline{\Omega}\}$. Working in charts, similar definitions hold for functions on (sub)manifolds.

We recall the following extension and trace estimates which we will be using for the weighted $H^1(\Omega, d^{1-2s})$ spaces. We remark that both Lemmas \ref{lem:trace} and \ref{lem:extension} are not new and had first been proved in \cite{N93}. We only provide a (rough) argument for these for completeness and the convenience of the reader.

\begin{lem}
\label{lem:trace}
Let $\Omega \subset \R^n$ be an open, bounded, $C^2$-regular set and let $u\in H^1(\Omega, d^{1-2s})$. Then there exists a continuous trace operator into $H^{s}(\partial \Omega)$, i.e. $u|_{\partial \Omega}$ exists in a weak sense, coincides with $u|_{\partial \Omega}$ if $u \in C^{\infty}(\Omega)$ and 
\begin{align*}
\|u|_{\partial \Omega}\|_{H^s(\partial \Omega)} \leq C \|u\|_{H^1(\Omega, d^{1-2s})}.
\end{align*}
\end{lem}

\begin{proof}
The claim follows from the flat result (see for instance \cite[Lemma 4.4]{RS17} for this) and a partition of unity. Indeed, using boundary normal coordinates and a partition of unity $\{\eta_k\}$ whose elements have a sufficiently small support we obtain with $u_k = u \eta_k$ and $\tilde{u}_k(x):= u_k \circ \phi_k(x)$ where $\phi_k$ locally maps the boundary of $\Omega$ to the flat boundary $\{x_{n+1}=0\}$
\begin{align*}
& C(\|d^{\frac{1-2s}{2}} \nabla u\|_{L^2(\Omega)} + \|d^{\frac{1-2s}{2}} u\|_{L^2(\Omega)})
\geq \sum\limits_{k=1}^{M} \left( \|d^{\frac{1-2s}{2}} \nabla u_k\|_{L^2(\Omega)} + \|d^{\frac{1-2s}{2}} u_k\|_{L^2(\Omega)} \right)\\
& \geq C^{-1}\sum\limits_{k=1}^{M} \left( \|x_{n+1}^{\frac{1-2s}{2}} \nabla \tilde{u}_k\|_{L^2(\R^{n+1}_+)} + \|x_{n+1}^{\frac{1-2s}{2}} \tilde{u}_k\|_{L^2(\R^{n+1}_+)} \right) \\
& \geq  C^{-1}\sum\limits_{k=1}^{M} \|\tilde{u}_k\|_{H^s(\{x_{n+1}=0\})} \\
& \geq  C^{-1}\sum\limits_{k=1}^{M} \|u_k\|_{H^s(\partial \Omega)} 
\geq C^{-1} \|u\|_{H^s(\partial \Omega)}.
\end{align*}
Here $C>1$ is a generic constant which may change from line to line. In the estimates, we have used that $|\nabla \phi_k|$ can be chosen as small as desired in the support of $u_k$ (by possibly enlarging $M\in \N$) and that $\|\tilde{u}_k\|_{H^s(\{x_{n+1}=0\})} \sim  \|u_k \|_{H^s(\partial \Omega)} $ (see for instance \cite[Theorem 3.23]{McLean}).
\end{proof}

\begin{lem}
\label{lem:extension}
Let $\Omega \subset \R^n$ be an open, bounded, $C^2$-regular set. For $f\in H^{s}(\partial \Omega)$ there exists a continuous extension operator $E_s(f)$ into $H^1(\Omega, d^{1-2s})$, i.e. $E(f)|_{\partial \Omega} = f$ and
\begin{align*}
\|E_s f\|_{H^{1}(\Omega, d^{1-2s})} \leq C \|f\|_{H^s(\partial \Omega)}.
\end{align*}
\end{lem}

\begin{proof}
Again, this follows by relying on a partition of unity $\{\eta_k\}_{k\in \N}$ and a flatting argument. Flattening $\partial \Omega$ by local diffeomorphisms $\phi_k$ with small $C^1$ norm, we consider $\tilde{f}_k := (f \eta_k)\circ \phi_k$. As $\tilde{f}_k$ may be assumed to be compactly supported in $\{x_{n+1}=0\}$, we obtain an extension $\tilde{u}_k$ satisfying the bound
\begin{align}
\label{eq:est_ext_expl}
\|\tilde{u}_k\|_{H^1(\R^n \times [0,4], x_{n+1}^{1-2s})} \leq C \|\tilde{f}_k\|_{H^s(\R^n)}.
\end{align}
One possibility of achieving this is by choosing $\tilde{u}_k$ to be the solution to
\begin{align*}
\nabla \cdot x_{n+1}^{1-2s} \nabla u &= 0 \mbox{ in } \R^{n+1}_+,\\
u & = \tilde{f}_k \mbox{ on } \R^n \times \{0\}.
\end{align*}
We, for instance, refer to the Appendix in \cite{GFR19} for the derivation of the associated estimates of the form \eqref{eq:est_ext_expl}. Finally, using the local diffeomorphisms $\phi_k$ and the behaviour of the $H^s(\partial \Omega)$ and $H^1(\R^{n+1}_+, x_{n+1}^{1-2s})$ norms under diffeomorphisms, the estimate \eqref{eq:est_ext_expl} turns into a corresponding estimate in $\Omega$. Defining $u:= \sum\limits_{k=1}^{M} \eta_k \tilde{u}_k \circ \phi_k^{-1}$ then concludes the proof. 
\end{proof}

With the trace estimates in hand, we further define the following spaces including boundary data. To this end, let $\Sigma \subset \partial \Omega$ be a $C^2$-regular, relatively open set. Then,

\begin{align*}
H^1_{\Sigma,0}(\Omega, d^{1-2s}) &:= \{u: \Omega \rightarrow \C: \ \|d^{\frac{1-2s}{2}} u\|_{L^2(\Omega)} + \|d^{\frac{1-2s}{2}} \nabla u\|_{L^2(\Omega)}<\infty , \ u|_{\Sigma}=0\}.
\end{align*}

\subsubsection{Test functions for the CGO construction}
\label{sec:CGO_test}
In addition to the weighted Sobolev spaces from above, in our construction of complex geometric optics solutions, we will make use of the following function spaces: For $\Omega \subset \R^{n+1}_+$ with $\overline{\Omega}\cap \{x_{n+1}= 0\}:= \overline{\Sigma}$, we set
\begin{align*}
x_{n+1}^{2s} C_c^{\infty}(\overline{\Omega}):=
\{ f: \Omega \rightarrow \C: f(x) = x_{n+1}^{2s} h(x',x_{n+1}) \mbox{ with } h \in C_c^{\infty}(\overline{\Omega}) \},
\end{align*}
where we use the notation 
\begin{align*}
C_c^{\infty}(\overline{\Omega}):=\{h : \Omega \rightarrow \C: \ h \mbox{ is infinitely often continuously differentiable and } \supp(h) \subset \Omega \cup \Sigma \}. 
\end{align*}
We stress that this in particular enforces that $h, \p_{\nu} h = 0$ on $\partial \Omega \setminus \overline{\Sigma}$ but that $h$ does \emph{not} necessarily vanish on $\overline{\Sigma}$.

For $\Sigma_1 \subset \partial \Omega$ a smooth, $n$-dimensional, star-shaped set, we further consider
\begin{align*}
\tilde{\mathcal{C}}& :=\{f: \Omega \rightarrow \R: \ f \in C^{\infty}(\overline{\Omega}) \mbox{ with } f|_{N(\Sigma_{2},\epsilon)} = 0, \ \p_{\nu} f|_{N(\Sigma_{2},\epsilon)} = 0 , \\
& \qquad \ \p_{n+1} f(x)=0 \mbox{ for } x\in N(\Sigma_1,\epsilon) \times [0,t] \mbox{ for some } \epsilon>0, t>0 \}.
\end{align*}
For simplicity of notation, we have set $\Sigma_2:= \partial \Omega \setminus \overline{\Sigma_1}$ and denote by $N(\Sigma_{2},\epsilon)$, $N(\Sigma_{1},\epsilon)$ an $\epsilon$-neighbourhood of $\Sigma_1, \Sigma_2$ on $\partial \Omega$.

As an important property which we will make use of in our construction of CGO solutions, we state a density result for the space $\mathcal{C}$:

\begin{prop}
\label{prop:dense}
Assume that the conditions from above hold. Then the set $\tilde{\mathcal{C}} \subset H^{1}_{\Sigma_2,0 }(\Omega, x_{n+1}^{1-2s})$ is dense.
\end{prop}

We postpone the proof of Proposition \ref{prop:dense} to the appendix.

Finally, for $s\in (0,1)$ we define
\begin{align}
\label{eq:C}
\mathcal{C} := \tilde{\mathcal{C}} + x_{n+1}^{2s}C_c^{\infty}(\overline{\Omega}).
\end{align}
We will use this space extensively in Section \ref{sec:CGO}.

\subsubsection{Semiclassical spaces and the Fourier transform}
In our construction of CGOs it will be useful to work with semiclassical Sobolev spaces. To this end, we use the following notation for the Fourier transform  
$$ \hat u(y) = \mathcal{F} u(y) = \int_{\mathbb R^{n+1}} e^{-ix\cdot y} u(x) dx\;. $$

We introduce the following definitions for the semiclassical Sobolev spaces. Let $\xi' \in \mathbb C^n$. Eventually we will consider the limit case $|\xi'|\rightarrow \infty$, and thus for us $|\xi'|^{-1}$ constitutes a small parameter. Following \cite{Zw12}, we define the semiclassical Fourier transform as
$$ \mathcal{F}_{sc} u(y) := \int_{\mathbb R^{n+1}} e^{-i |\xi'| x \cdot y} u(x) dx\;,$$
and then use it in order to define the semiclassical Sobolev norm
$$ \|u\|^2_{H^s_{sc}(\mathbb R^{n+1})} := \left(\frac{|\xi'|}{2\pi}\right)^{n+1} \| \langle y \rangle^s |  \mathcal{F}_{sc} u(y) | \|^2_{L^2(\mathbb R^{n+1})}\;,$$
\noindent where $s\in\mathbb R$, $u\in L^2(\mathbb R^{n+1})$ and $\langle y\rangle := (1+|y|^2)^{1/2}$ for $y\in \mathbb R^{n+1}$. The cases of interest for us are $s=0$ and $s=1$, for which we have
$$ \|u\|_{L^2_{sc}(\mathbb R^{n+1})} = \|u\|_{L^2(\mathbb R^{n+1})} \; \quad\mbox{and}\quad \|u\|_{H^1_{sc}(\mathbb R^{n+1})} = \|u\|_{L^2(\mathbb R^{n+1})} + |\xi'|^{-1} \|\nabla u\|_{L^2(\mathbb R^{n+1})}\;.$$
The semiclassical Sobolev spaces $L^2_{sc}(\mathbb R^{n+1})$ and $H^1_{sc}(\mathbb R^{n+1})$ are then defined as the subspaces of $L^2(\mathbb R^{n+1})$ where the corresponding semiclassical norms are finite. Moreover, if $\Omega$ is some open subset of $\mathbb R^{n+1}$ and $w(x)$ is a weight function, we define the weighted semiclassical Sobolev space $H^s_{sc}(\Omega,w)$ as the subspace of $L^2(\mathbb R^{n+1})$ where the norm
$$ \|u\|^2_{H^s_{sc}(\Omega,w)} := \left(\frac{|\xi'|}{2\pi}\right)^{n+1} \| \langle y \rangle^s |  \mathcal{F}_{sc} u(y) | \|^2_{L^2(\Omega,w)}\; $$
is finite. In the special cases $s=0$ and $s=1$ this of course gives
$$ \|u\|_{L^2_{sc}(\Omega,w)} = \|u\|_{L^2(\Omega,w)} \; \quad\mbox{and}\quad \|u\|_{H^1_{sc}(\Omega,w)} = \|u\|_{L^2(\Omega,w)} + |\xi'|^{-1} \|\nabla u\|_{L^2(\Omega,w)}\;.$$

\subsection{Trace estimates}

In this section we collect a number of (weighted) trace estimates. These are not new and have already been used in for instance \cite{Rue15, Rue17, RW19}.

We begin with the case $s= \frac{1}{2}$:

\begin{lem}
\label{lem:boundary_trace}
Let $\Omega \subset \R^n$ be an open, bounded, $C^2$-regular domain. Then there exist constants $C=C(\Omega, \partial \Omega)>1$, $c_0=c_0(\Omega)>1$ such that for all $u \in H^1(\Omega)$ and $\mu \geq c_0$ it holds
\begin{align*}
\|u\|_{L^2(\partial \Omega)} \leq C(\mu^{-1}\|\nabla u\|_{L^2(\Omega)} + \mu \|u\|_{L^2(\Omega)}).
\end{align*}
\end{lem} 

\begin{proof}
By density, we may without loss of generality assume that $u$ is smooth. We work in boundary normal coordinates and denote the coordinates by $x=(x',t)$, where $x = x' + \tau \nu(x')$, $x' \in \partial \Omega$ and $\nu(x')$ denotes the inner unit normal to $\partial \Omega$ at $x'$. By the fundamental theorem, we thus write for some $t>0$
\begin{align*}
u(x', 0) = \int\limits_{0}^t \p_s u(x', s) ds + u(x',t).
\end{align*}
As a consequence, by Hölder,
\begin{align*}
|u(x',0)|^2 \leq C (t \int\limits_{0}^t |\p_s u(x', s)|^2 ds  + |u(x',t)|^2).
\end{align*}
Integrating over $x'\in \partial \Omega$ thus yields
\begin{align*}
\|u\|_{L^2(\partial \Omega)}^2
&\leq C t \int\limits_{\partial \Omega}\int\limits_{0}^t |\p_s u(x', s)|^2 ds dx'  + \|u(\cdot,t)\|_{L^2(\partial \Omega)}^2\\
&\leq C t \|\nabla u\|_{L^2(\Omega)}^2  + \|u\|_{L^2(\partial \Omega_t)}^2.
\end{align*}
Integrating in $t\in (0,\mu^{-1})$ with $\mu \geq C_0(\Omega)>0$ leads to
\begin{align*}
\mu^{-1}\|u\|_{L^2(\partial \Omega)}^2
\leq C \mu^{-2} \|\nabla u\|_{L^2(\Omega)}^2  +  \|u\|_{L^2( \Omega)}^2.
\end{align*}
Multiplying by $\mu>0$ implies the desired result.
\end{proof}

More generally, also a weighted trace estimate holds for $s\in (0,1)$:

\begin{lem}
\label{lem:boundary_trace_s}
Let $\Omega \subset \R^n$ be an open, bounded, $C^2$-regular domain. Let $d: \Omega \rightarrow [0,\infty)$ be a $C^1$ regular function which close to the boundary $\partial \Omega$ coincides with the distance function to $\partial \Omega$. Then there exist constants $C=C(\Omega, \partial \Omega)>1$, $c_0=c_0(\Omega)>1$ such that for all $u \in H^1(\Omega)$ and $\mu \geq c_0$ it holds
\begin{align*}
\|u\|_{L^2(\partial \Omega)} \leq C(\mu^{-s}\|d^{\frac{1-2s}{2}} \nabla u\|_{L^2(\Omega)} + \mu^{1-s} \|d^{\frac{1-2s}{2}} u\|_{L^2(\Omega)}).
\end{align*}
\end{lem}

\begin{proof}
As above, we only prove the result for $u$ smooth and work in boundary normal coordinates $x=(x',t)$ as in the proof of Lemma \ref{lem:boundary_trace}. Again the fundamental theorem in combination with Hölder's inequality, yields
\begin{align*}
|u(x',0)| &\leq \int\limits_{0}^t |\p_r u(x', r)| dr + |u(x',r)|\\
&\leq C_s t^{s}\left( \int_{0}^{r} r^{1-2s} |\p_r u(x', r)|^2 dr \right)^{\frac{1}{2}} + |u(x',r)|.
\end{align*}
Squaring this and integrating in the tangential coordinates yields for $t = t(\Omega)>0$ sufficiently small that
\begin{align*}
\|u\|_{L^2(\partial \Omega)}^2
\leq C t^{2s} \|t^{\frac{1-2s}{2}} \p_t u \|_{L^2(\Omega)}^2 + C \|u\|_{L^2(\partial \Omega_t)}^2.
\end{align*}
Integrating this in $t \in [r,2r]$ for $r \in (0,r_0)$ and $r_0=r_0(\Omega)>0$ entails that
\begin{align*}
r \|u\|_{L^2(\partial \Omega)}^2
&\leq C r^{2s+1}\|t^{\frac{1-2s}{2}} \p_t u\|_{L^2(\Omega)}^2 + C \|u\|_{L^2(\Omega)}^2\\
&\leq C r^{2s+1}\|t^{\frac{1-2s}{2}} \p_t u\|_{L^2(\Omega)}^2 + C r^{2s-1}\|t^{\frac{1-2s}{2}} u\|_{L^2(\Omega)}^2\\
&\leq C r^{2s+1}\|d^{\frac{1-2s}{2}} \nabla u\|_{L^2(\Omega)}^2 + C r^{2s-1}\|d^{\frac{1-2s}{2}} u\|_{L^2(\Omega)}^2.
\end{align*}
Dividing by $r>0$ and defining $\mu^{-1}=r$, we obtain the desired result for $\mu\geq r_0^{-1}$.
\end{proof}

\subsection{Notation for sets}
\label{sec:not_sets}

In the following we will work with Calder\'on type problems with mixed boundary conditions. To this end, we will use the following notation in the remainder of the article. In Sections \ref{sec:well_posed1}-\ref{sec:Runges} we will always assume that $\Omega \subset \R^{n+1}_+$ is a relatively open, $C^2$-regular set. Furthermore, the sets $\Sigma_1, \Sigma_2 \subset \partial \Omega$ are $C^2$-regular and satisfy $\Sigma_1\cap \Sigma_2 = \emptyset$. For the sake of simplicity, in the sequel we will always assume that $\Omega_1 \Subset \Omega$ is a bounded, open set such that $\Omega \setminus \overline{\Omega_1}$ is simply connected. In Sections \ref{sec:Carl} and \ref{sec:CGO} we will in addition assume that all sets are smooth and that $\Sigma_1$ is star-shaped.

Working with sets in the neighbourhood of $\partial \Omega$ or with some distance to $\partial \Omega$, we further define for $\delta \in (0,1)$ sufficiently small
\begin{align}
\label{eq:boundary_not}
\begin{split}
\Omega_{\delta} & := \{x \in \Omega: \ \dist(x, \partial \Omega)\geq \delta\},\\
\partial \Omega_{\delta} &:=  \{x + t\nu(x): \ x \in \partial \Omega,\ t \in (0,\delta)\} \subset \Omega.
\end{split}
\end{align}
Here for $x\in \partial \Omega \subset \R^n$ the vector $\nu(x) \in S^{n-1}$ denotes the inner unit normal at the point $x$.
For a subset $\Sigma \subset \partial \Omega$ we further set
\begin{align*}
N( \Sigma, \delta)&:=\{x \in \p\Omega: \ \dist(x,\Sigma)\leq \delta\}.
\end{align*}

\section{Well-Posedness of the Mixed Boundary Value Problems \eqref{eq:Schroedinger} and \eqref{eq:frac_Schr}}
\label{sec:well_posed1}

In this section, we discuss the well-posedness of the (weak) forms of the equations \eqref{eq:Schroedinger} and \eqref{eq:frac_Schr} in the associated energy spaces. Based on this, we define the associated Dirichlet-to-Neumann maps and derive the central Alessandrini identities which we will use in the following sections when dealing with the associated inverse problems.

We begin by discussing the well-posedness of the problem \eqref{eq:Schroedinger}.

\begin{prop}[Well-posedness, $s=\frac{1}{2}$]
\label{prop:well-posed_12}
Let $B_{A,V,q}$ denote the bilinear form from \eqref{eq:Bform12} and let $\Omega, \Sigma_1, \Sigma_2$ be as above. 
Then, there exists a countable set $M \subset \C$ such that if $\lambda \in \C \setminus M$, for all $F \in (H^1_{\partial \Omega \setminus \Sigma_1,0}(\Omega))^{\ast}$, $f_2 \in H^{\frac{1}{2}}_{\overline{\Sigma_2}}$ and $f_1 \in H^{-\frac{1}{2}}(\Sigma_1)$, there is $u \in H^{1}_{\partial \Omega \setminus (\Sigma_1 \cup \Sigma_2),0}(\Omega)$ with $u|_{\Sigma_2}= {f_2}$ and with
\begin{align}
\label{eq:weak_form}
B_{A,V,q}(u,v) - \lambda(u,v)_{L^2(\Omega)} = \langle F, v \rangle + (f_1,v)_{L^2(\Sigma_1)} ,
\end{align}
for all $ v\in H^{1}_{\partial \Omega \setminus \Sigma_1,0}(\Omega)$. Here $\langle \cdot, \cdot \rangle$ denotes the $ (H^{1}_{\partial \Omega \setminus \Sigma_1,0}(\Omega))^{\ast}$, $ H^{1}_{\partial \Omega \setminus \Sigma_1,0}(\Omega)$ duality pairing.

If $\lambda \notin M$, there exists a constant $C>0$ such that
\begin{align}
\label{eq:apriori1/2}
\|u\|_{H^1(\Omega)} \leq C (\|F\|_{(H^1_{\partial \Omega \setminus \Sigma_1,0}(\Omega))^{\ast}} +  \|f_1\|_{H^{-\frac{1}{2}}(\Sigma_1)} + \|f_2\|_{H^{\frac{1}{2}}(\Sigma_2)}).
\end{align}
\end{prop}

\begin{rmk}
\label{rmk:sol1/2}
We remark that in Proposition \ref{prop:well-posed_12}, compared to the problem in \eqref{eq:Schroedinger}, we consider the slightly more general setting of constructing (weak) solutions to
\begin{align}
\label{eq:Schroedinger_gen}
\begin{split}
L_{\lambda}u:=-\D u - i A \cdot \nabla u - i \nabla \cdot (A u) + (|A|^2 +V + \lambda) u  & = F  \mbox{ in } \Omega,\\
\p_{\nu} u + qu &= f_1 \mbox{ on } \Sigma_1,\\
u & = f_2 \mbox{ on } \Sigma_2,\\
u & = 0 \mbox{ on } \partial \Omega \setminus (\Sigma_1 \cup \Sigma_2).
\end{split}
\end{align}
This will be convenient when discussing density properties by studying the adjoint equation (see Section \ref{sec:Q11}).

For $\lambda \notin M$, we will refer to solutions of \eqref{eq:weak_form} with the described properties as \emph{weak solutions} to \eqref{eq:Schroedinger_gen}. It is this notion of a solution that we will work with in the sequel.
\end{rmk}

\begin{proof}
We argue in several steps.
\medskip

\emph{Step 1: Reduction.}
We first reduce the problem to the case of $f_2 =0$ by considering $u = u_1 + E(f_2)$ where $E(f_2)$ is an $H^1_{\partial \Omega \setminus \Sigma_2, 0}(\Omega)$ extension of $f_2$ satisfying the bound $\|E(f_2)\|_{H^1(\Omega)} \leq C \|f_2\|_{H^{\frac{1}{2}}(\Sigma_2)}$. This is possible by for instance defining $E(f_2)$ to be the harmonic extension of $f_2$ into $\Omega$. The function $u_1$ thus solves a  similar problem as the original function $u$ with a new functional $\tilde{F}:= F - L_{\lambda}(E(f_2)) \in (H^1_{\partial \Omega \setminus \Sigma_1,0}(\Omega))^{\ast}$, but now in addition satisfies $\tilde{f}_2:=u_1|_{\Sigma_2}=0$. Here the expression $L_{\lambda}(E(f_2))$ is understood in the weak sense, i.e. as the functional $H^{1}_{\partial \Omega \setminus \Sigma_1,0}(\Omega) \ni v\mapsto -B_{A,V,q}(E(f_2),v)$.
With slight abuse of notation, in the following we will only work with the function $u_1$ and drop the subindex in the notation for $u_1$ and the tildas in the data.
\medskip

\emph{Step 2: Continuity.}

We observe that for $v \in H^1_{\partial \Omega \setminus \Sigma_1,0}(\Omega)$ as above, we have (using the trace inequality)
\begin{align*}
|B_{A,V,q}(u,v)| \leq C \|u\|_{H^1(\Omega)} \|v\|_{H^1(\Omega)}.
\end{align*}
Here the constant $C>0$ depends on $\lambda, \|q\|_{L^{\infty}}, \|A\|_{L^{\infty}}, \|V\|_{L^{\infty}}$.
This proves the continuity of the bilinear form.

\medskip

\emph{Step 3: Coercivity.}
We next study the coercivity properties of the bilinear form. By Cauchy-Schwarz
\begin{align*}
\left| \int\limits_{\partial \Omega} q u \overline{v} d\mathcal{H}^{n-1} \right| \leq  \|q\|_{L^{\infty}(\partial \Omega)} \|u\|_{L^{2}(\partial \Omega)} \|v\|_{L^{2}(\partial \Omega)}.
\end{align*}
Thus, by the trace inequality from Lemma \ref{lem:boundary_trace} we infer that
\begin{align*}
\left| \int\limits_{\partial \Omega} q |u|^2 d\mathcal{H}^{n-1} \right|
\leq  C\|q\|_{L^{\infty}(\partial \Omega)} \|u\|_{L^2(\partial \Omega)}^2
\leq C\|q\|_{L^{\infty}(\partial \Omega)}(\mu^{-2}\|\nabla u\|_{L^2(\Omega)}^2 + \mu^2 \|u\|_{L^2(\Omega)}^2).
\end{align*}
Choosing $\mu>1$ such that $C\|q\|_{L^{\infty}(\partial \Omega)}\mu^{-2} \leq \frac{1}{10}$, we thus obtain
\begin{align*}
\left| \int\limits_{\partial \Omega} q |u|^2 d\mathcal{H}^{n-1} \right|
\leq \frac{1}{10}\|\nabla u\|_{L^2(\Omega)}^2 + C\|q\|_{L^2(\partial \Omega)} \|u\|_{L^2(\Omega)}^2.
\end{align*}
Moreover, by Young's inequality,
\begin{align*}
\left| \int\limits_{\Omega} u A \cdot \overline{\nabla u} dx\right|
\leq  \frac{1}{10} \|\nabla u\|_{L^2(\Omega)}^2 + C \|A\|_{L^{\infty}(\Omega)}^2 \|u\|_{L^2(\Omega)}^2.
\end{align*} 
Noting that by Poincar\'e's inequality there exists a constant $C>0$ such that for all $u\in H^{1}_{\partial \Omega \setminus \Sigma_1,0}(\Omega)$ we have
\begin{align*}
C\|\nabla u\|_{L^2(\Omega)} \geq \|u\|_{H^1(\Omega)},
\end{align*}
and combining this with the previous estimates for the lower order bulk and boundary contributions, we thus obtain that for $\mu = \|V_-\|_{L^{\infty}(\Omega)} + C_1 \|A\|_{L^{\infty}(\Omega)}  + C_1 \|q\|_{L^{\infty}(\partial \Omega)}$ with suitable constants $C_1,C_2>0$, we have
\begin{align*}
B_{\mu}(u,u):= B_{A,V,q}(u,u) + \mu (u,u)_{L^2(\Omega)} \geq C_2\|u\|_{H^1(\Omega)}^2.
\end{align*}
\medskip

\emph{Step 4: Conclusion.}
By the discussion in Steps 2 and 3 above, $B_{\mu}(\cdot, \cdot)$ is a scalar product and the Riesz representation theorem is applicable. Since $F \in (H^1_{\partial \Omega \setminus \Sigma_1,0}(\Omega))^{\ast}$ and also for $f_1 \in H^{-\frac{1}{2}}(\Sigma_1)$ the map
\begin{align*}
H^1_{\partial \Omega \setminus \Sigma_1, 0}(\Omega) \ni v \mapsto (v, f_1)_{L^2(\Sigma_1)},
\end{align*}
is a bounded linear functional on $H^1_{\partial \Omega \setminus \Sigma_1, 0}(\Omega)$, this yields the existence of a unique function $u:= G_{\mu}(F, f_1)$ such that
\begin{align*}
B_{\mu}(u,v) = \langle F, v \rangle + (f_1, v)_{L^2(\Sigma_1)} \mbox{ for all } v \in H^1_{\partial \Omega \setminus \Sigma_1,0}(\Omega).
\end{align*}
Moreover, the operator 
\begin{align*}
G_{\mu}: (H^1_{\partial \Omega \setminus \Sigma_1,0}(\Omega))^{\ast} \times H^{-\frac{1}{2}}(\Sigma_1) \rightarrow H^1_{\partial \Omega \setminus \Sigma_1,0}(\Omega)
\end{align*}
is bounded.
Now, the equation
\begin{align*}
B_{A,V,q}(u,v) - \lambda(u,v) = \tilde{F}(v)
\end{align*}
with $v$ as above and $\tilde{F}$ a functional on this space, is equivalent to 
\begin{align}
\label{eq:compact_perturb}
u= G_{\mu}((\mu + \lambda) u + \tilde{F}).
\end{align}
 As $G_{\mu}: L^2(\Omega) \times L^2(\Sigma_1) \rightarrow L^2(\Omega)$ is compact and self-adjoint, the spectral theorem for compact, self-adjoint operators yields the existence of a set $\tilde{M}$ such that for $\lambda \notin \tilde{M}$ \eqref{eq:compact_perturb} is (uniquely) solvable. Hence, the original equation is (uniquely) solvable outside of the set $M:=\left\{ \frac{1}{\lambda_j + \mu} \right\}_{j=1}^{\infty}$.
\end{proof}

With the well-posedness result available, it is possible to define the Poisson operator associated with the equation \eqref{eq:Schroedinger}.

\begin{defi}
\label{defi:Poisson}
Let $M\subset \C$ be as in Proposition \ref{prop:well-posed_12} and assume that $0 \notin M$.
Let $f \in H^{\frac{1}{2}}_{\overline{\Sigma_2}}$ and let $u \in H^1(\Omega)$ be the solution constructed in Proposition \ref{prop:well-posed_12} with $F=0, f_1 = 0$ and $f_2=f$. Then, we define the \emph{Poisson operator} 
\begin{align*}
P: H^{\frac{1}{2}}_{\overline{\Sigma_2}} \rightarrow H^1(\Omega), \ f \mapsto u.
\end{align*}
\end{defi}

We remark that by the apriori estimates from Proposition \ref{prop:well-posed_12} the Poisson operator is bounded.

\medskip

With the well-posedness of \eqref{eq:Schroedinger} at our disposal, we proceed to the well-posedness of the equation \eqref{eq:frac_Schr}. 

\begin{prop}[Well-posedness, $s\in (0,1)$]
\label{prop:well-posed_s}
Let $B_{s,A,V,q}$ denote the bilinear form from \eqref{eq:Bforms}. Then, there exists a countable set $M \subset \C$ such that if $\lambda \in \C \setminus M$ for any $F \in (H^1_{\partial \Omega \setminus \Sigma_1,0}(\Omega,d^{1-2s}))^{\ast}$, $f_2 \in H^{s}_{\overline{\Sigma_2}}$ and $f_1 \in H^{-s}(\Sigma_1)$ there is $u \in H^{1}_{\partial \Omega \setminus (\Sigma_1 \cup \Sigma_2),0}(\Omega,d^{1-2s})$ with $u|_{\Sigma_2}= f_2$ and with 
\begin{align*}
B_{s,A,V,q}(u,v) - \lambda(u,v)_{L^2(\Omega)} = \langle F, v \rangle + (f_1,v)_{L^2(\Sigma_1)} 
\end{align*}
for all $v\in H^{1}_{\partial \Omega \setminus \Sigma_1,0}(\Omega, d^{1-2s})$. Here $\langle \cdot, \cdot \rangle$ denotes the $(H^1(\Omega, d^{1-2s}))^{\ast}$, $H^1(\Omega, d^{1-2s})$ duality pairing.
 If $\lambda \notin M$, there exists a constant $C>0$ such that
\begin{align}
\label{eq:aprioris}
\|u\|_{H^1(\Omega, d^{1-2s})} \leq C (\|F\|_{(H^1_{\partial \Omega \setminus \Sigma_1, 0}(\Omega,d^{1-2s}))^{\ast}} + \|f\|_{H^{s}(\Sigma_2)}).
\end{align}
\end{prop}

\begin{rmk}
\label{rmk:sols}
As in Proposition \ref{prop:well-posed_12}, compared to the problem in \eqref{eq:frac_Schr}, we here consider the slightly more general setting of constructing (weak) solutions to
\begin{align}
\label{eq:Schroedinger_gen_s}
\begin{split}
-\nabla \cdot d^{1-2s} \nabla u - i A d^{1-2s} \cdot \nabla u - i \nabla \cdot (d^{1-2s} A u) + d^{1-2s}(|A|^2+ V) u & = F \mbox{ in } \Omega,\\
\lim\limits_{d(x) \rightarrow 0} d^{1-2s} \p_{\nu} u + qu &= f_1 \mbox{ on } \Sigma_1, \\
u & = f_2 \mbox{ on } \Sigma_2, \\
u & = 0 \mbox{ on } \partial \Omega \setminus (\Sigma_1 \cup \Sigma_2).
\end{split}
\end{align}
Again this will be convenient when discussing density properties by means of studying the adjoint equation.
For convenience of notation, we define
\begin{align*}
L_{\lambda,s}u:=-\nabla \cdot d^{1-2s} \nabla u - i A d^{1-2s} \cdot \nabla u - i \nabla \cdot (d^{1-2s} A u) + d^{1-2s}(|A|^2+ V) u .
\end{align*}
As in the case $s= \frac{1}{2}$, for $\lambda \notin M$ we define a \emph{weak solution} to \eqref{eq:Schroedinger_gen_s} to be the corresponding function $u \in H^1_{\partial \Omega \setminus (\Sigma_1\cup \Sigma_2),0}(\Omega, d^{1-2s})$ from Proposition \ref{prop:well-posed_s}.
\end{rmk}

The proof of Proposition \ref{prop:well-posed_12} follows along similar lines as the proof of Proposition \ref{prop:well-posed_s}. Due to the presence of the weights we however need to rely on suitable modifications of the boundary-bulk inequalities as recalled in Section \ref{sec:notation_spaces}.

\begin{proof}
\emph{Step 1: Reduction.}
As in the proof of Proposition \ref{prop:well-posed_12} we first reduce the setting to $f_2 = 0$ by considering $u=u_1 + E_s(f_2)$, where $E_s(f_2 )\in H^1_{\partial \Omega \setminus \Sigma_2,0}(\Omega, d^{1-2s})$ is obtained from Lemma \ref{lem:extension} and has the property that
$\|E_s(f_2)\|_{H^1_{\partial \Omega \setminus \Sigma_2,0}(\Omega, d^{1-2s})} \leq C \|f_2\|_{H^{s}(\Sigma_2)}$.

Working with the equation for $u_1$ yields an equation of the desired form with $u_1|_{\Sigma_2}= \tilde{f}_2 =0$ and a new inhomogeneity $\tilde{F}:=F-L_{\lambda,s}(E_s(f_2)) \in (H^1_{\partial \Omega \setminus \Sigma_1,0}(\Omega, d^{1-2s}))^{\ast}$. As in the case $s=\frac{1}{2}$, the functional $L_{\lambda,s}(E(f_2))$ is interpreted weakly, in that it is given by
\begin{align*}
H^1_{\partial \Omega \setminus \Sigma_1,0}(\Omega, d^{1-2s}) \ni v \mapsto B_{s,A,V,q}(E_s(f_2),v).
\end{align*}
With a slight abuse of notation, we drop the subscript in the notation for $u_1$ and the tildas in the notation for the data in the following.\\

\emph{Step 3: Continuity.}
The continuity of the bilinear form then is a consequence of the following observations and estimates:
\begin{itemize}
\item[(i)] \emph{Continuity of the boundary terms.} We observe that by Lemma \ref{lem:trace}
\begin{align*}
\left| \int\limits_{\partial \Omega} q u \overline{v} dx \right| 
& \leq \|q\|_{L^{\infty}(\partial \Omega)} \|u\|_{L^2(\partial \Omega)} \|v\|_{L^2(\partial \Omega)}\\
&\leq \|q\|_{L^{\infty}(\partial \Omega)} \|u\|_{H^s(\partial \Omega)} \|v\|_{H^s(\partial \Omega)}\\
&\leq  C\|q\|_{L^{\infty}(\partial \Omega)} \|u\|_{H^1( \Omega, d^{1-2s})} \|v\|_{H^1( \Omega, d^{1-2s})}.
\end{align*}
\item[(ii)] \emph{Continuity of the bulk terms.} As the continuity of all the bulk terms follows analogously, we only discuss the first magnetic term: In this case, by Cauchy-Schwarz, we obtain
\begin{align*}
\left| \int\limits_{\Omega} d^{1-2s} v A_1 \cdot \overline{\nabla u} dx \right| 
\leq \|A_1\|_{L^{\infty}(\Omega)} \|v\|_{L^{2}(\Omega, d^{1-2s})} \|u\|_{H^1(\Omega, d^{1-2s})}.
\end{align*}
\item[(iii)] \emph{Boundedness of the right hand side.} The mapping
\begin{align*}
H^{-s}(\Sigma_1) \ni f_1 \mapsto (f_1, v)_{L^2(\Sigma_1)}
\end{align*}
for $v\in H^{1}_{\partial \Omega \setminus \Sigma_1,0}(\Omega, d^{1-2s})$ satisfies the bound
\begin{align*}
\left| (f_1, v)_{L^2(\Sigma_1)} \right| 
\leq \|f_1\|_{H^{-s}(\Sigma_1)} \|v\|_{H^{s}(\partial \Omega)}
\leq \|f_1\|_{H^{-s}(\Sigma_1)} \|v\|_{H^1(\Omega, d^{1-2s})}.
\end{align*}
It is thus a bounded linear functional on $(H^{1}_{\partial \Omega \setminus \Sigma_1, 0}(\Omega, d^{1-2s}))^{\ast}$.
Similarly, for $F \in (H^{1}_{\partial \Omega \setminus \Sigma_1, 0}(\Omega, d^{1-2s}))^{\ast}$, by definition, the map $v\mapsto \langle F,v \rangle$ is a bounded linear functional on $H^{1}_{\partial \Omega \setminus \Sigma_1, 0}(\Omega, d^{1-2s})$. 
\end{itemize}

\emph{Step 4: Coercivity.}
For the coercivity estimate, we need to bound $B_{s,A,V,q}(u,u) $ from below. Again the bulk estimates follow from the Cauchy-Schwarz and Young's inequality. The main point is to consider the boundary contributions and to prove their coercivity. This is a consequence of the trace estimate from Lemma \ref{lem:boundary_trace_s}. Indeed, we deduce that for $\mu>1$ to be chosen below
\begin{align*}
\left| \int\limits_{\partial \Omega} q |u|^2 dx \right|
& \leq \|q\|_{L^{\infty}(\partial \Omega)} \|u\|_{L^2(\partial \Omega)}^2\\
& \leq C\|q\|_{L^{\infty}(\partial \Omega)} (\mu^{-2s} \|d^{\frac{1-2s}{2}} \nabla u\|_{L^2(\Omega)}^2 + \mu^{2-2s}\|d^{\frac{1-2s}{2}} u\|_{L^2(\Omega)}^2).
\end{align*}
Choosing $\mu>1$ such that 
\begin{align*}
C\|q\|_{L^{\infty}(\partial \Omega)} \mu^{-2s} \leq \frac{1}{10},
\end{align*}
we thus infer that 
\begin{align*}
\left| \int\limits_{\partial \Omega} q |u|^2 dx \right|
& \leq \frac{1}{10}\|u\|_{H^1(\Omega, d^{1-2s})}^2 + C \|q\|_{L^{\infty}(\Omega)}^{{\frac{1}{s}-1}} \|u\|_{L^2(\Omega, d^{1-2s})}^{{2}}.
\end{align*}
Next we note that there exists a constant $C>0$ such that for all $u\in H^{1}_{\partial \Omega \setminus \Sigma_1,0}(\Omega)$ it holds that
\begin{align*}
C \|\nabla u\|_{L^2(\Omega,d^{1-2s})} \geq \|u\|_{H^1(\Omega,d^{1-2s})}.
\end{align*}
This follows from the fact that $u|_{\partial \Omega \setminus \Sigma_1} =0$ and an application of Poincar\'e's inequality, see for instance \cite[Lemma 4.7]{RS17} and the proof of Lemma \ref{lem:boundary_trace_s}.
Combining all these observations and also invoking the estimate in Step 3(ii) for the bulk contributions (to which we still apply Young's inequality), we hence infer that the bilinear form $B_{s,A,V,q}(\cdot, \cdot)$ is coercive, i.e. that
\begin{align*}
B_{s,A,V,q}(u,u) \geq \frac{1}{2}\|u\|_{H^{1}(\Omega,d^{1-2s})}^{2} - C_{low} \|u\|_{L^2(\Omega, d^{1-2s})}^{2},
\end{align*}
where the constant $C_{low}>0$ depends on $s, \|A\|_{L^{\infty}(\Omega)},  \|V\|_{L^{\infty}(\Omega)}, \|q\|_{L^{\infty}(\partial \Omega)},n$. 

\medskip

\emph{Step 5: Conclusion.} With the available upper and lower bounds, we conclude as in the proof of Proposition \ref{prop:well-posed_12}. More precisely, for $\mu > C_{low}$ the bilinear form $B_{s,\mu}(u,v):= B_{s,A,V,q}(u,v) + \mu (u,v)_{L^2(\Omega, d^{1-2s})}$ is a scalar product. Hence, in combination with the third estimate in Step 3, the Riesz representation theorem is applicable and yields a unique solution $u = G(\bar{F}) \in H^{1}_{\partial \Omega \setminus \Sigma_1,0}(\Omega, d^{1-2s})$ with $\bar{F}=(F,f_1)$ to the equation
\begin{align*}
B_{s,\mu}(u,v) = (F, v)_{L^2(\Omega)} + (f_1,v)_{L^2(\Sigma_1)} \mbox{ for all } v\in H^{1}_{\partial \Omega \setminus \Sigma_1,0}(\Omega, d^{1-2s}).
\end{align*}
Using the compactness of the space $L^2(\Omega, d^{1-2s}) \subset H^1(\Omega, d^{1-2s})$ and the fact that 
\begin{align*}
G_{\mu}: L^2(\Omega, d^{2s-1}) {\times L^2(\Sigma_1)} \subset (H^1_{\partial \Omega \setminus \Sigma_1,0}(\Omega, d^{1-2s}))^{\ast} {\times H^{-s}(\Sigma_1)}\\
 \rightarrow H^1_{\partial \Omega \setminus \Sigma_1,0}(\Omega, d^{1-2s}) \subset L^2(\Omega, d^{1-2s}),
\end{align*}
the claim on the set $M \subset \R$ follows from the spectral theorem for self-adjoint, compact operators in the same way as in Proposition \ref{prop:well-posed_12} (see for instance \cite[Theorem 2.37 and Corollary 2.39]{McLean}).
\end{proof}

As in the case $s= \frac{1}{2}$ the well-posedness result allows us to define the Poisson operator associated with the Schrödinger equation \eqref{eq:frac_Schr}.

\begin{defi}
\label{defi:Poissons}
Let $M\subset \C$ be as in Proposition \ref{prop:well-posed_s} and assume that $0 \notin M$.
Let $f \in H^{s}_{\overline{\Sigma_2}}$ and let $u \in H^1(\Omega,d^{1-2s})$ be the solution constructed in Proposition \ref{prop:well-posed_12} with $F=0, f_1 = 0$ and $f_2 = f$. Then, we define the \emph{Poisson operator} 
\begin{align*}
P_s: H^{s}_{\overline{\Sigma_2}} \rightarrow H^1(\Omega, d^{1-2s}), \ f \mapsto u.
\end{align*}
\end{defi}

Again this operator is bounded by the apriori estimates from the well-posedness result in Proposition \ref{prop:well-posed_s}.

In order to simplify our discussion, for convenience we will, for the remainder of the article, always make the following assumption:

\begin{assume}
\label{as:EV}
For the remainder of the article we will assume that zero is not an eigenvalue of the Schrödinger operators \eqref{eq:Schroedinger} and \eqref{eq:frac_Schr}, i.e. we will assume that
 $\lambda=0 \notin M$, where $M$ denotes the sets constructed in Propositions \ref{prop:well-posed_12} and  \ref{prop:well-posed_s}, respectively.
 \end{assume}

We remark that as a consequence of Proposition \ref{prop:well-posed_s}, we also obtain the following regularity result for the weighted normal derivative:

\begin{lem}
\label{lem:normal_deriv}
Let $u$ be a weak solution to \eqref{eq:frac_Schr} possibly also with a bulk inhomogeneity $F \in L^2(\Omega,d^{2s-1})$.
Then, there exists a constant $C>0$ such that for each $\delta>0$ sufficiently small we have that $ d^{1-2s} \p_{\nu} u \in H^{-s}(\partial \Omega_{\delta})$ with
\begin{align}
\label{eq:normalHs}
\|d^{1-2s} \p_{\nu} u\|_{H^{-s}(\partial \Omega_{\delta})}
\leq C \left( \|F\|_{L^2(\Omega, d^{2s-1})} + \| u\|_{H^{s}(\partial \Omega)} \right),
\end{align}
where $\partial \Omega_{\delta}:= \{x + t\nu(x): \ x \in \partial \Omega,\ t \in (0,\delta)\}$ and where $\nu(x)$ denotes the inward pointing unit normal at a point $x\in \partial \Omega$.
Moreover,
\begin{align}
\label{eq:normalHsconv}
\|\lim\limits_{d \rightarrow 0} d^{1-2s} \p_{\nu} u - d^{1-2s} \p_{\nu} u\|_{H^{-s}(\partial \Omega_{\delta})} \rightarrow 0 \mbox{ as } \delta \rightarrow 0.
\end{align}
\end{lem}

\begin{proof}
The fact that for any $\delta \in (0,1)$ sufficiently small $d^{1-2s} \p_{\nu} u|_{\partial \Omega_{\delta}} \in H^{-s}(\partial \Omega_{\delta})$ with a uniform estimate (in $\delta>0$) follows by duality and the weak form of the equation: Indeed, due to the validity of the equation \eqref{eq:frac_Schr} and the fact that this is a uniformly elliptic equation away from $\partial \Omega$, we have that $d^{1-2s}\p_{\nu} u \in L^2(\partial \Omega_{\delta})$ for any $\delta>0$. Integrating by parts, we further observe that for any $w\in H^s(\partial \Omega_{\delta})$ and an associated extension $E_s(w)\in H^1(\Omega_{\delta},d^{1-2s})$, we obtain
\begin{align}
\label{eq:normal_id}
(w, d^{1-2s} \p_{\nu} u)_{L^2(\partial \Omega_{\delta})}
= B_{s,A,V,q}(E_s(w),u).
\end{align}
Using \eqref{eq:normal_id} and the boundary estimate from Lemma \ref{lem:extension}, we hence estimate
\begin{align*}
|(w, d^{1-2s}\p_{\nu} u)_{L^2(\partial \Omega_{\delta})}|
&\leq |(E_s (w),F)_{L^2(\Omega_{\delta})}| + |(\nabla E_s(w), d^{1-2s} \nabla u)_{L^2(\Omega_{\delta})}|\\
&\leq C \|w\|_{H^s(\partial \Omega_{\delta})} (\| u\|_{H^1(\Omega, d^{1-2s})} + \|F\|_{L^2(\Omega,d^{2s-1})} )\\
& \leq C\|w\|_{H^s(\partial \Omega_{\delta})}(\|u\|_{H^s(\partial \Omega)}+ \|F\|_{L^2(\Omega,d^{2s-1})}).
\end{align*}
Thus, taking the supremum in $w\in H^s(\partial \Omega_{\delta})$ with $\|w\|_{H^{s}(\partial \Omega_{\delta})}=1$ implies the claim \eqref{eq:normalHs}.

Moreover, by the definition of the normal derivative by means of the bilinear form as in \eqref{eq:normal_id} for $\delta_1, \delta_2 >0$ small,
\begin{align*}
&\|(d^{1-2s} \p_{\nu} u)(\cdot + \delta_1 \nu) - (d^{1-2s} \p_{\nu} u)(\cdot + \delta_2 \nu)\|_{H^{-s}(\partial \Omega)}\\
&= \sup\limits_{\|w\|_{H^s(\partial \Omega)}\leq 1} (w, (d^{1-2s} \p_{\nu} u)(\cdot + \delta_1 \nu) - (d^{1-2s} \p_{\nu} u)(\cdot + \delta_1 \nu))_{L^2(\partial \Omega)}\\
&= \sup\limits_{\|w\|_{H^s(\partial \Omega)}\leq 1}\left( (E_s(w)\chi_{\Omega_{\delta_1}}-E_s(w)\chi_{\Omega_{\delta_2}},F)_{L^2(\Omega)} \right.\\
& \quad \left.+ (d^{1-2s}(\chi_{\Omega_{\delta_1}}\nabla E_s(w)- \chi_{\Omega_{\delta_2}}\nabla E_s(w)), \nabla u)_{L^2(\Omega)}\right)\\
& \leq C \|w\|_{H^s(\partial \Omega)}\left( \|(\chi_{\Omega_{\delta_1}}-\chi_{\Omega_{\delta_2}}) F\|_{L^2(\Omega, d^{2s-1})} + \|d^{1-2s}(\chi_{\Omega_{\delta_1}}-\chi_{\Omega_{\delta_2}}) \nabla u\|_{L^2(\Omega)} \right)
\rightarrow 0 \\
&\qquad \mbox{ as } \delta_1, \delta_2 \rightarrow 0.
\end{align*}
Here we used that $\nabla u \in L^2(\Omega, d^{1-2s})$ by the apriori estimates from the well-posedness results and have set $\Omega_{\delta}:= \{x \in \Omega: \ \dist(x, \partial \Omega)\geq \delta\}$ for $\delta>0$ sufficiently small. 
This proves that $\{(d^{1-2s} \p_{\nu} u)(\cdot + n^{-1} \nu)\}_{n \in \N}$ is a Cauchy sequence in $H^{-s}(\partial \Omega)$, that $\lim\limits_{n \rightarrow \infty}(d^{1-2s} \p_{\nu} u)(\cdot + n^{-1} \nu)$ exists in $H^{-s}(\partial \Omega)$ as $n \rightarrow \infty$ and that \eqref{eq:normalHsconv} holds. 
\end{proof}

With the well-posedness results of Propositions \ref{prop:well-posed_12}, \ref{prop:well-posed_s} and the global Assumption \ref{as:EV} in hand, we can now also define the (partial data) Dirichlet-to-Neumann maps which we will study in the sequel.

\begin{defi}[Partial Dirichlet-to-Neumann maps]
\label{defi:DtN}
Let $s\in (0,1)$ and let $B_{A,V,q}(\cdot, \cdot)$ and $B_{s,A,V,q}(\cdot, \cdot)$ denote the bilinear forms from \eqref{eq:Bform12}, \eqref{eq:Bforms}. We then define the (partial) Dirichlet-to-Neumann maps $\Lambda_{A,V,q}:\widetilde{H}^{\frac{1}{2}}(\Sigma_2) \rightarrow H^{- \frac{1}{2}}(\Sigma_2)$ and $\Lambda_{s,A,V,q}:\widetilde{H}^{s}(\Sigma_2) \rightarrow H^{-s}(\Sigma_2)$ weakly as
\begin{align*}
\langle \Lambda_{A,V,q} f, g \rangle_{\ast} &:= B_{A,V,q}(u_f, E(g)), \\
\langle \Lambda_{s,A,V,q} f, g \rangle_{\ast_s} &:= B_{s, A,V,q}(u_f, E_s(g)), 
\end{align*}
where $E(g)$ denotes an $H^1_{\partial \Omega \setminus (\Sigma_1 \cup \Sigma_2),0}(\Omega)$ extension of $g$ into $\Omega$ and $u_f$ denotes a weak solution (in the sense of Proposition \ref{prop:well-posed_12} of \eqref{eq:Schroedinger}). Similarly, $E_s(g)$ is an $H^1_{\partial \Omega \setminus (\Sigma_1 \cup \Sigma_2),0}(\Omega, d^{1-2s})$ extension of $g$ into $\Omega$ and $u_f$ denotes a weak solution (in the sense of Proposition \ref{prop:well-posed_s}) of \eqref{eq:frac_Schr}. Here the notation $\langle \cdot, \cdot \rangle_{\ast}$ and $\langle \cdot, \cdot \rangle_{\ast_s}$ denotes the duality pairing between $H^{-\frac{1}{2}}(\Sigma_2)$ and $H^{\frac{1}{2}}(\Sigma_2)$ and between $H^{-s}(\Sigma_2)$ and $H^{s}(\Sigma_2)$, respectively.
\end{defi}

\begin{rmk}
By definition we of course have that $\Lambda_{A,V,q}= \Lambda_{\frac{1}{2},A,V,q}$.
\end{rmk}

As in the standard (partial data) setting, these Dirichlet-to-Neumann maps are well-defined and do not depend on the choice of the extension.

\begin{lem}
\label{lem:welldef}
Let $\Lambda_{A,V,q}:\widetilde{H}^{\frac{1}{2}}(\Sigma_2) \rightarrow H^{- \frac{1}{2}}(\Sigma_2)$ and $\Lambda_{s,A,V,q}: \widetilde{H}^{s}(\Sigma_2) \rightarrow H^{-s}(\Sigma_2)$ be as in Definition \ref{defi:DtN}. Then these maps are well-defined, i.e. they do not depend on the choice of the extension $E(g)$ and $E_s(g)$. Moreover, both maps are linear and bounded.
\end{lem}

\begin{proof}
The independence of the choice of the extension follows from the well-posedness theory. Indeed, considering two extensions $E(g)$ and $\tilde{E}(g)$ of $g\in \widetilde{H}^{1/2}(\Sigma_2)$, we deduce that $E(g)-\tilde{E}(g) \in H^{1}_{\partial \Omega \setminus \Sigma_1,0}(\Omega)$. Hence,
we obtain that $B_{A,V,q}(u_f, E(g)-\tilde{E}(g))=0$ since $u_f$ is a weak solution to the equation \eqref{eq:Schroedinger}. A similar argument holds for the weighted operator. The linearity of the map follows from the linearity of the Schrödinger equations \eqref{eq:Schroedinger} and \eqref{eq:frac_Schr}. The boundedness follows from the apriori estimates \eqref{eq:apriori1/2} and \eqref{eq:aprioris}.
\end{proof}

As in the classical setting, the (partial data) Dirichlet-to-Neumann maps are self-adjoint operators:

\begin{lem}[Symmetry]
\label{lem:symm}
Let $\Lambda_{A,V,q}$ and $\Lambda_{s,A,V,q}$ be as in \eqref{defi:DtN}. Then, we have
\begin{align*}
\langle \Lambda_{A,V,q} f, g\rangle_{\ast} &= \langle  f, \Lambda_{A,V,q} g\rangle_{\ast},\\
\langle \Lambda_{s,A,V,q} f,g\rangle_{\ast_s} &= \langle  g, \Lambda_{s,A,V,q} g\rangle_{\ast_s}.
\end{align*}
\end{lem}

\begin{proof}
The claim follows from the fact that two solutions $u_f$ and $u_g$ associated with the data $f,g$ in \eqref{eq:Schroedinger} or \eqref{eq:frac_Schr} are particular extensions of $f,g \in \widetilde{H}^s(\Sigma_2)$. Since the bilinear forms $B_{A,V,q}(\cdot, \cdot)$ and $B_{s,A,V,q}(\cdot, \cdot)$ are symmetric (with respect to the complex scalar product) the claim follows.
\end{proof}

Furthermore, a central Alessandrini identity involving all potentials holds true:

\begin{lem}[Alessandrini]
\label{lem:Aless}
Let $A_j,V_j,q_j$ and $\Lambda_{s,A_j,V_j,q_j}$ with $j \in \{1,2\}$ be as above. Then, for two solutions $u_1, u_2$ of \eqref{eq:frac_Schr} associated with the respective boundary data and potentials,
\begin{align*}
\langle (\Lambda_{s,A_1,V_1,q_1} - \Lambda_{s,A_2,V_2, q_2})f_1, f_2 \rangle_{\ast_s}
&= \int\limits_{\Omega}(V_1 - V_2 + A_1^2 - A_2^2) u_1 \overline{u_2} d^{1-2s} dx \\
& \quad + i \int\limits_{\Omega} d^{1-2s} (A_1 - A_2) \cdot (u_1 \overline{\nabla u_2} - u_2 \cdot \overline{\nabla  u_1})   dx \\
& \quad + \int\limits_{\Sigma_1}(q_1- q_2) u_1 \overline{u_2} d \mathcal{H}^{n-1}.
\end{align*}
\end{lem}

\begin{proof}
This follows by using the symmetry result of Lemma \ref{lem:symm} in combination with the structure of $B_{s,A,V,q}$ and the fact that all (bulk, boundary and gradient) potentials are real valued:
\begin{align*}
\langle (\Lambda_{s,A_1,V_1,q_1} - \Lambda_{s,A_2,V_2, q_2})f_1, f_2 \rangle_{\ast_s}
&= \langle \Lambda_{s,A_1,V_1,q_1} f_1, f_2 \rangle_{\ast_s}
- \langle f_1 , \Lambda_{s,A_2, V_2,q_2} f_2 \rangle_{\ast_s}\\
&= B_{s,A_1, V_1, q_1} (u_1, u_2) - B_{s,A_2, V_2,q_2}(u_1, u_2)\\
& =  \int\limits_{\Omega}(V_1 - V_2 + A_1^2 - A_2^2) u_1 \overline{u_2} d^{1-2s} dx
+ \int\limits_{\Sigma_1} (q_1 - q_2) u_1 \overline{u}_2 d \mathcal{H}^{n-1}\\
& \quad + i \int\limits_{\Omega} d^{1-2s} (A_1 - A_2) \cdot (u_1 \overline{\nabla u_2} - u_2 \cdot \overline{\nabla  u_1})   dx .
\end{align*}
This proves the claim.
\end{proof}

\section{Simultaneous Runge Approximation in the Bulk and on the Boundary -- Resolution of the Question (Q1) for $s= \frac{1}{2}$}
\label{sec:Q11}

In this section we discuss the resolution of the question (Q1) for the case $s= \frac{1}{2}$ by proving simultaneous Runge approximation results. This requires a certain ``safety distance'' between $\Omega_1$ and the sets $\Sigma_1, \Sigma_2$ and a topological condition on the connectedness of $\Omega \setminus \Omega_1$. We refer to the set-up which had been layed out in Section \ref{sec:not_sets} for the precise conditions. Although our setting could have been generalized to allow for $\Omega_1$ including some boundary portions (see for instance \cite{RS20}), for clarity of exposition, we do not address this in the present article.

Let 
\begin{align}
\label{eq:solutions_spaces}
\begin{split}
S_{A,V,q}&:= \{u \in L^2(\Omega): u \mbox{ is a weak solution to \eqref{eq:Schroedinger} in } \Omega \},\\
\tilde{S}_{A,V,q}&:= \{u \in H^1(\Omega_1): u \mbox{ is a weak solution to \eqref{eq:Schroedinger} in } \Omega_1 \} \subset L^2(\Omega_1).
\end{split}
\end{align}

Here by a weak solution we mean a solution as obtained in our well-posedness discussion in Section \ref{sec:well_posed1}. For simplicity, we also simply set $S_{V,q}:= S_{0,V,q}$ and $\tilde{S}_{V,q}:= \tilde{S}_{0,V,q}$.
 
As a first step towards answering the question (Q1), we prove the simultaneous Runge approximation result (in the absence of magnetic potentials) from Lemma \ref{lem:sim_dense}.

\begin{rmk}
Together with the (known) existence results of whole space CGO solutions, this approximation result allows us to recover the potentials $V\in L^{\infty}(\Omega_1)$ and $q \in L^{\infty}(\partial \Omega)$ simultaneously in the inverse problem for \eqref{eq:Schroedinger}. Instead of explaining this at this point already, we refer to the proof of Theorem \ref{thm:Q11/2}, where this is deduced even in the presence of magnetic potentials.
\end{rmk}

\begin{proof}[Proof of Lemma \ref{lem:sim_dense}]
By the Hahn-Banach theorem, it suffices to prove that if $v=(v_1, v_2) \in L^2(\Sigma_1)\times L^2(\Omega_1)$ satisfies $v \perp \mathcal{R}_{bb}$ (with respect to the scalar product in $L^2(\Sigma_1) \times L^2(\Omega_1)$), then we have
\begin{align*}
v \perp (L^2(\Sigma_1) \times \tilde{S}_{V,q}).
\end{align*}

To this end, let $f \in C^\infty_c(\Sigma_2)$ and define $u:= P f$. Moreover, let $w$ be a solution to the associated adjoint problem
\begin{align}
\label{eq:12dual_prob}
\begin{split}
-\D w + Vw & = v_2 \chi_{\Omega_1} \mbox{ in } \Omega,\\
\p_\nu w +qw & = v_1 \mbox{ on } \Sigma_1,\\
w & = 0 \mbox{ on } \partial \Omega \setminus \Sigma_1.
\end{split}
\end{align}
Here $\chi_{\Omega_1}$ denotes the characteristic function of the set $\Omega_1$.
\noindent By the assumption $v \perp \mathcal{R}_{bb}$ and the definitions of $u$ and $w$, we have
\begin{align}
\label{eq:slightly_formal}
\begin{split}
0 & = (v_1,u|_{\Sigma_1})_{L^2(\Sigma_1)}+ (v_2,u|_{\Omega_1})_{L^2(\Omega_1)} \\ 
   & = \langle \p_\nu w +qw,u-f \rangle_{\ast} + (-\D w + Vw,u)_{L^2(\Omega)} \\
   & = \langle \p_\nu w +qw,u-f\rangle_{\ast}+ (-\D u +Vu,w)_{L^2(\Omega)} + \langle \p_\nu u,w\rangle_{\ast} - \langle  u,\p_\nu w\rangle_{\ast} \\ 
   & = \langle  qw,u\rangle_{\ast} -  \langle  \p_\nu w +qw,f\rangle_{\ast}+ \langle  \p_\nu u,w\rangle_{\ast} \\ 
   & = -  \langle \p_\nu w +qw,f\rangle_{\ast} \;,
\end{split}
\end{align}

\noindent where we integrated by parts twice and where $\langle  \cdot, \cdot \rangle_{\ast}$ denotes the $H^{-\frac{1}{2}}(\Sigma_1)$, $H^{\frac{1}{2}}(\Sigma_1)$ duality pairing. We remark that this computation which -- a priori is formal, since due to the mixed boundary conditions $w,u$ may not be in $H^2(\Omega)$ -- can be justified by considering the identities in a smaller domain $\Omega_{\epsilon}$ for $\epsilon>0$ sufficiently small first and then passing to the limit $\epsilon \rightarrow 0$. More precisely, by standard regularity theory, we obtain that $w, u \in H^2(\Omega_{\epsilon})$ which allows us to justify the following manipulations:
\begin{align*}
\label{eq:int_by_parts_dense}
\begin{split}
& (\p_\nu w +qw,u- u|_{\partial \Omega_{\epsilon}\setminus \Sigma_{1,\epsilon}})_{L^2(\p\Omega_{\epsilon})}+ (-\D w + Vw,u)_{L^2(\Omega_{\epsilon})} \\
   & = (\p_\nu w +qw,u- u|_{\partial \Omega_{\epsilon}\setminus \Sigma_{1,\epsilon}})_{L^2(\p\Omega_{\epsilon})}+ (-\D u +Vu,w)_{L^2(\Omega_{\epsilon})} + (\p_\nu u,w)_{L^2(\partial\Omega_{\epsilon})} - ( u,\p_\nu w)_{L^2(\partial\Omega_{\epsilon})} \\ 
   & = (\p_\nu w +qw,u- u|_{\partial \Omega_{\epsilon}\setminus \Sigma_{1,\epsilon}})_{L^2(\p\Omega_{\epsilon})} + (\p_\nu u,w)_{L^2(\partial\Omega_{\epsilon})} - ( u,\p_\nu w)_{L^2(\partial\Omega_{\epsilon})} .
\end{split}
\end{align*}
Here $\partial \Omega_{\epsilon}$ is defined as in \eqref{eq:boundary_not} and $\Sigma_{1,\epsilon}:=\{x\in \Omega: \ x = y + \epsilon \nu(y), \ y \in \Sigma_1\}$.
Then passing to the limit $\epsilon \rightarrow 0$ and using the observations from Lemma \ref{lem:normal_deriv} allows us to recover the first and fourth lines in \eqref{eq:slightly_formal}, i.e. 
\begin{align*}
&(\p_\nu w +qw,u- u|_{\partial \Omega_{\epsilon}\setminus \Sigma_{1,\epsilon}})_{L^2(\p\Omega_{\epsilon})}+ (-\D w + Vw,u)_{L^2(\Omega_{\epsilon})} \\
& \rightarrow \langle \p_\nu w +qw,u- f\rangle_{\ast}+ ( v_2,u|_{\Omega_1} )_{L^2(\Omega_1)},\\
& (\p_\nu w +qw,u- u|_{\partial \Omega_{\epsilon}\setminus \Sigma_{1,\epsilon}})_{L^2(\p\Omega_{\epsilon})} + (\p_\nu u,w)_{L^2(\partial\Omega_{\epsilon})} - ( u,\p_\nu w)_{L^2(\partial\Omega_{\epsilon})}\\
&\rightarrow  \langle qw,u \rangle_{\ast} -  \langle\p_\nu w +qw,f\rangle_{\ast}+ \langle\p_\nu u,w\rangle_{\ast} .
\end{align*}
This then allows us to conclude the identity $\langle\p_{\nu} w + q w, f \rangle_{\ast} =0$ as in the formal argument from \eqref{eq:slightly_formal}.

By the arbitrary choice of $f\in C^\infty_c(\Sigma_2)$, \eqref{eq:slightly_formal} yields that $\p_\nu w + qw = 0$ in $\Sigma_2$, which in turn by the defining property of $w$ gives $\p_\nu w|_{\Sigma_2}=w|_{\Sigma_2}=0$. Thus now the unique continuation property (see for instance \cite{ARV09}) implies that $w=0$ in $\Omega\setminus \overline{\Omega_1}$, and therefore 
\begin{equation}\label{eq:zero_on_boundary_nomag}
 w|_{\Sigma_1}=\nabla w|_{\Sigma_1}=0 \quad\mbox{and}\quad w|_{\p\Omega_1}=\nabla w|_{\p\Omega_1}=0\;.
\end{equation}
 
\noindent The first part of \eqref{eq:zero_on_boundary_nomag} implies $v_1=0$ by the definition of the associated dual problem \eqref{eq:12dual_prob}. In particular, $(v_1,\psi_1)_{L^2(\Sigma_1)}=0$ for all $\psi_1 \in L^2(\Sigma_1)$. If now $\psi_2 \in \tilde{S}_{V,q}$, denoting the $H^{-\frac{1}{2}}(\partial \Omega_1)$, $H^{\frac{1}{2}}(\partial \Omega_1)$ duality pairing by $\langle \cdot, \cdot \rangle_{\ast, \partial \Omega_1}$ and integrating by parts we get
\begin{align*}
(v_2,\psi_2)_{L^2(\Omega_1)} & = (-\D w + Vw, \psi_2)_{L^2(\Omega_1)} \\ 
    & =  (-\D \psi_2 + V\psi_2,w)_{L^2(\Omega_1)} + \langle \p_\nu \psi_2,w \rangle_{\ast, \partial \Omega_1} - \langle\psi_2,\p_\nu w\rangle_{\ast, \partial \Omega_1} \;,
\end{align*} 

\noindent which vanishes because of the second part of formula \eqref{eq:zero_on_boundary_nomag} and because $\psi_2\in \tilde{S}_{V,q}$ (which is also true in the weak form of the equation by definition). Hence, 
$$ (v,\psi)_{L^2(\Sigma_1)\times L^2(\Omega_1)} =0 \quad\mbox{for all} \quad \psi = (\psi_1,\psi_2) \in (L^2(\Sigma_1) \times \tilde{S}_{V,q})\;,$$
that is, $v \perp (L^2(\Sigma_1) \times \tilde{S}_{V,q})$ with respect to the $L^2(\Sigma_1)\times L^2(\Omega_1)$ scalar product as desired. 
\end{proof}

For our next step towards the solution of question (Q1), we shall consider a generalization of equation \eqref{eq:Schroedinger}, namely 
\begin{align}
\label{eq:Schroedinger_2}
\begin{split}
Lu := -\nabla\cdot (g \nabla  u) - i A_1 \cdot \nabla u - i \nabla \cdot (A_2 u) + V u & = 0  \mbox{ in } \Omega,\\
\nu\cdot(g \nabla u) + qu &= 0 \mbox{ on } \Sigma_1,\\
u & = f \mbox{ on } \Sigma_2,\\
u & = 0 \mbox{ on } \partial \Omega \setminus (\Sigma_1 \cup \Sigma_2),
\end{split}
\end{align}
\noindent where $g = (g_{ij})_{i,j=1,\dots,n}$ is a $C^2$  metric, i.e. a symmetric, positive definite, elliptic, $C^2$-regular matrix valued function on $\Omega$, and the magnetic potentials $A_1$ and $A_2$ do not necessarily coincide. 
We avoid discussing the well-posedness for this and refer to \cite{McLean} and \cite{GT} for a discussion of it. In the sequel, we will assume the well-posedness of this problem and its associated dual problem.

In connection to the problem \eqref{eq:Schroedinger_2} we define the sets
\begin{align*}
S_{g,A_1,A_2,V,q}&:= \{u \in L^2(\Omega): u \mbox{ is a weak solution to \eqref{eq:Schroedinger_2} in } \Omega \},\\
\tilde{S}_{g,A_1,A_2,V,q}&:= \{u \in H^1(\Omega_1): u \mbox{ is a weak solution to \eqref{eq:Schroedinger_2} in } \Omega_1 \} \subset L^2(\Omega_1).
\end{align*}

The next Lemma \ref{Lem:Runge_with_magnetic_potentials} shows that the result of Lemma \ref{lem:sim_dense} still holds for equation \eqref{eq:Schroedinger_2}, and the approximation can even be given in $H^1(\Omega_1)$ instead of $L^2(\Omega_1)$.

\begin{lem}\label{Lem:Runge_with_magnetic_potentials}
Assume that the set-up is as above. Then, the set
\begin{align*}
\mathcal{R}_{bb}:=\{(u|_{\Sigma_1}, u|_{\Omega_1})  : \ u|_{\Sigma_1} = P f|_{\Sigma_1} \mbox{ and } u|_{\Omega_1} = P f|_{\Omega_1} \mbox{ with } f\in C_c^{\infty}(\Sigma_2)\} \subset L^2(\Sigma_1) \times H^1(\Omega_1)
\end{align*}
is dense in $L^2(\Sigma_1) \times \tilde{S}_{g,A_1,A_2,V,q}$ equipped with the $L^2(\Sigma_1)\times H^1(\Omega_1)$ topology. Here $P$ denotes the Poisson operator which is defined in analogy to Definition \ref{defi:Poisson} and in particular maps data $f\in C_c^{\infty}(\Sigma_2)$ into the associated (weak) solution $u$ to the equation \eqref{eq:Schroedinger_2}.
\end{lem}

\begin{proof}
We use the same strategy as in the proof of the previous Lemma. Let $(v_1, v_2^*) \in L^2(\Sigma_1)\times (H^1(\Omega_1))^*$, and consider the unique Riesz representative $v_2 \in H^1(\Omega_1)$ of the functional $v_2^*\in (H^1(\Omega_1))^*$. By the Hahn-Banach theorem, it suffices to prove that if $v=(v_1, v_2) \in L^2(\Sigma_1)\times H^1(\Omega_1)$ satisfies $v \perp \mathcal{R}_{bb}$ with respect to the scalar product in $L^2(\Sigma_1) \times H^1(\Omega_1)$, then we have
\begin{align*}
v \perp (L^2(\Sigma_1) \times \tilde{S}_{g,A_1,A_2,V,q}).
\end{align*}

To this end, let $f \in C^\infty_c(\Sigma_2)$ and define $u:= P f$. Moreover, let $w$ be a solution to the associated adjoint problem
\begin{align}
\label{eq:dualH1_approx}
\begin{split}
L^* w & = \tilde{v}_2^* \mbox{ in } \Omega,\\
\nu\cdot(g \nabla w) + (q-i\nu\cdot A_1 -i \nu\cdot A_2)w & = v_1 \mbox{ on } \Sigma_1,\\
w & = 0 \mbox{ on } \partial \Omega \setminus \Sigma_1,
\end{split}
\end{align}
\noindent where $L^* := -\nabla\cdot (g \nabla ) + i A_2 \cdot \nabla + i \nabla \cdot A_1 + V$ and $\tilde{v}_2^*(\cdot) := v_2^*(\cdot|_{\Omega_1})$. First we observe that $\tilde{v}_2^* \in (H^1_b(\Omega))^{\ast} := (H^1_{\partial \Omega \setminus \Sigma_1,0}(\Omega) + \tilde{H}^{\frac{1}{2}}(\Sigma_2\cup \Sigma_1))^*$.
\noindent The associated bound is easily proved, since for $u \in H^1_b(\Omega)$
$$ |\tilde{v}_2^*(u)| = |{v_2}^*(u|_{\Omega_1})| \leq \|{v_2}^*\|\,\|u|_{\Omega_1}\|_{H^1(\Omega_1)} \leq \|v_2^*\|\,\|u\|_{H^1(\Omega)}\;.$$ 
\noindent  Now we have
$$ (v_2,u|_{\Omega_1})_{H^1(\Omega_1)} =  v_2^*(u|_{\Omega_1}) =  \tilde{v}_2^*(u)\;, $$

\noindent which by the assumption $v \perp \mathcal{R}_{bb}$ leads to
\begin{align}
\label{eq:runge_mag_step1}
\begin{split}
0 & = (v_1,u|_{\Sigma_1})_{L^2(\Sigma_1)}+ (v_2,u|_{\Omega_1})_{H^1(\Omega_1)} \\ 
   & = \langle \nu\cdot(g \nabla w) + (q-i\nu\cdot A_1 -i \nu\cdot A_2)w,u|_{\Sigma_1}\rangle_{\ast} + \tilde{v}_2^*(u) \\
   & = \langle \nu\cdot(g \nabla w) + (q-i\nu\cdot A_1 -i \nu\cdot A_2)w,u-f \rangle_{\ast}+ (L^*w,u)_{L^2(\Omega)}\;.  
\end{split}
\end{align}
As in the previous proof, $\langle \cdot, \cdot \rangle_{\ast}$ denotes the $H^{\frac{1}{2}}(\partial \Omega)$, $H^{-\frac{1}{2}}(\partial \Omega)$ duality pairing.

Integrating by parts twice (which can be justified in the same way as in the previous section), we obtain the following formula linking the operators $L$ and $L^*$:
\begin{align}\label{eq:int_by_parts}
\begin{split}
(Lu,w)_{\Omega}-(u,L^*w)_{\Omega} & 
= -(\nabla\cdot(g\nabla u),w)_{\Omega} -i ( A_1\cdot \nabla u,w)_{\Omega} -i ( \nabla\cdot(A_2 u),w)_{\Omega} \\ 
&\quad + (\nabla\cdot(g\nabla w),u)_{\Omega} -i ( A_2\cdot \nabla w,u)_{\Omega} -i (\nabla\cdot(A_1 w),u)_{\Omega} \\ 
& = -\langle \nu\cdot(g\nabla u),w \rangle_{\ast} -i \langle \nu\cdot(A_1 + A_2)u,w \rangle_{\ast} + \langle u,\nu\cdot(g\nabla w)\rangle_{\ast}\;.
\end{split}
\end{align}

\noindent Here we have used $(\cdot,\cdot)_\Omega$ as a shorthand notation for $(\cdot,\cdot)_{L^2(\Omega)}$. Combining formulas \eqref{eq:runge_mag_step1} and \eqref{eq:int_by_parts}, we infer
$$ \langle \nu\cdot(g\nabla w) + (q-i\nu\cdot A_1 -i \nu\cdot A_2)w,f \rangle_{\ast} =0\;, $$
\noindent which by the arbitrary choice of $f\in C^\infty_c(\Sigma_2)$ gives
$$ \nu\cdot(g\nabla w) + (q-i\nu\cdot A_1 -i \nu\cdot A_2)w = 0 \quad \mbox{on} \;\;\Sigma_2\;. $$

Thus, by definition of the adjoint equation \eqref{eq:dualH1_approx}, we are left with
\begin{align*}
L^*w & =0 \mbox{ in } \Omega\setminus \overline{\Omega_1},\\
\nu\cdot(g\nabla w) & =0\mbox{ on } \Sigma_2,\;\;\;\;\;\;\;\;\;\\
w&=0 \mbox{ on } \Sigma_2,\;
\end{align*}

\noindent and now the UCP (see for instance \cite{ARV09}) leads to $w=0$ in $\Omega\setminus \overline{\Omega_1}$. As a consequence of this fact, we obtain $w|_{\p(\Omega\setminus \overline{\Omega_1})}=0$ and $\nabla w|_{\p(\Omega\setminus \overline{\Omega_1})}=0$, which in particular implies
\begin{equation}\label{eq:zero_on_boundary}
 w|_{\Sigma_1}=\nabla w|_{\Sigma_1}=0 \quad\mbox{and}\quad w|_{\p\Omega_1}=\nabla w|_{\p\Omega_1}=0\;.
\end{equation}
 
\noindent The first part of \eqref{eq:zero_on_boundary} implies $v_1=0$ by the associated dual problem \eqref{eq:dualH1_approx}. In particular, $\langle v_1,\psi_1 \rangle_{\ast} =0$ for all $\psi_1 \in L^2(\Sigma_1)$.
\vspace{2mm}

Let now $\psi_2 \in \tilde{S}_{g,A_1,A_2,V,q}$. If $E: H^1(\Omega_1)\rightarrow H^1(\Omega)$ is any extension operator, we have
\begin{align*}
(v_2,\psi_2)_{H^1(\Omega_1)} & = (v_2,(E\psi_2)|_{\Omega_1})_{H^1(\Omega_1)}  = v_2^*((E\psi_2)|_{\Omega_1}) \\ & = \tilde{v}_2^*(E\psi_2) = (L^*w, E\psi_2)_{L^2(\Omega)} =  (L^*w, \psi_2)_{L^2(\Omega_1)} \;.
\end{align*}
\noindent Using the integration by parts formula \eqref{eq:int_by_parts} with $\Omega_1$ instead of $\Omega$, we infer
\begin{align*}
(v_2,\psi_2)_{H^1(\Omega_1)} & = (L^*w, \psi_2)_{L^2(\Omega_1)} \\ 
    & =  (L\psi_2,w)_{\Omega_1} +\langle \nu\cdot(g \nabla \psi_2),w \rangle_{\ast, \partial \Omega_1} \\ & \quad +i \langle \nu\cdot(A_1 + A_2)\psi_2,w \rangle_{\ast, \partial \Omega_1} - \langle \psi_2,\nu\cdot(g\nabla w) \rangle_{\ast, \partial \Omega_1} \;.
\end{align*} 
Here we have denoted the $H^{-\frac{1}{2}}(\partial \Omega_1)$, $H^{\frac{1}{2}}(\partial \Omega_1)$ duality pairing by $\langle \cdot, \cdot \rangle_{\ast, \partial \Omega_1}$.

The right hand side of the above equation vanishes because of the second part of formula \eqref{eq:zero_on_boundary} and because $\psi_2\in \tilde{S}_{g,A_1,A_2,V,q}$. Thus we have obtained that 
$$ (v,\psi)_{L^2(\Sigma_1)\times H^1(\Omega_1)} =0 \quad\mbox{for all} \quad \psi = (\psi_1,\psi_2) \in (L^2(\Sigma_1) \times \tilde{S}_{g,A_1,A_2,V,q})\;,$$
that is, $v \perp (L^2(\Sigma_1) \times \tilde{S}_{g,A_1,A_2,V,q})$. 
\end{proof}

The desired uniqueness result of Theorem \ref{thm:Q11/2}
now follows from Alessandrini's identity.

\begin{proof}[Proof of Theorem \ref{thm:Q11/2}]
Using the assumption that the DN maps coincide and Lemma \ref{lem:Aless} with $s=1/2$, we see that
\begin{align}\label{Q1_1}
\begin{split}
0
&=\langle(\Lambda_1 - \Lambda_2)f_1, f_2 \rangle_{\ast} \\
&= \int\limits_{\Omega_1}(U_1 - U_2) u_1 \overline{u_2} dx + i \int\limits_{\Omega_1} (A_1 - A_2) \cdot (u_1 \overline{\nabla u_2} - u_2 \cdot \overline{\nabla  u_1})   dx + \int\limits_{\Sigma_1}(q_1- q_2) u_1 \overline{u_2} d \mathcal{H}^{n-1}
\end{split}
\end{align}
\noindent holds for every $f_1, f_2 \in C^\infty_c(\Sigma_2)$, where $u_1, u_2$ are the solutions of \eqref{eq:Schroedinger} associated with the respective boundary data and potentials. For the sake of simplicity, here we set $U_j := V_j+|A_j|^2$ and $\Lambda_j := \Lambda_{A_j,V_j,q_j}$.

 Let $\phi_j \in L^2(\Sigma_1)$ and $\psi_j \in \tilde{S}_{Id,A_j,A_j,V_j,q_j}$ for $j=1,2$. By Lemma \ref{Lem:Runge_with_magnetic_potentials}, for every $k\in\mathbb N$ we can find $f_1^{(k)},f_2^{(k)} \in C^\infty_c(\Sigma_2)$ such that $$\|\phi_j-u_j^{(k)}|_{\Sigma_1}\|_{L^2(\Sigma_1)} < k^{-1} \quad\mbox{ and }\quad \|\psi_j-u_j^{(k)}|_{\Omega_1}\|_{H^1(\Omega_1)} < k^{-1} \;, \quad\quad j=1,2\;,$$
\noindent where $u_j^{(k)}$ solves  \eqref{eq:Schroedinger} with boundary value $f_j^{(k)}$ and potentials $A_j, V_j, q_j$. We now substitute these solutions $u_1^{(k)}$, $u_2^{(k)}$ into formula \eqref{Q1_1} and send $k\rightarrow \infty$. Given the approximations above, by Cauchy-Schwarz the limits can be moved inside the integrals. In fact, 
\begin{align*}
\int_{\Sigma_1} & |(q_1-q_2)u_{1}^{(k)}\overline{u_{2}^{(k)}}| d \mathcal{H}^{n-1} \leq \|q_1-q_2\|_{L^\infty(\Sigma_1)}\|u_{1}^{(k)}|_{\Sigma_1}\|_{L^2(\Sigma_1)}\|u_{2}^{(k)}|_{\Sigma_1}\|_{L^2(\Sigma_1)} \\ & \leq c ( \|\phi_1-u_{1}^{(k)}|_{\Sigma_1}\|_{L^2(\Sigma_1)} + \|\phi_1\|_{L^2(\Sigma_1)} )(\|\phi_2-u_{2}^{(k)}|_{\Sigma_1}\|_{L^2(\Sigma_1)} + \|\phi_2\|_{L^2(\Sigma_1)}) \\ & \leq c ( 1 + \|\phi_1\|_{L^2(\Sigma_1)} )(1 + \|\phi_2\|_{L^2(\Sigma_1)})  < c\;,
\end{align*} \begin{align*}
\int_{\Omega_1} & |u_{1}^{(k)}(A_1-A_2)\cdot\overline{\nabla u_{2}^{(k)}}| dx  \leq \|A_1-A_2\|_{L^\infty(\Omega_1)}\|u_{1}^{(k)}|_{\Omega_1}\|_{L^2(\Omega_1)}\|\nabla u_{2}^{(k)}|_{\Omega_1}\|_{L^2(\Omega_1)} \\ & \leq c ( \|\psi_1-u_{1}^{(k)}|_{\Omega_1}\|_{L^2(\Omega_1)} + \|\psi_1\|_{L^2(\Omega_1)} )(\|\nabla \psi_2-\nabla u_{2}^{(k)}|_{\Omega_1}\|_{L^2(\Omega_1)} + \|\nabla \psi_2\|_{L^2(\Omega_1)}) \\ & \leq c ( \|\psi_1-u_{1}^{(k)}|_{\Omega_1}\|_{H^1(\Omega_1)} + \|\psi_1\|_{H^1(\Omega_1)} )(\|\psi_2-u_{2}^{(k)}|_{\Omega_1}\|_{H^1(\Omega_1)} + \|\psi_2\|_{H^1(\Omega_1)}) \\ & \leq c ( 1 + \|\psi_1\|_{H^1(\Omega_1)} )(1 + \|\psi_2\|_{H^1(\Omega_1)})  < c\;,
\end{align*}
\noindent and similarly for the other terms. Eventually, we have proved that the following formula holds for every $\phi_j \in L^2(\Sigma_1)$ and $\psi_j \in \tilde{S}_{Id,A_j,A_j,V_j,q_j}$ for $j=1,2$:
\begin{align}\label{Q1_2}\begin{split}
\int_{\Omega_1} (A_1-A_2)\cdot(\psi_2\overline{\nabla \psi_1} - \psi_1\overline{\nabla \psi_2}) \; dx + \int_{\Omega_1} (U_1-U_2)\psi_1 \overline{\psi_2} dx +\int_{\Sigma_1} (q_1-q_2)\phi_1 \overline{\phi_2} d \mathcal{H}^{n-1} &=0\;.
\end{split}
\end{align}
If we substitute $\psi_1 = \psi_2 = 0$ and $\phi_2 = 1$ into \eqref{Q1_2}, we are left with only $$\int_{\Sigma_1} (q_1-q_2)\phi_1 d \mathcal{H}^{n-1} =0\;,$$ \noindent which by the arbitrary choice of $\phi_1 \in L^2(\Sigma_1)$ implies $q_1 = q_2$ in $\Sigma_1$. In light of this, formula \eqref{Q1_2} is reduced to \begin{equation}\label{Alex_final} \int_{\Omega_1} (A_1-A_2)\cdot(\psi_2\overline{\nabla \psi_1} - \psi_1\overline{\nabla \psi_2}) \; dx + \int_{\Omega_1} (U_1-U_2)\psi_1 \overline{\psi_2} dx=0\;. \end{equation}
\noindent The problem of deducing information about the magnetic and electric potentials from the above equation has been studied e.g. in \cite{NSU95, T98, DDSKSU07}, see also the survey \cite{S06}. In all these uniqueness results the key step consists in the construction of suitable complex geometrical optics solutions of the form
$$ u(x) = e^{\frac{\phi+i\psi}{h}}(a(x) + hr(x,h))\;, $$
\noindent for appropriate phase functions $\phi, \psi$, amplitudes $a$ and decaying errors $r$. Substituting such a special solution into equation \eqref{Alex_final} and using our Runge approximation results from Lemma \ref{Lem:Runge_with_magnetic_potentials} allows us to deduce that $V_1 = V_2$ and $dA_1 = dA_2$ as in the cited references, which concludes the proof of Theorem \ref{thm:Q11/2}.
\end{proof}

\section{Simultaneous Runge Approximation $s\in (0,1)$}
\label{sec:Runges}

Similarly as in deriving the results in the previous section, we can also deduce simultaneous Runge approximation results for the ``Caffarelli-Silvestre extension'' for general $s\in (0,1)$.

In analogy to the setting in the previous section we thus set
\begin{align*}
S_{s,A,V,q} &:= \{u \in L^2(\Omega, d^{1-2s}): \ u \mbox{ is a weak solution to } \eqref{eq:frac_Schr} \mbox{ in } \Omega \},\\
\tilde{S}_{s,A,V,q} &:= \{u \in H^1(\Omega, d^{1-2s}): \ u \mbox{ is a weak solution to } \eqref{eq:frac_Schr} \mbox{ in } \Omega_1 \} \subset L^2(\Omega_1, d^{1-2s}).
\end{align*}

In order to illustrate these ideas we only discuss the $L^2(\Sigma_1) \times L^2(\Omega_1, d^{1-2s})$ approximation result in the case that $A=0$.

\begin{prop}
\label{prop:s_Runge}
Assume that $A=0$ and that $V,q$ satisfy the conditions from \eqref{eq:A1A2} and \eqref{eq:bds} and let $P_{s}$ be the associated Poisson operator. Then the set
\begin{align*}
\mathcal{R}_{bbs}:=\{(u|_{\Sigma_1}, u|_{\Omega_1}): u|_{\Sigma_1} = P_{s} f|_{\Sigma_1} \mbox{ and } u|_{\Omega_1} = P_s f|_{\Omega_1} \mbox{ with } f \in C_c^{\infty}(\Sigma_2)\} \subset L^2(\Sigma_1) \times S_{s,0,V,q}
\end{align*}
is dense in $L^2(\Sigma_1)\times \tilde{S}_{s,0,V,q}$ equipped with the $L^2(\Sigma_1)\times L^2(\Omega_1, d^{1-2s})$ topology. The operator $P_s$ denotes the Poisson operator from Definition \ref{defi:Poissons}.
\end{prop}

\begin{proof}
\emph{Step 1: Set-up.}
The argument is similar as in the one for Lemma \ref{lem:sim_dense}. To this end, we first note that with respect to the $L^2(\Omega)$ scalar product, we have that $ (L^2(\Omega, d^{1-2s}))^{\ast} \sim L^2(\Omega, d^{2s-1})$. As a consequence, as above, we seek to prove that if $(v_1, v_2) \in L^2(\Sigma_1) \times L^2(\Omega, d^{2s-1}) $  satisfies $(v_1, v_2)  \perp (u|_{\Sigma_1}, u|_{\Omega})$ with $u = P_{s} (f)$ with $f\in C_c^{\infty}(\Sigma_2)$ (with orthogonality with respect to the $L^2(\Sigma_1)\times L^2(\Omega_1)$ scalar product), then also $(v_1, v_2) \perp L^2(\Sigma_1) \times \tilde{S}_{s,0,V,q}$ holds.
To this end, we consider weak solutions to the adjoint problem
\begin{align}
\label{eq:dual_s}
\begin{split}
- \nabla \cdot d^{1-2s} \nabla w + d^{1-2s} V w & = v_2 \chi_{\Omega_1} \mbox{ in } \Omega,\\
\lim\limits_{d(x) \rightarrow 0} d^{1-2s}\p_{\nu} w + q w & = -v_1 \mbox{ on } \Sigma_1,\\
w & = 0 \mbox{ on } \partial \Omega \setminus \Sigma_1.
\end{split}
\end{align}
Let us thus assume that $(v_1, v_2)\in L^2(\Sigma_1) \times L^2(\Omega, d^{2s-1})$ are such that for all $u = P_{s}(f)$ with $f\in C^{\infty}_c(\Sigma_1)$ we have
\begin{align}
\label{eq:orthog_s}
0 = (v_1, u)_{L^2(\Sigma_2)} + (v_2, u)_{L^2(\Omega)} .
\end{align}
We remark that due to the assumptions that $v_2 \in L^2(\Omega, d^{2s-1})$ and $u \in L^2(\Omega, d^{1-2s})$ the bulk $L^2(\Omega)$ scalar product is well-defined. 

\medskip

\emph{Step 2: Orthogonality.} We argue on the level of the strong equation. This can be justified as in the proof of Lemma \ref{lem:sim_dense} using the boundedness and convergence results from Lemma \ref{lem:normal_deriv}.
Beginning with the bulk contribution and using the dual equation, we then obtain
\begin{align*}
(u,v_2)_{L^2(\Omega_1)} 
& = (u, -\nabla \cdot d^{1-2s} \nabla w + d^{1-2s} V w)_{L^2(\Omega)}\\
& = \langle u, \lim\limits_{d \rightarrow 0} d^{1-2s} \p_{\nu} w \rangle_{\ast,s} - \langle \lim\limits_{d \rightarrow 0} d^{1-2s}\p_{\nu} u, w \rangle_{\ast,s} \\
& = (u, qw)_{L^2(\Sigma_1)} - (u, v_1)_{L^2(\Sigma_1)} - ( q u, w)_{L^2(\Sigma_1)} + \langle f, \lim\limits_{d \rightarrow 0} d^{1-2s} \p_{\nu} w \rangle_{\ast,s}\\
& = -(u,v_1)_{L^2(\Sigma_1)} + \langle f, \lim\limits_{d \rightarrow 0} d^{1-2s} \p_{\nu} w \rangle_{\ast,s} ,
\end{align*}
where $u = P_{s}(f)$ and $f\in C_c^{\infty}(\Sigma_2)$ and $\langle \cdot, \cdot \rangle_{\ast,s}$ denotes the $H^{-s}(\partial \Omega)$, $H^{s}(\partial \Omega)$ duality pairing.
Combining this with \eqref{eq:orthog_s}, we obtain that 
\begin{align*}
0= \langle f, \lim\limits_{d \rightarrow 0} d^{1-2s} \p_{\nu} w \rangle_{\ast,s} \mbox{ for all } f \in C_c^{\infty}(\Sigma_2).
\end{align*}
Hence, $\lim\limits_{d \rightarrow 0} d^{1-2s} \p_{\nu} w = 0$ on $\Sigma_2$. Since moreover also $w|_{\Sigma_2} =0$, boundary unique continuation for the fractional Schrödinger equation \eqref{eq:dual_s} implies that $w \equiv 0$ in $\Omega \setminus \Omega_1$. Indeed, it is possible to flatten the boundary $\partial \Omega$ by a suitable diffeomorphism and then invoke the unique continuation results from for instance \cite{Rue15, FF14} or \cite{Yu16}.
Now, by definition of $w$ (see \eqref{eq:dual_s}), this however implies that $v_1 \equiv 0$. 

Further, for $h \in \tilde{S}_{s,0,V,q}$, by the vanishing of $w|_{\partial \Omega_1}$ and $\lim\limits_{d \rightarrow 0} d^{1-2s} \p_{\nu} w|_{\partial \Omega_1}$, we infer that
\begin{align*}
(h,v_2)_{L^2(\Omega_1)} &= (h, -\nabla \cdot d^{1-2s} \nabla w + V d^{1-2s} w)_{L^2(\Omega_1)}\\
&= (-\nabla \cdot d^{1-2s} \nabla h + d^{1-2s} V h, w)_{L^2(\Omega_1)} =0.
\end{align*}
Here the last equality follows from the fact that $h\in \tilde{S}_{s,0,V,q}$. Thus, in particular, $v_2 \perp  \tilde{S}_{s,0,V,q}$, which concludes the argument.
\end{proof}

Using the simultaneous bulk and boundary approximation result of Proposition \ref{prop:s_Runge}, it is possible to recover $V$ and $q$ simultaneously also in this weighted setting:

\begin{thm}
\label{thm:Q1s}
Let $\Omega \subset \mathbb R^n$, $n\geq 3$, be an open, bounded and $C^3$-regular domain. Suppose that $d\in C^2(\Omega)$. Assume $\Omega_1\Subset \Omega$ is an open, bounded set with $\Omega \setminus \Omega_1$ simply connected and that $\Sigma_1, \Sigma_2 \subset \p\Omega$ are two disjoint, relatively open sets. If the potentials $q_1, q_2 \in L^\infty(\Sigma_1)$ and $V_1, V_2 \in C^\infty_c(\Omega_1)$ in the equation \eqref{eq:frac_Schr} are such that $$\Lambda_{s,1}:= \Lambda_{s,0, V_1, q_1}=\Lambda_{s,0,V_2,q_2}=:\Lambda_{s,2}\;,$$ 
\noindent then $q_1=q_2$ and $\;V_1 = V_2$.
\end{thm}

\begin{proof}
By virtue of the Alessandrini identity we obtain
\begin{align*}
0 = \int\limits_{\Omega_1} d^{1-2s} (V_1 - V_2)u_1 \overline{u_2} dx + \int\limits_{\Sigma_1} (q_1- q_2) u_1 \overline{u_2} d\mathcal{H}^{n-1},
\end{align*}
for $u_1, u_2$ weak solutions to \eqref{eq:frac_Schr}. Now an approximation argument as in the proof of Theorem \ref{thm:Q11/2} implies that for every $\phi_j \in L^2(\Sigma_1)$ and $\psi_j \in \tilde{S}_{s,0,V_j,q_j}$ and $j\in \{1,2\}$ we obtain
\begin{align*}
0 = \int\limits_{\Omega_1} d^{1-2s} (V_1 - V_2) \psi_1  \overline{\psi_2} dx + \int\limits_{\Sigma_1} (q_1-q_2) \phi_1 \overline{\phi_2} d\mathcal{H}^{n-1}.
\end{align*}
With this in hand, the proof that $q_1 = q_2$ is immediate by choosing $\psi_1 = \psi_2 =0$, $\phi_1 \in C_c^{\infty}(\Sigma_1)$ arbitrary and $\phi_2 =1$.
The uniqueness $V_1 = V_2$ follows by a reduction of the problem in $\Omega_1$ to a Schrödinger type problem. Carrying out a Liouville transform (see for instance \cite{Salo08}), the equation
\begin{align*}
- \nabla \cdot d^{1-2s} \nabla u + V d^{1-2s} u = 0 \mbox{ in } \Omega_1
\end{align*}
is transferred to the Schrödinger type problem
\begin{align*}
- \D w +(Q + V d^{\frac{1-2s}{2}}) w = 0 \mbox{ in } \Omega_1\;,
\end{align*}
where $Q := \frac{\D d^{\frac{1-2s}{2}}}{d^{\frac{1-2s}{2}}}$ and $w:=d^{\frac{1-2s}{2}}u$. We note that $Q \in L^{\infty}(\Omega_1)$ since $d \in C^2(\Omega)$ and $\dist(\partial \Omega_1, \partial \Omega)>0$. Now, standard CGO constructions allow to obtain solutions of the form
\begin{align*}
w_1 = e^{i \xi \cdot x}(e^{i k \cdot x} + r_1), \ w_2 = e^{i \xi' \cdot x}(e^{- i k \cdot x} + r_{2}),
\end{align*}
with $\xi, \xi' \in \C^n$, $k\in \R^n$, $\xi \cdot \xi = k\cdot \xi=0$, $\xi' = - \text{Re}(\xi) + i \text{Im}(\xi)$ and $\|r_{j}\|_{L^2(\Omega_1)} \rightarrow 0$ as $|\xi'| \rightarrow \infty$.
Then the functions
\begin{align*}
u_j := d^{\frac{2s-1}{2}} w_j, \ j \in \{1,2\}
\end{align*}
however solve the equation
\begin{align*}
d^{\frac{2s-1}{2}}(- \nabla \cdot d^{1-2s} \nabla w_j + V d^{1-2s} w_j) = 0 \mbox{ in } \Omega_1
\end{align*}
in a weak sense. Due to the assumed regularity of $d$, they also satisfy 
\begin{align*}
- \nabla \cdot d^{1-2s} \nabla w_j + V d^{1-2s} w_j = 0 \mbox{ in } \Omega_1
\end{align*}
in a weak sense. By virtue of the result from Proposition \ref{prop:s_Runge} we may thus approximate these functions by functions $\psi_j \in S_{s,0,V,q}$. Inserting these into the Alessandrini identity, recalling that $q_1 = q_2$ and passing to the limit in the approximation parameter then implies 
\begin{align*}
0 = \int\limits_{\Omega_1} (V_1 - V_2) e^{2 i k \cdot x} dx.
\end{align*}
As a consequence, also $V_1 = V_2$.
\end{proof}

\begin{rmk}
\label{rmk:CGO} 
While in the study of the question (Q1) the situation in which $\Omega_1 \Subset \Omega$, the construction of CGOs to the degenerate equation \eqref{eq:frac_Schr} can essentially be avoided by using the non-degeneracy of the equation in $\Omega_1$, this can no longer be circumvented in the setting of question (Q2).

We refer to the next two sections for the construction of a new family of CGO type solutions for a closely related equation. These will be used to answer the question (Q2) in the case $s \in (\frac{1}{2},1)$ and will also provide a partial answer in the case $s= \frac{1}{2}$.
\end{rmk}

\section{On a Carleman Estimate for the ``Caffarelli-Silvestre Extension''}
\label{sec:Carl}

In this and the next section we address the question (Q2) for $s\geq \frac{1}{2}$ in the absence of magnetic potentials. As a major ingredient, we here construct CGO solutions to the equation
\begin{align} 
\label{eq:weighted_problem_a} 
\begin{split}
\nabla \cdot x_{n+1}^{1-2s} \nabla u + x_{n+1}^{1-2s} V u & = 0 \mbox{ in } \Omega,\\
\lim\limits_{x_{n+1}\rightarrow 0} x_{n+1}^{1-2s} \p_{n+1} u+ q u & = 0 \mbox{ on } \Sigma_1,
\end{split}
\end{align}
where $\Sigma_1 = \overline{\Omega}\cap \{x_{n+1}=0\}$ is assumed to be a smooth, $n$-dimensional set and $\partial \Omega$ is $C^{\infty}$ regular (the arguments from below show that $C^m$-regular with $m=m(s)>0$ would suffice).
The CGO construction is achieved by virtue of a duality argument and a suitable Carleman estimate.

The degenerate behaviour of the equation is reflected in the form of the CGOs. In order to avoid issues with the Muckenhoupt weight in the equation at $x_{n+1}=0$, using the notation $x=(x', x_{n+1}) \in \mathbb{\R}^{n+1}_+$, we only consider wave vectors $\xi' \in \C^n$ with $\xi' \cdot \xi' = 0$ which are orthogonal to $e_{n+1}$. More precisely, we seek to construct solutions of the form
\begin{align*}
u(x) = e^{ \xi' \cdot x'}( a(x)+ r(x)).
\end{align*}
with amplitudes $a(x)=e^{ik'\cdot x' + i k_{n+1} x_{n+1}^{2s}} $, $k\in \R^{n+1}$, and errors $r: \Omega \rightarrow \R$. We emphasize that the nonlinear (in $x_{n+1}$) phase dependence $i k_{n+1} x_{n+1}^{2s}$ is also a consequence of the degenerate elliptic character of the equation (see the estimate for $\tilde{L}_{-\xi',V}^s$ in \eqref{eq:CGO_1} in the proof of Proposition \ref{CGO_construction}).
The function $r: \Omega \rightarrow \R$ is an error for which we seek to produce decay estimates as $|\xi'| \rightarrow \infty$ by means of a suitable Carleman estimate.

We begin by a discussion of the Carleman estimate which underlies our CGO construction. 

\begin{prop}[Carleman estimate]
\label{prop:Carl}
Let $s\in [\frac{1}{2},1)$ and let $\xi' \in \C^n$ be such that $\xi' \cdot \xi' = 0$. Assume that $\Omega \subset \overline{\R}^{n+1}_+$ is a smooth domain and that $\overline\Omega \cap \{x_{n+1}=0\} =: \Sigma_1$ is a smooth, $n$-dimensional set.  If $s=\frac{1}{2}$, further assume that $\|q\|_{L^{\infty}(\Sigma_1)}$ is sufficiently small.
Let $f\in (H^{1}(\Omega,x_{n+1}^{1-2s}))^{\ast}$ with $\supp(f) \subset \Omega \cup (\overline{\Omega}\cap \{x_{n+1}=0\})$ and $g \in L^2(\Sigma_1)$.
Then, for $u \in H^1_{\partial \Omega \setminus \overline{\Sigma_1},0}(\Omega, x_{n+1}^{1-2s})\cap \mathcal{C}$ with $u = 0$ and $\lim\limits_{x \rightarrow \partial \Omega} x_{n+1}^{1-2s} \p_{\nu} u = 0$ on $\partial \Omega \setminus \overline{\Sigma_1}$ being a weak solution to
\begin{align}
\label{eq:main_eq}
\begin{split}
\nabla \cdot x_{n+1}^{1-2s} \nabla u & = f \mbox{ in } \Omega,\\
\lim\limits_{x_{n+1}\rightarrow 0} x_{n+1}^{1-2s} \p_{n+1} u + q u & = g \mbox{ on } \Sigma_1,
\end{split}
\end{align}
we have
\begin{align}
\label{eq:Carl}
\begin{split}
&|\xi'|^{s} \|e^{ \xi' \cdot x'} u\|_{L^2(\Sigma_1)} 
+ |\xi'| \|e^{ \xi' \cdot x'} x_{n+1}^{\frac{1-2s}{2}} u\|_{L^2(\Omega )} + \|e^{ \xi' \cdot x' } x_{n+1}^{\frac{1-2s}{2}} \nabla u\|_{L^2(\Omega)}\\
&\leq C (|\xi'| \|e^{ \xi' \cdot x '} x_{n+1}^{\frac{1-2s}{2}} \tilde{F} \|_{L^2(\Omega)} +  \|e^{ \xi' \cdot x '} x_{n+1}^{\frac{1-2s}{2}} F_0 \|_{L^2(\Omega)} + |\xi'|^{1-s}\|e^{ \xi' \cdot x '} g\|_{L^{2}(\Sigma_1)}).
\end{split}
\end{align}
Here the constant $C>0$ depends on $\|q\|_{L^{\infty}(\Sigma_1)}$ and $F =(F_0, \tilde{F}) \in L^2(\R^{n+1}_+,\R^{n+2})$ is the Riesz representation of $f$, i.e., it is such that
\begin{align*}
f(v) = (v, x_{n+1}^{1-2s} F_0)_{L^2(\Omega)} + (\nabla v, x_{n+1}^{1-2s} \tilde{F})_{L^2(\Omega)} \mbox{ for all } v \in H^{1}(\Omega, x_{n+1}^{1-2s}).
\end{align*}
\end{prop}

We remark that $\|F\|_{L^2(\Omega,x_{n+1}^{1-2s})} = \|f\|_{(H^{1}(\Omega,x_{n+1}^{1-2s}))^{\ast}}$.

\begin{rmk}
\label{rmk:bdry}
We remark that as $\partial \Omega$ is smooth and as $x_{n+1}=0$ on $\Sigma_1$, we have that $x_{n+1}$ vanishes to infinite order at $\partial \Sigma_1$, i.e. that the domain is arbitrarily flat in a neighbourhood of $\partial \Sigma_1$.
\end{rmk}

\begin{proof}
We argue in three steps using a splitting strategy. More precisely, we write $u = u_1 + u_2$ where $u_1$ (weakly) solves the problem
\begin{align*}
-\nabla \cdot x_{n+1}^{1-2s} \nabla u_1 + K|\xi'|^2x_{n+1}^{1-2s} u_1& = -f \mbox{ in } \Omega,\\
\lim\limits_{x_{n+1} \rightarrow 0} x_{n+1}^{1-2s} \p_{n+1} u_1  & = -q u + g \mbox{ on } \Sigma_1,\\
x_{n+1}^{1-2s}\p_{\nu} u_1 & = 0 \mbox{ on } \partial \Omega \setminus \overline{\Sigma_1 }.
\end{align*}
By the Lax-Milgram theorem, a unique (weak) solution to this problem exists in $H^1(\Omega, x_{n+1}^{1-2s})$ if $K>0$ is sufficiently large. It satisfies
\begin{align*}
(x_{n+1}^{1-2s} \nabla u_1, \nabla \varphi)_{\Omega} + K |\xi'|^2 (x_{n+1}^{1-2s} u_1, \varphi)_{\Omega} & = (F_0, x_{n+1}^{1-2s} \varphi)_{\Omega} + (\tilde{F}, x_{n+1}^{1-2s} \nabla \varphi)_{\Omega} + (-q u + g, \varphi)_{\Sigma_1}
\end{align*}
for any $\varphi \in H^1(\Omega, x_{n+1}^{1-2s})$. Here the notation $(\cdot, \cdot)_{\Omega}$ and $(\cdot, \cdot)_{\Sigma_1}$ refer to the $L^2(\Omega)$ and $L^2(\Sigma_1)$ scalar products respectively.
The function $u_2=u-u_1$ is defined accordingly.
\medskip

\emph{Step 1: Estimate for $u_1$.} We first estimate $u_1$. To this end, we test the equation for $u_1$ with $\varphi:= |\xi'|^2 e^{2 x' \cdot \xi'}u_1$. This yields
\begin{align*}
& |\xi'|^2(x_{n+1}^{1-2s} \nabla u_1, \nabla (e^{2x'\cdot \xi'} u_1))_{\Omega} + K |\xi'|^4 (x_{n+1}^{1-2s} u_1, e^{2 x' \cdot \xi'} u_1)_{\Omega}  = |\xi'|^2 (-q u + g, e^{2 x' \cdot \xi'} u_1)_{\Sigma_1}\\
& \quad - |\xi'|^2 (F_0, x_{n+1}^{1-2s}  e^{2 x' \cdot \xi'} u_1)_{\Omega} - |\xi'|^2(\tilde{F}, x_{n+1}^{1-2s}\nabla (e^{2\xi' \cdot x'} u_1))_{\Omega}.
\end{align*}
Using Young's inequality and choosing $K>0$ sufficiently large this implies that 
\begin{align}
\label{eq:apriori}
\begin{split}
&\frac{K}{2} |\xi'|^4 \|x_{n+1}^{\frac{1-2s}{2}} e^{x' \cdot \xi'} u_1\|_{L^2(\Omega)}^2
+  |\xi'|^2 \|x_{n+1}^{\frac{1-2s}{2}} e^{x' \cdot \xi'} \nabla u_1\|_{L^2(\Omega)}^2\\
&\leq C |\xi'|^2 \|x_{n+1}^{\frac{1-2s}{2}} e^{x' \cdot \xi'} \tilde{F}\|_{L^2(\Omega)}^2 + C \|x_{n+1}^{\frac{1-2s}{2}} e^{x' \cdot \xi'} F_0\|_{L^2(\Omega)}^2
+ \epsilon |\xi'|^{2+2s} \|e^{x' \cdot \xi'} u_1\|_{L^2(\Sigma_1)}^2 \\
& \quad + C_{\epsilon}|\xi'|^{2-2s}( \|e^{x'\cdot \xi'}g\|_{L^2(\Sigma_1)}^2 + \|q\|_{L^{\infty}(\Sigma_1)}^2 \|e^{x'\cdot \xi'}u\|_{L^2(\Sigma_1)}^2).
\end{split}
\end{align}
Now the boundary-bulk interpolation estimate from Lemma \ref{lem:boundary_trace_s} allows us to further add a boundary contribution to the left hand side of this:
\begin{align}
\label{eq:apriori1}
\begin{split}
&|\xi'|^{2+2s}\|e^{x' \cdot \xi'} u_1\|^2_{L^2(\Sigma_1)}
+ \frac{K}{2} |\xi'|^4 \|x_{n+1}^{\frac{1-2s}{2}} e^{x' \cdot \xi'} u_1\|_{L^2(\Omega)}^2
+  |\xi'|^2 \|x_{n+1}^{\frac{1-2s}{2}} e^{x' \cdot \xi'} \nabla u_1\|_{L^2(\Omega)}^2\\
&\leq C |\xi'|^2 \|x_{n+1}^{\frac{1-2s}{2}} e^{x' \cdot \xi'} \tilde{F}\|_{L^2(\Omega)}^2 + C \|x_{n+1}^{\frac{1-2s}{2}} e^{x' \cdot \xi'} F_0\|_{L^2(\Omega)}^2
+ \epsilon |\xi'|^{2+2s}\|e^{x' \cdot \xi'} u_1\|_{L^2(\Sigma_1)}^2 \\
& \quad + C_{\epsilon}|\xi'|^{2-2s}( \|e^{x'\cdot \xi'}g\|_{L^2(\Sigma_1)}^2 + \|q\|_{L^{\infty}(\Sigma_1)}^2\|e^{x'\cdot \xi'}u\|_{L^2(\Sigma_1)}^2).
\end{split}
\end{align}
In particular, this allows us to absorb the boundary contributions involving $u_1$ from the right hand side of \eqref{eq:apriori1} into the left hand side of this inequality. As a consequence, we obtain the bound
\begin{align}
\label{eq:apriori11}
\begin{split}
&|\xi'|^{2+2s}\|e^{x' \cdot \xi'} u_1\|^2_{L^2(\Sigma_1)}
+ \frac{K}{2} |\xi'|^4 \|x_{n+1}^{\frac{1-2s}{2}} e^{x' \cdot \xi'} u_1\|_{L^2(\Omega)}^2
+  |\xi'|^2 \|x_{n+1}^{\frac{1-2s}{2}} e^{x' \cdot \xi'} \nabla u_1\|_{L^2(\Omega)}^2\\
&\leq C |\xi'|^2 \|x_{n+1}^{\frac{1-2s}{2}} e^{x' \cdot \xi'} \tilde{F}\|_{L^2(\Omega)}^2 + C \|x_{n+1}^{\frac{1-2s}{2}} e^{x' \cdot \xi'} F_0\|_{L^2(\Omega)}^2 + \\ &\quad 
+  C_{\epsilon}|\xi'|^{2-2s}( \|e^{x'\cdot \xi'}g\|_{L^2(\Sigma_1)}^2 + \|q\|_{L^{\infty}(\Sigma_1)}^2 \|e^{x'\cdot \xi'} u\|_{L^2(\Sigma_1)}^2).
\end{split}
\end{align}

\medskip

\emph{Step 2: Estimate for $u_2$.}
Next we estimate the contribution from $u_2$ which (weakly) solves the equation
\begin{align}
\label{eq:u_2}
\begin{split}
\nabla \cdot x_{n+1}^{1-2s} \nabla u_2  & =  -K|\xi'|^2x_{n+1}^{1-2s} u_1\mbox{ in } \Omega,\\
\lim\limits_{x_{n+1} \rightarrow 0} x_{n+1}^{1-2s} \p_{n+1} u_2  & = 0 \mbox{ on } \Sigma_1,\\
x_{n+1}^{1-2s}\p_{\nu} u_2 & = 0 \mbox{ on } \partial \Omega \setminus \Sigma_1.
\end{split}
\end{align}
In order to estimate $u_2$, we first assume that $u_1 \in C^{0,\alpha}(\Omega)$ for some $\alpha \in (0,1)$. With only slight modifications it is then possible to invoke the regularity results from \cite[Appendix A]{KRSIV}. Indeed, the regularity estimates from \cite[Proposition 8.2]{KRSIV} yield $C^{2,\alpha}$ regularity up to the boundary in $\inte(\Sigma_1)$. Classical, uniformly elliptic regularity estimates in turn yield $C^{2,\alpha}$ regularity in a neighbourhood of $\partial \Omega \setminus \overline{\Sigma_1}$ up to the boundary. Thus, it remains to discuss the regularity in a neighbourhood of $\partial \Sigma_1$ up to the boundary. This however follows from the $C^2$ regularity of the boundary which implies that the approximation by the flat problem at that point is still valid. Combining these results yields the global $C^{2,\alpha}(\Omega)$ regularity of $u_2$.

In order to estimate $u_2$, we conjugate the operator $L_s := \nabla \cdot x_{n+1}^{1-2s} \nabla$ with the weight $e^{ x' \cdot \xi'}$. This yields the conjugated operator 
\begin{align*}
\tilde{L}_{s,\phi}:= \nabla \cdot x_{n+1}^{1-2s} \nabla - 2 x_{n+1}^{1-2s} \xi' \cdot \nabla'.
\end{align*}

Next, we define $u_2 = x_{n+1}^{\frac{2s-1}{2}} e^{-x'\cdot \xi'}w$ and multiply the operator $\tilde{L}_{s,\phi}$ by $x_{n+1}^{\frac{2s-1}{2}}$. As a consequence, the operator acting on $w$ turns into 
\begin{align*}
L_{s,\phi}:= x_{n+1}^{\frac{2s-1}{2}}\nabla \cdot x_{n+1}^{1-2s} \nabla x_{n+1}^{\frac{2s-1}{2}} - 2 \xi' \cdot \nabla',
\end{align*}
\noindent and since $\xi' \perp e_{n+1}$ the boundary condition on $\Sigma_1$ correspondingly becomes
\begin{align}
\label{eq:bdry1}
\lim\limits_{x_{n+1}\rightarrow 0} x_{n+1}^{1-2s} \p_{n+1}(x_{n+1}^{\frac{2s-1}{2}} w)  & = 0.
\end{align}
On $\partial \Omega \setminus \overline{\Sigma_1}$ the boundary contributions however is non-trivial and turns into 
\begin{align}
\label{eq:bdry2}
\lim\limits_{d \rightarrow 0} x_{n+1}^{1-2s} \p_{\nu}(x_{n+1}^{\frac{2s-1}{2}} w)  & =  \lim\limits_{d \rightarrow 0}  x_{n+1}^{1-2s}(\nu \cdot \xi') (x_{n+1}^{\frac{2s-1}{2}} w).
\end{align}

Up to boundary contributions the bulk part of the operator can be split into its symmetric and antisymmetric parts:
\begin{align*}
S_{\phi} &= x_{n+1}^{\frac{2s-1}{2}}\nabla \cdot x_{n+1}^{1-2s}\nabla x_{n+1}^{\frac{2s-1}{2}} ,\\
A_{\phi} &= -2 \xi' \cdot \nabla'.
\end{align*}
Expanding the norm, computing the boundary terms (BC) and using the regularity of $u_2$, we thus infer 
\begin{align}
\label{eq:Carl_exp}
\| L_{s,\phi} w\|_{L^2(\Omega)}^2
& = \|S_{\phi} w \|_{L^2(\Omega)}^2 + \|A_{\phi} w\|_{L^2(\Omega)}^2 + \mbox{ (BC)}.
\end{align}
We emphasise that the $C^{2,\alpha}$ regularity of $u_2$ allows us to carry out the expansion of $L_{s,\phi}w$ as classically differentiable functions away from the boundary and that the resulting boundary contributions are given as classical boundary integrals. Using the observations from \eqref{eq:bdry1} and \eqref{eq:bdry2} these are of the form
\begin{align*}
(BC) = (BC)_1 + (BC)_2,
\end{align*}
where the contributions from $(BC)_1$ come from shifting $(S_{\phi} w, A_{\phi} w)_{L^2(\Omega)} = (w, S_{\phi} A_{\phi}w)_{L^2(\Omega)} + (BC)_1 $ and the ones from $(BC)_2$ from $(S_{\phi} w, A_{\phi} w)_{L^2(\Omega)} = -(A_{\phi} S_{\phi} w, w)_{L^2(\Omega)} + (BC)_2 $. 

We next estimate these boundary contributions individually.\\

\emph{Step 2a: $(BC)_1$.}
For the boundary contribution $(BC)_1$ we obtain
\begin{align}
\label{eq:BC1}
\begin{split}
(BC)_1& :=-2 (x_{n+1}^{1-2s} \p_{\nu}(x_{n+1}^{\frac{2s-1}{2}} w), \xi' \cdot \nabla' (x_{n+1}^{\frac{2s-1}{2}}w))_{L^2(\partial \Omega)}\\
&\quad + 2 (x_{n+1}^{\frac{2s-1}{2}} w, x_{n+1}^{1-2s}\p_{\nu}((\xi'\cdot \nabla')(x_{n+1}^{\frac{2s-1}{2}}w)))_{L^2(\partial \Omega)}\\
& = - 2 (x_{n+1}^{1-2s} \p_{\nu}(x_{n+1}^{\frac{2s-1}{2}} w), \xi' \cdot \nabla' (x_{n+1}^{\frac{2s-1}{2}}w))_{L^2(\partial \Omega)}\\
& \quad - 2 (x_{n+1}^{\frac{2s-1}{2}} w, x_{n+1}^{1-2s} [(\xi'\cdot \nabla')\nu] \cdot \nabla (x_{n+1}^{\frac{2s-1}{2}} w) )_{L^2(\partial \Omega)}\\
& \quad + 2 (x_{n+1}^{\frac{2s-1}{2}} w, x_{n+1}^{1-2s} (\xi'\cdot \nabla') \p_{\nu} (x_{n+1}^{\frac{2s-1}{2}} w) )_{L^2(\partial \Omega)}\\
& = -2 (x_{n+1}^{1-2s} (\nu \cdot \xi') (x_{n+1}^{\frac{2s-1}{2}} w), \xi' \cdot \nabla' (x_{n+1}^{\frac{2s-1}{2}}w))_{L^2(\partial \Omega)}\\
& \quad -  2 (x_{n+1}^{\frac{2s-1}{2}} w, x_{n+1}^{1-2s} [(\xi'\cdot \nabla')\nu] \cdot \nabla (x_{n+1}^{\frac{2s-1}{2}} w) )_{L^2(\partial \Omega)}\\
& \quad +2 ((\xi' \cdot \nabla') [x_{n+1}^{1-2s} (\nu \cdot \xi') (x_{n+1}^{\frac{2s-1}{2}} w) ], x_{n+1}^{\frac{2s-1}{2}}w)_{L^2(\partial \Omega)} .
\end{split}
\end{align}

Here we have used \eqref{eq:bdry1} and \eqref{eq:bdry2} in the third equality. We now discuss these contributions separately. We split the derivative $\xi' \cdot \nabla'$ into a tangential and a normal contribution. If $\tau_j(x)$, $j=1,...,n$ are unit vectors depending smoothly on $x$ and forming with the addition of $\nu(x)$ an orthonormal basis of $\mathbb R^{n+1}$, then we can write 
$$ \nabla = \nu(x) \p_\nu + \sum_{j=1}^n \tau_j(x) (\tau_j(x)\cdot \nabla)\;, $$
\noindent and therefore 
\begin{align}
\label{eq:tan_norm_split}
\xi' \cdot \nabla'  = |\xi'|[ (e_{\xi'}\cdot \nu(x)) \p_\nu + \sum_{j=1}^n (e_{\xi'}\cdot \tau_j(x)) (\tau_j(x)\cdot \nabla) ] =  |\xi'|[ (e_{\xi'}\cdot \nu(x)) \p_\nu + \beta(x)\cdot\nabla_\tau ] \;,
\end{align} 
where $e_{\xi'}:= \frac{1}{|\xi'|}\xi'$, $ \beta$ is a smooth vector function whose norm is bounded uniformly, independently of $|\xi'|$ and whose j-th component is $e_{\xi'}\cdot \tau_j(x)$, and the operator $\nabla_\tau$ represents the tangential derivatives $\tau_j(x)\cdot \nabla$. 

For the first contribution in \eqref{eq:BC1}, we use the splitting \eqref{eq:tan_norm_split} in combination with \eqref{eq:bdry1}, \eqref{eq:bdry2} for the normal derivatives and integrate by parts in the tangential directions:
\begin{align}
\label{eq:bdry_est1}
\begin{split}
2 (x_{n+1}^{1-2s} (\nu \cdot \xi')& (x_{n+1}^{\frac{2s-1}{2}} w),\xi' \cdot \nabla' (x_{n+1}^{\frac{2s-1}{2}}w))_{L^2(\partial \Omega)}
\\ & = 2|\xi'|^2  (x_{n+1}^{1-2s} (\nu \cdot e_{\xi'}) (x_{n+1}^{\frac{2s-1}{2}} w), (e_{\xi'}\cdot \nu) \p_\nu(x_{n+1}^{\frac{2s-1}{2}}w))_{L^2(\partial \Omega)}
\\ & \quad + 2 |\xi'|^2 (x_{n+1}^{1-2s} (\nu \cdot e_{\xi'}) (x_{n+1}^{\frac{2s-1}{2}} w), \beta(x)\cdot\nabla_\tau (x_{n+1}^{\frac{2s-1}{2}}w))_{L^2(\partial \Omega)}
\\ & = 2|\xi'|^3([x_{n+1}^{1-2s}(\nu \cdot e_{\xi'})^3] (x_{n+1}^{\frac{2s-1}{2}} w), (x_{n+1}^{\frac{2s-1}{2}}w))_{L^2(\partial \Omega )}
\\ & \quad + |\xi'|^2 (x_{n+1}^{1-2s} (\nu \cdot e_{\xi'})\beta(x), \nabla_\tau (x_{n+1}^{2s-1}|w|^2))_{L^2(\partial \Omega)}
\\ &=  -|\xi'|^2 ( (x_{n+1}^{\frac{2s-1}{2}} w)[\di_{\partial \Omega}(\beta(x) x_{n+1}^{1-2s} (\nu \cdot e_{\xi'}))]  , x_{n+1}^{\frac{2s-1}{2}} w )_{L^2(\partial \Omega )}\\
& \quad + 2|\xi'|^3([x_{n+1}^{1-2s}(\nu \cdot e_{\xi'})^3] (x_{n+1}^{\frac{2s-1}{2}} w), (x_{n+1}^{\frac{2s-1}{2}}w))_{L^2(\partial \Omega )}.
\end{split}
\end{align}
We remark that both boundary terms are controlled by 
\begin{align}
\label{eq:bdry_dominated_term}
|\xi'|^3 \|x_{n+1}^{\frac{2s-1}{2}} w\|_{L^2(\partial \Omega \setminus \Sigma_1)}^2.
\end{align}
Indeed, to observe this, it suffices to prove that for $x\in \partial \Omega$ with $x_{n+1}\rightarrow 0$ we have that for the weights 
\begin{align}
\label{eq:weights_bdry}
[x_{n+1}^{1-2s}(\nu(x) \cdot e_{\xi'})^3] \rightarrow 0,\
[\di_{\partial \Omega}(\beta(x) x_{n+1}^{1-2s} (\nu(x) \cdot e_{\xi'}))]   \rightarrow 0 \mbox{ as } x_{n+1}\rightarrow 0.
\end{align}
Parametrizing the boundary $\partial \Omega$ in a neighbourhood of $\partial \Sigma_1$, we obtain that if $\partial \Omega$ is sufficiently smooth and thus sufficiently flat at $\partial \Omega$ the claim of \eqref{eq:weights_bdry} can always be ensured. Indeed, in this case the boundary can be locally parametrized by $\psi(x) = (x', |x'-\gamma(x')|^m)$, where $\gamma(x')$ is a smooth function describing $\partial \Sigma_1$. Thus, expressing $x_{n+1}$ and $\nu(x')\cdot e_{\xi'}$ in terms of $x'$, for instance yields 
\begin{align*}
|x_{n+1}^{1-2s}(\nu(x')\cdot e_{\xi'}) |\leq C_{\gamma, |\nabla' \gamma|} |x'-\gamma(x')|^{m(1-2s)} |x'-\gamma(x')|^{m-1} \rightarrow 0,
\end{align*}  
as $x' \rightarrow \gamma(x')$ and thus $x_{n+1}\rightarrow 0$
by choosing $m=m(s)>0$ sufficiently large (which is ensured by the boundary smoothness, see Remark \ref{rmk:bdry}). Since in local coordinates the expression for the divergence only involves derivatives in the tangential directions, the same argument applies to the second expression in \eqref{eq:bdry_est1}. Together with the boundedness of $\overline{\Omega}$ this proves the bound \eqref{eq:bdry_dominated_term}.

The third term in \eqref{eq:BC1} can be treated analogously as the first term in \eqref{eq:BC1} . To this end, we first note that 
\begin{align}
\label{eq:bdry_est3}
\begin{split}
&2((\xi' \cdot \nabla') [x_{n+1}^{1-2s} (\nu \cdot \xi') (x_{n+1}^{\frac{2s-1}{2}} w) ], x_{n+1}^{\frac{2s-1}{2}}w)_{L^2(\partial \Omega)}\\
&= 2( x_{n+1}^{1-2s} (\nu \cdot \xi') (\xi' \cdot \nabla')(x_{n+1}^{\frac{2s-1}{2}} w) , x_{n+1}^{\frac{2s-1}{2}}w)_{L^2(\partial \Omega)}\\
& \quad + 2 ((x_{n+1}^{\frac{2s-1}{2}} w)  [(\xi' \cdot \nabla') (x_{n+1}^{1-2s} (\nu \cdot \xi') )], x_{n+1}^{\frac{2s-1}{2}}w)_{L^2(\partial \Omega)}.
\end{split}
\end{align}
Hence, the first contribution is of the same form as the term from \eqref{eq:bdry_est1}. It suffices to deal with the second one and to prove that 
\begin{align*}
[(\xi' \cdot \nabla') (x_{n+1}^{1-2s} (\nu \cdot \xi') )] \rightarrow 0
\end{align*}
for $x \in \partial \Omega$ with $x_{n+1}\rightarrow 0$. This however follows in the same way as in \eqref{eq:weights_bdry} and implies that the contributions in \eqref{eq:bdry_est3} are also controlled by terms of the form \eqref{eq:bdry_dominated_term}.

Finally, it remains to deal with the second contribution in \eqref{eq:BC1}. For this we observe that $(\xi' \cdot \nabla') \nu$ does not have any normal component. Thus, an integration by parts yields
\begin{align}
\label{eq:bdry_est2}
\begin{split}
& - 2 (x_{n+1}^{\frac{2s-1}{2}} w, x_{n+1}^{1-2s} [(\xi'\cdot \nabla')\nu] \cdot \nabla (x_{n+1}^{\frac{2s-1}{2}} w) )_{L^2(\partial \Omega)}\\
& = - (x_{n+1}^{1-2s} [(\xi'\cdot \nabla')\nu] ,\nabla (x_{n+1}^{2s-1} |w|^2) )_{L^2(\partial \Omega)}\\
& =  ([\di_{\partial \Omega}( x_{n+1}^{1-2s} [(\xi'\cdot \nabla')\nu])] ( x_{n+1}^{\frac{2s-1}{2}} w),  (x_{n+1}^{\frac{2s-1}{2}} w) )_{L^2(\partial \Omega)}.
\end{split}
\end{align}
It remains to prove that
\begin{align*}
[\di_{\partial \Omega}( x_{n+1}^{1-2s} [(\xi'\cdot \nabla')\nu])] \rightarrow 0 
\end{align*}
for $x \in \partial \Omega$ with $x_{n+1}\rightarrow 0$, as this then ensures that also the boundary contribution in \eqref{eq:bdry_est2} is controlled by \eqref{eq:bdry_dominated_term}. The desired estimate however follows from the explicit parametrization $\psi(x)=(x', |x'-\gamma(x')|^m)$, which yields that
\begin{align*}
|[\di_{\partial \Omega}( x_{n+1}^{1-2s} [(\xi'\cdot \nabla')\nu])]| \leq C |x'-\gamma(x')|^{m(1-2s)+m-2}|\xi'|.
\end{align*}
Thus, for $m=m(s)>0$ sufficiently large, the claim follows.

Inspecting the quantities in \eqref{eq:bdry_est1}-\eqref{eq:bdry_est2} and recalling that $\xi' \perp e_{n+1}$, we note that all right hand side contributions in \eqref{eq:bdry_est1}-\eqref{eq:bdry_est2} are really only integrals over $\partial \Omega \setminus \overline{\Sigma_1}$.
Thus, due to the assumed boundary regularity of $\Omega$  and the boundedness of $\Omega$, all of the contributions on the right hand side of \eqref{eq:BC1} are bounded in terms of \eqref{eq:bdry_dominated_term}.

Last but not least, we seek to estimate the quantity \eqref{eq:bdry_dominated_term} by bulk contributions of $u_1$.
Rewriting \eqref{eq:bdry_dominated_term} in terms of $u_2$, recalling that $u_2 = u -u_1$ and that $u|_{\partial \Omega \setminus \overline{\Sigma_1}} = 0$, we infer that all boundary contributions in $(BC)_1$ are controlled by 
\begin{align}
\label{eq:bdry_aux1}
|\xi'|^3 \|e^{\xi' \cdot x'} u_2\|_{L^2(\partial \Omega \setminus \overline{\Sigma_1})}^2 \leq |\xi'|^3 \|e^{\xi' \cdot x'} u_1\|_{L^2(\partial \Omega \setminus \overline{\Sigma_1})}^2.
\end{align}
Using the trace estimate from Lemma \ref{lem:boundary_trace} and the fact that $s\geq \frac{1}{2}$, we deduce that
\begin{align}
\label{eq:bdry_aux2}
\begin{split}
|\xi'|^3 \|e^{\xi' \cdot x'} u_2\|_{L^2(\partial \Omega \setminus \overline{\Sigma_1})}^2 &\leq |\xi'|^3 \|e^{\xi' \cdot x'} u_1\|_{L^2(\partial \Omega \setminus \overline{\Sigma_1})}^2\\
&\leq C(|\xi'|^{4}\|e^{\xi' \cdot x'} u_1\|_{L^2(\Omega)}^2 + |\xi'|^{2}\|\nabla ( e^{\xi' \cdot x'} u_1) \|_{L^2(\Omega)}^2)\\
&\leq C (|\xi'|^{4} \|e^{\xi' \cdot x'} x_{n+1}^{\frac{1-2s}{2}} u_1\|_{L^2(\Omega)}^2 + |\xi'|^{2} \| e^{\xi' \cdot x'} x_{n+1}^{\frac{1-2s}{2}} \nabla u_1\|_{L^2(\Omega)}^2).
\end{split}
\end{align}

\emph{Step 2b: $(BC)_2$.}
Next we deal with the contributions in $(BC)_2$. These are of the form
\begin{align}
\label{eq:bdry_comm2}
\begin{split}
(x_{n+1}^{\frac{2s-1}{2}}\nabla \cdot x_{n+1}^{1-2s} \nabla  (x_{n+1}^{\frac{2s-1}{2}}w), (\xi' \cdot \nu) w)_{L^2(\partial \Omega)}
&= -K|\xi'|^2(x_{n+1}^{\frac{1-2s}{2}}e^{\xi' \cdot x' }u_1, (\xi'\cdot \nu) w)_{L^2(\partial \Omega)} \\
& \quad + 2(\xi' \cdot \nabla' w, (\xi' \cdot \nu) w)_{L^2(\partial \Omega)}.
\end{split}
\end{align}
Here we have used the bulk equation for $w$ which, due to the regularity of $w$, is continuous up to the boundary.

Splitting $\xi' \cdot \nabla' $ into tangential and normal components as in \eqref{eq:tan_norm_split}, the second term can be dealt with similarly as in the argument for \eqref{eq:bdry_est2}: Indeed,
\begin{align*}
 2(\xi' \cdot \nabla' w, (\xi' \cdot \nu) w)_{L^2(\partial \Omega)}
 &=  2|\xi'|^2(e_{\xi'} \cdot \nabla' (x_{n+1}^{\frac{2s-1}{2}}w), (e_{\xi'} \cdot \nu) x_{n+1}^{1-2s} x_{n+1}^{\frac{2s-1}{2}} w)_{L^2(\partial \Omega)}\\
& = 2|\xi'|^2 (  \p_{\nu}(x_{n+1}^{\frac{2s-1}{2}}w), (e_{\xi'} \cdot \nu)^2 x_{n+1}^{1-2s} (x_{n+1}^{\frac{2s-1}{2}} w))_{L^2(\partial \Omega)}\\
&\quad  - |\xi'|^2(  x_{n+1}^{\frac{2s-1}{2}} w, (x_{n+1}^{\frac{2s-1}{2}}w)  [\di_{\partial \Omega}(\beta (e_{\xi'} \cdot \nu) x_{n+1}^{1-2s})]  )_{L^2(\partial \Omega)}\\
& = 2|\xi'|^3 (  (x_{n+1}^{\frac{2s-1}{2}}w), (e_{\xi'} \cdot \nu)^3 x_{n+1}^{1-2s} (x_{n+1}^{\frac{2s-1}{2}} w))_{L^2(\partial \Omega)}\\
&\quad  - |\xi'|^2(  x_{n+1}^{\frac{2s-1}{2}} w, (x_{n+1}^{\frac{2s-1}{2}}w)  [\di_{\partial \Omega}(\beta (e_{\xi'} \cdot \nu) x_{n+1}^{1-2s})]  )_{L^2(\partial \Omega)}.
\end{align*}
Using the regularity of $\partial \Omega$, both terms can be estimates by a contribution of the form \eqref{eq:bdry_dominated_term}. 

For the first term on the right hand side of \eqref{eq:bdry_comm2}, we note that 
\begin{align*}
 -K|\xi'|^2(x_{n+1}^{\frac{1-2s}{2}}e^{\xi' \cdot x' }u_1, (\xi'\cdot \nu) w)_{L^2(\partial \Omega)}
 =  -K|\xi'|^2(e^{\xi' \cdot x' }u_1, x_{n+1}^{1-2s}(\xi'\cdot \nu) x_{n+1}^{\frac{2s-1}{2}} w)_{L^2(\partial \Omega)}.
\end{align*}
Since $x_{n+1}^{1-2s}(\xi'\cdot \nu) \rightarrow 0$ for $x \in \partial \Omega$ with $x_{n+1}\rightarrow 0$ and since $x_{n+1}^{\frac{2s-1}{2}} w = e^{\xi' \cdot x'} u_2$,
it is only active at the boundary $\partial \Omega \setminus \overline{\Sigma_1}$. Rewriting $w = e^{\xi' \cdot x'} x_{n+1}^{\frac{1-2s}{2}} u_2 = e^{\xi' \cdot x'} x_{n+1}^{\frac{1-2s}{2}} (u - u_1) $ and using the boundary conditions for $u_1$, the first term in \eqref{eq:bdry_comm2} hence turns into 
\begin{align*}
K|\xi'|^2 (e^{\xi' \cdot x'} u_1, x_{n+1}^{1-2s}(\xi' \cdot \nu) e^{\xi' \cdot x'  } u_1 )_{L^2(\partial \Omega {\setminus \overline{\Sigma_1}})}.
\end{align*}
Due to the boundary regularity, we observe that this contribution is bounded by
\begin{align}
\label{eq:bdry_dom2}
C K|\xi'|^3 \|e^{\xi' \cdot x'} u_1\|_{L^2(\partial \Omega {\setminus \overline{\Sigma_1}})}^2,
\end{align}
where $C=C(\Omega)>1$.
Using the boundary trace estimate from Lemma \ref{lem:boundary_trace} (with $\mu = |\xi'|^{\frac{1}{2}}$) we may control this by bulk contributions:
\begin{align}
\label{eq:bdry_dom2_est}
\begin{split}
K|\xi'|^3 \|e^{\xi' \cdot x'} u_1\|_{L^2(\partial \Omega {\setminus \overline{\Sigma_1}})}^2 
&\leq C K (|\xi'|^{2} \|\nabla (e^{\xi'\cdot x'} u_1)\|_{L^2(\Omega)}^2 + |\xi'|^{4} \| e^{\xi'\cdot x'} u_1\|_{L^2(\Omega)}^2)\\
& \leq C K (|\xi'|^{2} \|e^{\xi'\cdot x'} \nabla  u_1\|_{L^2(\Omega)}^2 + |\xi'|^{4} \| e^{\xi'\cdot x'} u_1\|_{L^2(\Omega)}^2)\\
& \leq C K (|\xi'|^{2} \|e^{\xi'\cdot x'}x_{n+1}^{\frac{1-2s}{2}} \nabla  u_1\|_{L^2(\Omega)}^2 + |\xi'|^{4} \| e^{\xi'\cdot x'} x_{n+1}^{\frac{1-2s}{2}} u_1\|_{L^2(\Omega)}^2).
\end{split}
\end{align}

\emph{Step 2c: Antisymmetric and symmetric terms.}
Next, we invoke the compact support of $u$ to deduce a lower bound for $A_{\phi}$: Rewriting $w= e^{\xi' \cdot x'} u_2 = e^{\xi' \cdot x'}(u-u_1)$, then the compact support of $u$ in the tangential slices yields by virtue of Poincar\'e's inequality that
\begin{align}
\label{eq:antisym}
\begin{split}
\|A_{\phi} w\|_{L^2(\Omega)}
&\geq  \|\xi' \cdot \nabla'(x_{n+1}^{\frac{1-2s}{2}}e^{x'\cdot \xi'} u)\|_{L^2(\Omega)} -  |\xi'|\|x_{n+1}^{\frac{1-2s}{2}}\nabla (e^{x'\cdot \xi'} u_1)\|_{L^2(\Omega)}\\
& \geq  C^{-1}|\xi'|\|x_{n+1}^{\frac{1-2s}{2}} e^{x'\cdot \xi'}u\|_{L^2(\Omega)} -  |\xi'|\|x_{n+1}^{\frac{1-2s}{2}}\nabla (e^{x'\cdot \xi'} u_1)\|_{L^2(\Omega)}\\
&  \geq  C^{-1}|\xi'|\|x_{n+1}^{\frac{1-2s}{2}} e^{x'\cdot \xi'}u_2\|_{L^2(\Omega)} -  |\xi'|\|x_{n+1}^{\frac{1-2s}{2}}\nabla (e^{x'\cdot \xi'} u_1)\|_{L^2(\Omega)}\\
& \quad -  |\xi'|\|x_{n+1}^{\frac{1-2s}{2}} (e^{x'\cdot \xi'} u_1)\|_{L^2(\Omega)} \\
& =  C^{-1} |\xi'|\|w\|_{L^2(\Omega)} -  |\xi'|\|x_{n+1}^{\frac{1-2s}{2}}\nabla (e^{x'\cdot \xi'} u_1)\|_{L^2(\Omega)} -   |\xi'|\|x_{n+1}^{\frac{1-2s}{2}} (e^{x'\cdot \xi'} u_1)\|_{L^2(\Omega)}\\
& \geq  C^{-1} |\xi'|\|w\|_{L^2(\Omega)} -  |\xi'| \|e^{x'\cdot \xi'} x_{n+1}^{\frac{1-2s}{2}}\nabla  u_1\|_{L^2(\Omega)} -   |\xi'|^2\|x_{n+1}^{\frac{1-2s}{2}} (e^{x'\cdot \xi'} u_1)\|_{L^2(\Omega)}.
\end{split}
\end{align}

Testing the symmetric part of the operator with $w$ itself, we further obtain that
\begin{align}
\label{eq:symm}
\begin{split}
\|x_{n+1}^{\frac{1-2s}{2}}\nabla (x_{n+1}^{\frac{2s-1}{2}} w)\|_{L^2(\Omega)} 
& \leq \|S_{\phi} w\|_{L^2(\Omega)} \|w\|_{L^2(\Omega)}\\
&  \quad + (\lim\limits_{x \rightarrow \partial \Omega} x_{n+1}^{1-2s} \p_{\nu} (x_{n+1}^{\frac{2s-1}{2}} w), x_{n+1}^{\frac{2s-1}{2}} w)_{L^2(\partial \Omega)}\\
& \leq \|S_{\phi} w\|_{L^2(\Omega)} \|w\|_{L^2(\Omega)}\\
& \quad  + (x_{n+1}^{1-2s} (\xi' \cdot \nu)(x_{n+1}^{\frac{2s-1}{2}} w), x_{n+1}^{\frac{2s-1}{2}} w)_{L^2(\partial \Omega)}\\
& = \|S_{\phi} w\|_{L^2(\Omega)} \|w\|_{L^2(\Omega)}\\
& \quad  + (x_{n+1}^{1-2s} (\xi' \cdot \nu)(x_{n+1}^{\frac{2s-1}{2}} w), x_{n+1}^{\frac{2s-1}{2}} w)_{L^2(\partial \Omega \setminus \overline{\Sigma_1})}.
\end{split}
\end{align}
We may now estimate the boundary contribution arising in these estimates as above (see \eqref{eq:bdry1}, \eqref{eq:bdry2}), as it is controlled by \eqref{eq:bdry_dominated_term}.

\medskip
\emph{Step 2d: Conclusion of the estimate for $u_2$.}

Thus, for $|\xi'|\geq 1$, combining the estimates \eqref{eq:Carl_exp}-\eqref{eq:symm}, in total, the Carleman estimate turns into 
\begin{align}
\label{eq:Carl_aux1}
\begin{split}
 &|\xi'| \|w\|_{L^2(\Omega)} + \|x_{n+1}^{\frac{1-2s}{2}}\nabla (x_{n+1}^{\frac{2s-1}{2}} w)\|_{L^2(\Omega)} \\
&\leq C( \|L_{s,\phi} w\|_{L^2(\Omega)} +  |\xi'| \|x_{n+1}^{\frac{1-2s}{2}}\nabla (e^{x'\cdot \xi'} u_1)\|_{L^2(\Omega)} +  |\xi'|^{2}\|x_{n+1}^{\frac{1-2s}{2}} (e^{x'\cdot \xi'} u_1)\|_{L^2(\Omega)})  .
\end{split}
\end{align}

Next we seek to complement \eqref{eq:Carl_aux1} with a boundary contribution on the left hand side of the Carleman inequality. To this end, we use the boundary-bulk-interpolation estimate from Lemma \ref{lem:boundary_trace_s}. This implies that
\begin{align*}
|\xi'|^{s}\|x_{n+1}^{\frac{2s-1}{2}} w\|_{L^2(\Sigma_1)} \leq C |\xi'| \|x_{n+1}^{\frac{1-2s}{2}} (x_{n+1}^{\frac{2s-1}{2}}w)\|_{L^2(\Omega)} + \|x_{n+1}^{\frac{1-2s}{2}}\nabla (x_{n+1}^{\frac{2s-1}{2}} w)\|_{L^2(\Omega)}.
\end{align*}
As a consequence, the estimate \eqref{eq:Carl_aux1} becomes
\begin{align}
\label{eq:Carl_aux2}
\begin{split}
&|\xi'|^{s}\|x_{n+1}^{\frac{2s-1}{2}} w\|_{L^2(\Sigma_1)} + |\xi'| \|w\|_{L^2(\Omega)} + \|x_{n+1}^{\frac{1-2s}{2}}\nabla (x_{n+1}^{\frac{2s-1}{2}} w)\|_{L^2(\Omega)}\\
&\leq C (\|L_{s,\phi} w\|_{L^2(\Omega)} +  |\xi'|\|e^{x'\cdot \xi'} x_{n+1}^{\frac{1-2s}{2}}\nabla  u_1\|_{L^2(\Omega)} +  |\xi'|^2 \|x_{n+1}^{\frac{1-2s}{2}} (e^{x'\cdot \xi'} u_1)\|_{L^2(\Omega)}).
\end{split}
\end{align}
Returning to $u_2$ then yields the bound
\begin{align}
\label{eq:Carl_aux3}
\begin{split}
&|\xi'|^{s}\|e^{x'\cdot \xi'} u_2\|_{L^2(\Sigma_1)} + |\xi'| \|e^{x'\cdot \xi'} x_{n+1}^{\frac{1-2s}{2}} u_2\|_{L^2(\Omega)} + \|e^{x'\cdot \xi'} x_{n+1}^{\frac{1-2s}{2}}\nabla u_2 \|_{L^2(\R^{n+1}_+)}\\
&  \leq C(\|L_{s,\phi} (e^{x'\cdot \xi'} x_{n+1}^{\frac{1-2s}{2}} u_2 )\|_{L^2(\R^{n+1}_+)} 
 +  |\xi'|\|e^{x'\cdot \xi'}  x_{n+1}^{\frac{1-2s}{2}}\nabla u_1\|_{L^2(\Omega)} +  |\xi'|^2\|x_{n+1}^{\frac{1-2s}{2}} (e^{x'\cdot \xi'} u_1)\|_{L^2(\Omega)} )\\
& = C (K |\xi'|^2 \|e^{x'\cdot \xi'} x_{n+1}^{\frac{1-2s}{2}} u_1\|_{L^2(\Omega)}
  +  |\xi'|\|x_{n+1}^{\frac{1-2s}{2}}\nabla (e^{x'\cdot \xi'} u_1)\|_{L^2(\Omega)} +  |\xi'|^2\|x_{n+1}^{\frac{1-2s}{2}} (e^{x'\cdot \xi'} u_1)\|_{L^2(\Omega)})\\
  & \leq  C (K |\xi'|^2 \|e^{x'\cdot \xi'} x_{n+1}^{\frac{1-2s}{2}} u_1\|_{L^2(\Omega)}
  +  |\xi'|\|e^{x'\cdot \xi'}  x_{n+1}^{\frac{1-2s}{2}}\nabla u_1 \|_{L^2(\Omega)}) .
\end{split}
\end{align}

Now, if $u_1 \in H^{1}(\Omega, x_{n+1}^{1-2s})$ is not $C^{0,\alpha}(\Omega)$ for some $\alpha \in (0,1)$, we simply replace $u_1$ by $u_{1,\epsilon}:=(u_1 \chi_{\Omega})\ast \varphi_{\epsilon} \in C^{0,\alpha}(\Omega)$ (where $\chi_{\Omega}$ is the characteristic function of $\Omega$ and $\varphi_{\epsilon}$ is a standard mollifier) and consider the equation \eqref{eq:u_2} with $u_1$ replaced by $u_{1,\epsilon}$. We denote the corresponding solution by $u_{2,\epsilon}$. This allows us to derive all estimates including \eqref{eq:Carl_aux3} with $u_1, u_2$ replaced by $u_{1,\epsilon}$ and $u_{2,\epsilon}$. Combining the estimate \eqref{eq:Carl_aux3}, weak lower semi-continuity and the $H^1(\Omega, x_{n+1}^{1-2s})$ regularity of $u_1$ then allows us to pass to the limit $\epsilon \rightarrow 0$. This then also yields \eqref{eq:Carl_aux3} with the functions $u_1, u_2$ (instead of $u_{1,\epsilon}, u_{2,\epsilon}$).
\medskip

\emph{Step 3: Conclusion.}
Combining the estimates from \eqref{eq:apriori11} and \eqref{eq:Carl_aux3}, by the triangle inequality, we obtain that
\begin{align}
\label{eq:Carl_comb}
\begin{split}
&|\xi'|^{s}\|e^{x'\cdot \xi'} u\|_{L^2(\Sigma_1)} + |\xi'| \|e^{x'\cdot \xi'} x_{n+1}^{\frac{1-2s}{2}} u\|_{L^2(\Omega)} + \|e^{x'\cdot \xi'} x_{n+1}^{\frac{1-2s}{2}}\nabla u \|_{L^2(\Omega)}\\
&\leq  C K |\xi'|^2 \|e^{x'\cdot \xi'} x_{n+1}^{\frac{1-2s}{2}} u_1\|_{L^2(\Omega)} + |\xi'|\|e^{x'\cdot \xi'}  x_{n+1}^{\frac{1-2s}{2}}\nabla u_1\|_{L^2(\Omega)}\\
& \quad + C \|x_{n+1}^{\frac{1-2s}{2}} e^{x' \cdot \xi'} \tilde{F}\|_{L^2(\Omega)} + C |\xi'|^{-1} \|x_{n+1}^{\frac{1-2s}{2}} e^{x' \cdot \xi'} F_0\|_{L^2(\Omega)} + 
\\ & \quad + C_\epsilon |\xi'|^{-s}  \left( \|e^{x'\cdot \xi'} g\|_{L^2(\Sigma_1)} + \|q\|_{L^{\infty}(\Sigma_1)} \|e^{x'\cdot \xi'} u\|_{L^2(\Sigma_1)} \right) \\
& \leq CK|\xi'| \|e^{ \xi' \cdot x '} x_{n+1}^{\frac{1-2s}{2}} \tilde{F} \|_{L^2(\Omega)} +  CK\|e^{ \xi' \cdot x '} x_{n+1}^{\frac{1-2s}{2}} F_0 \|_{L^2(\Omega)}
\\ & \quad+  C_{\epsilon}K |\xi'|^{1-s}( \|e^{x'\cdot \xi'}g\|_{L^2(\Sigma_1)} + \|q\|_{L^{\infty}(\Sigma_1)} \|e^{x'\cdot \xi'} u\|_{L^2(\Sigma_1)}).
\end{split}
\end{align}

Now, if $s> \frac{1}{2}$ and $|\xi'|\gg 1$ is sufficiently large (depending on $\|q\|_{L^{\infty}(\Sigma_1)}$), it is possible to absorb the boundary term involving $q$ on the right hand side into the left hand side of \eqref{eq:Carl_comb}. If $s = \frac{1}{2}$, the absorption is still possible if we assume that $\|q\|_{L^{\infty}(\Sigma_1)}$ is sufficiently small. Under these assumptions, \eqref{eq:Carl_comb} thus turns into the desired estimate \eqref{eq:Carl}.
\end{proof}

\begin{rmk}
\label{rmk:Lopatinskii}
We expect that for $s= \frac{1}{2}$ it might be possible to improve the Carleman estimate by relying on the Lopatinskii condition. For $s\in (\frac{1}{2},1)$ this is less clear. We postpone this to a future project.
\end{rmk}

As a corollary to Proposition \ref{prop:Carl} we note that the estimate \eqref{eq:Carl} remains true if in \eqref{eq:main_eq} we consider the bulk equation
\begin{align*}
\nabla \cdot x_{n+1}^{1-2s} \nabla u + V x_{n+1}^{1-2s} u & = f \mbox{ in } \Omega,
\end{align*}
with $f \in (H^1(\Omega, x_{n+1}^{1-2s}))^{\ast}$.

\begin{cor}
\label{cor:Carl}
Let $s \in [\frac{1}{2},1)$, $\xi' \in \C^n$ such that $\xi' \cdot \xi' = 0$. Assume that the same conditions as in Proposition \ref{prop:Carl} hold for $\Omega$, $q$, $f$ and $g$. Let $V \in L^{\infty}(\Omega)$ and assume that  $u \in H^1_{\partial \Omega \setminus \overline{\Sigma_1},0}(\overline{\Omega}, x_{n+1}^{1-2s})$ with $u = 0$ and $\lim\limits_{x \rightarrow \partial \Omega} x_{n+1}^{1-2s} \p_{\nu} u = 0$ on $\partial \Omega \setminus \overline{\Sigma_1}$ is a weak solution to
\begin{align}
\label{eq:main_eqV}
\begin{split}
\nabla \cdot x_{n+1}^{1-2s} \nabla u + V x_{n+1}^{1-2s} u& = f \mbox{ in } \Omega,\\
\lim\limits_{x_{n+1}\rightarrow 0} x_{n+1}^{1-2s} \p_{n+1} u + q u & = g \mbox{ on } \Sigma_1.
\end{split}
\end{align}
Then, we have
\begin{align}
\label{eq:CarlV}
\begin{split}
&|\xi'|^{s} \|e^{ \xi' \cdot x'} u\|_{L^2(\Sigma_1)} 
+ |\xi'| \|e^{ \xi' \cdot x'} x_{n+1}^{\frac{1-2s}{2}} u\|_{L^2(\Omega )} + \|e^{ \xi' \cdot x' } x_{n+1}^{\frac{1-2s}{2}} \nabla u\|_{L^2(\Omega)}\\
&\leq C (|\xi'| \|e^{ \xi' \cdot x '} x_{n+1}^{\frac{1-2s}{2}} \tilde{F} \|_{L^2(\Omega)} +  \|e^{ \xi' \cdot x '} x_{n+1}^{\frac{1-2s}{2}} F_0 \|_{L^2(\Omega)} + |\xi'|^{1-s}\|e^{ \xi' \cdot x '} g\|_{L^{2}(\Sigma_1)}).
\end{split}
\end{align}
Here the constant $C>0$ depends on $\|q\|_{L^{\infty}(\Sigma_1)}$ and $\|V\|_{L^{\infty}(\Omega)}$, while $F =(F_0, \tilde{F}) \in L^2(\R^{n+1}_+,\R^{n+2})$ is the Riesz representation of $f$, i.e., it is such that
\begin{align*}
f(v) = (v, x_{n+1}^{1-2s} F_0)_{L^2(\Omega)} + (\nabla v, x_{n+1}^{1-2s} \tilde{F})_{L^2(\Omega)} \mbox{ for all } v \in H^{1}(\Omega, x_{n+1}^{1-2s}).
\end{align*}
\end{cor}

\begin{proof}
The proof follows directly by a reduction to the setting of Proposition \ref{prop:Carl}. Indeed, we interpret \eqref{eq:main_eqV} as an equation of the form \eqref{eq:main_eq} with $\tilde{f} = f - x_{n+1}^{1-2s} V u$. If the Riesz representative of $f$ had been given by $F=(F_0, \tilde{F})$, the one for $\tilde{f}$ is now given by $\bar{F}=(F_0-Vu, \tilde{F})$. As a consequence, \eqref{eq:Carl} turns into 
\begin{align*}
&|\xi'|^{s} \|e^{ \xi' \cdot x'} u\|_{L^2(\Sigma_1)} 
+ |\xi'| \|e^{ \xi' \cdot x'} x_{n+1}^{\frac{1-2s}{2}} u\|_{L^2(\Omega )} + \|e^{ \xi' \cdot x' } x_{n+1}^{\frac{1-2s}{2}} \nabla u\|_{L^2(\Omega)}\\
&\leq C (|\xi'| \|e^{ \xi' \cdot x '} x_{n+1}^{\frac{1-2s}{2}} \tilde{F} \|_{L^2(\Omega)} +  \|e^{ \xi' \cdot x '} x_{n+1}^{\frac{1-2s}{2}} (F_0 -V u)\|_{L^2(\Omega)} + |\xi'|^{1-s}\|e^{ \xi' \cdot x '} g\|_{L^{2}(\Sigma_1)}).
\end{align*}
Applying the triangle inequality, we obtain
\begin{align*}
&|\xi'|^{s} \|e^{ \xi' \cdot x'} u\|_{L^2(\Sigma_1)} 
+ |\xi'| \|e^{ \xi' \cdot x'} x_{n+1}^{\frac{1-2s}{2}} u\|_{L^2(\Omega )} + \|e^{ \xi' \cdot x' } x_{n+1}^{\frac{1-2s}{2}} \nabla u\|_{L^2(\Omega)}\\
&\leq C (|\xi'| \|e^{ \xi' \cdot x '} x_{n+1}^{\frac{1-2s}{2}} \tilde{F} \|_{L^2(\Omega)} +  \|e^{ \xi' \cdot x '} x_{n+1}^{\frac{1-2s}{2}} F_0 \|_{L^2(\Omega)} +  \|V\|_{L^{\infty}(\Omega)}\|e^{ \xi' \cdot x '} x_{n+1}^{\frac{1-2s}{2}} u\|_{L^2(\Omega)} \\
& \quad + |\xi'|^{1-s}\|e^{ \xi' \cdot x '} g\|_{L^{2}(\Sigma_1)}).
\end{align*}
Now choosing $|\xi'|>1$ so large that $C \|V\|_{L^{\infty}(\Omega)} \leq \frac{1}{2}|\xi'|$, it is possible to absorb the contribution involving $V$ from the right hand side into the left hand side of the Carleman estimate. This implies the desired bound.
\end{proof}


\section{Construction of CGOs for the Generalized Caffarelli-Silvestre Extension}
\label{sec:CGO}

We shall now use estimate \eqref{eq:Carl} in order to prove the result of Proposition \ref{CGO_construction} and to thus deduce the existence of CGOs (associated with the weak form of the equation \eqref{eq:CGO_model}) by means of a duality argument.

\begin{proof}[Proof of Proposition \ref{CGO_construction}] Fix $k\in \R^{n+1}$ and consider two vectors $\zeta_1,\zeta_2 \in (k^\perp \cap e_{n+1}^\perp)$ such that $|\zeta_1|=|\zeta_2|$ and $\zeta_1\cdot\zeta_2=0$. This is possible by the assumption $n\geq 3$, since then $\dim(k^\perp \cap e_{n+1}^\perp) \geq (n+1)-2 =n-1 \geq 2$. If now we let $\xi' := \zeta_1 + i\zeta_2$, we can observe that the condition $\xi'\cdot\xi'=0$ is satisfied. One also has $\xi'\cdot k' = \xi' \cdot k =0$, the two equalities being respectively consequences of $\xi'\in e_{n+1}^\perp$ and $\xi' \in k^\perp$.
\vspace{2mm} 

Substituting the required solution $u(x)=e^{\xi' \cdot x'}(e^{ik'\cdot x' + i k_{n+1} x_{n+1}^{2s}}+r(x))$ into problem \eqref{eq:weighted_problem_a}, we are left with an equivalent problem for the function $r(x)$:
 
\begin{align} \label{eq:problem_r} \begin{split}
 \tilde{L}^s_{-\xi',V} (e^{ik'\cdot x' + i k_{n+1} x_{n+1}^{2s}} +  r)& = 0 \mbox{ in } \Omega,\\
\lim\limits_{x_{n+1}\rightarrow 0} x_{n+1}^{1-2s} \p_{n+1} (e^{ik'\cdot x' + i k_{n+1} x_{n+1}^{2s}}+r)+ q (e^{ik'\cdot x' + i k_{n+1} x_{n+1}^{2s}}+r) & = 0 \mbox{ on } \Sigma_1.
\end{split}\end{align}
Here $\tilde{L}_{-\xi', V}^s = \nabla \cdot x_{n+1}^{1-2s} \nabla + x_{n+1}^{1-2s} V + 2 x_{n+1}^{1-2s} \xi' \cdot \nabla'$.

We shall first study the following norm:
\begin{align} \label{eq:CGO_1} \begin{split}
 \| \tilde{L}^s_{-\xi',V} &e^{ik'\cdot x' + i k_{n+1} x_{n+1}^{2s}} \|_{L^2(\Omega, x_{n+1}^{2s-1})}  
 = \|  (\tilde{L}^s_{-\xi'} + x_{n+1}^{1-2s} V)(e^{ik'\cdot x' + i k_{n+1} x_{n+1}^{2s}}) \|_{L^2(\Omega, x_{n+1}^{2s-1})}
 \\ 
 & \leq  \| x_{n+1}^{\frac{1-2s}{2}}  V e^{ik'\cdot x' + i k_{n+1} x_{n+1}^{2s}} \|_{L^2(\Omega)} +  \|  \tilde L^s_{-\xi'}(e^{ik'\cdot x' + i k_{n+1} x_{n+1}^{2s}}) \|_{L^2(\Omega, x_{n+1}^{2s-1})}
 \\ 
 & \leq \|V\|_{L^\infty(\Omega)} \|  x_{n+1}^{1/2-s} \|_{L^2(\Omega)} +  \| (\nabla \cdot x_{n+1}^{1-2s} \nabla + 2x_{n+1}^{1-2s}\xi' \cdot \nabla')e^{ik'\cdot x' + i k_{n+1} x_{n+1}^{2s}} \|_{L^2(\Omega, x_{n+1}^{2s-1})}
 \\
 & = \|V\|_{L^\infty(\Omega)} \|  x_{n+1}^{1/2-s} \|_{L^2(\Omega)}+ \| (x_{n+1}^{1/2-s}|k'|^2 + (2s)^2 x_{n+1}^{3s-3/2}k_{n+1}^2)e^{ik'\cdot x' + i k_{n+1} x_{n+1}^{2s}} \|_{L^2(\Omega)}
 \\ 
 & \leq (\|V\|_{L^\infty(\Omega)} + |k'|^2)\|  x_{n+1}^{1/2-s} \|_{L^2(\Omega)} + 4s^2 k_{n+1}^2 \|  x_{n+1}^{3s-3/2} \|_{L^2(\Omega)}
 \\ 
 & \leq C_{\Omega,V,k,s} < \infty.
\end{split}\end{align}

\noindent In the last step we have used our assumption that $s\geq 1/2$ and that $\xi' \cdot k = 0$. If we define 
$$f(x):= -\tilde{L}^s_{-\xi',V} (e^{ik'\cdot x' + i k_{n+1} x_{n+1}^{2s}}),$$ 
then by \eqref{eq:CGO_1} we have proved that $\| f \|_{L^2(\Omega, x_{n+1}^{2s-1})} = O(1)$ with respect to $|\xi'|\rightarrow \infty$. Next, we compute that for almost every $x'\in\Sigma_1$
\begin{align*}
& \left|\lim\limits_{x_{n+1}\rightarrow 0} (x_{n+1}^{1-2s} \p_{n+1} e^{ik'\cdot x' + i k_{n+1} x_{n+1}^{2s}}+ q(x') e^{ik'\cdot x' + i k_{n+1} x_{n+1}^{2s}})\right| \\
& = |e^{ik'\cdot x' }||q (x')+ 2s i \,k_{n+1}| \leq C_{q,k} < \infty\;.
\end{align*}

\noindent Thus, we define 
$$g(x'):= -e^{ik'\cdot x'}(2si\,k_{n+1}+q(x')) ,$$ 
and obtain that $\|g\|_{L^2(\Sigma_1)} \leq C_{q,k} |\Sigma_1|^{1/2} = O(1)$ with respect to $|\xi'|\rightarrow \infty$. 
\vspace{2mm}

In light of the above computations, we can rewrite \eqref{eq:problem_r} as an inhomogeneous problem for $ r$:
\begin{align} 
\label{eq:problem_r2} 
\begin{split}
 \tilde{L}^s_{-\xi',V}r & = f \mbox{ in } \Omega,\\
\lim\limits_{x_{n+1}\rightarrow 0} x_{n+1}^{1-2s} \p_{n+1}  r+ q  r & = g \mbox{ on } \Sigma_1.
\end{split}
\end{align}

We will construct a solution to the problem \eqref{eq:problem_r2} with the claimed decay properties by using a duality argument and the Carleman estimate \eqref{eq:CarlV}.

\noindent 

To this end, we first recall the function space $\mathcal{C}$ from \eqref{eq:C} in Section \ref{sec:CGO_test}
which is a subvector space of $L^2(\Omega, x_{n+1}^{2s-1})$ and has the property that 
\begin{align*}
\lim\limits_{x_{n+1}\rightarrow 0} (x_{n+1}^{1-2s} \p_{n+1} w + q w)  \in L^2(\Sigma_1) \mbox{ and } \tilde{L}^s_{\xi',V} w \in L^2(\Omega, x_{n+1}^{2s-1}) \subset (H^{1}(\Omega,x_{n+1}^{1-2s}))^{\ast}
\end{align*}
and $\supp( \tilde{L}^s_{\xi',V} w ) \subset \Omega \cup (\Omega \cap \{x_{n+1}=0\})$.

We define the operator $\mathcal{B}_s: \mathcal{C} \rightarrow L^2(\Sigma_1), \ w \mapsto \lim\limits_{x_{n+1}\rightarrow 0} x_{n+1}^{1-2s} \p_{n+1} w + q w $.

We now seek to study a suitable functional which builds on the injectivity of the following mapping:
For $u \in \mathcal{C}$ consider 
\begin{align}
\label{eq:map}
(\tilde L^s_{\xi',V} u , \mathcal{B}_{s} u) \mapsto u.
\end{align}

In order to derive the injectivity of the map in \eqref{eq:map}, we invoke the Carleman estimates from Proposition \ref{prop:Carl} and Corollary \ref{cor:Carl}. To this end, we rephrase the Carleman estimate from Proposition \ref{prop:Carl} and Corollary \ref{cor:Carl} in terms of an estimate for the operators $\tilde{L}^s_{\xi',V}$ and $\mathcal{B}_s$. For $u \in \mathcal{C}$ we consider the Carleman estimate of Corollary \ref{cor:Carl} for the function $\tilde{u}:=e^{-x'\cdot \xi'} u$. This function clearly satisfies the boundary conditions stated in Corollary \ref{cor:Carl} on $\partial \Omega \setminus \overline{\Sigma_1}$. Now, if $u$ is a solution to the equation
\begin{align*}
\tilde{L}^{s}_{\xi',V} u & = f \mbox{ in } \Omega,\\
\mathcal{B}_s(u) & = g \mbox{ on } \Sigma_1,
\end{align*}
for some $f\in (H^1(\Omega,x_{n+1}^{1-2s}))^{\ast}$ and $g\in L^2(\Sigma_1)$, then the function $\tilde{u}$ satisfies an equation of the form \eqref{eq:main_eqV} with a bulk inhomogeneity $\tilde{f} = e^{-\xi' \cdot x'} f$ and a boundary inhomogeneity $\tilde{g}= e^{-\xi' \cdot x'} g$. If $(F_0, \bar{F})$ was the Riesz representative of $f$ in $(H^1(\Omega, x_{n+1}^{1-2s}))^{\ast}$, then the Riesz representative of $\tilde{f}$ is given by $(\tilde{F}_0, \tilde{\bar{F}}):=(e^{-x'\cdot \xi'}F_0 - e^{-x' \cdot \xi'} \bar{F}_{n+1}, e^{-x'\cdot \xi'} \bar{F} )$.
The Carleman estimate from Corollary \ref{cor:Carl} for $\tilde{u}$ is thus applicable and yields
\begin{align*}
&|\xi'|^{s} \|e^{ \xi' \cdot x'} \tilde{u}\|_{L^2(\Sigma_1)} 
+ |\xi'| \|e^{ \xi' \cdot x'} x_{n+1}^{\frac{1-2s}{2}} \tilde{u}\|_{L^2(\Omega)} + \|e^{ \xi' \cdot x' } x_{n+1}^{\frac{1-2s}{2}} \nabla {\tilde{u}}\|_{L^2(\Omega)}\\
&\leq C (|\xi'|\|e^{ \xi' \cdot x '} x_{n+1}^{\frac{1-2s}{2}} \tilde{\bar{F}} \|_{L^2(\Omega)} + \|e^{ \xi' \cdot x '} x_{n+1}^{\frac{1-2s}{2}} \tilde{F}_0 \|_{L^2(\Omega)} + |\xi'|^{1-s}\|e^{x'\cdot \xi'}\tilde{g}\|_{L^2(\Sigma_1)}).
\end{align*}
Using the triangle inequality, this can now be rewritten in terms of $u$, the operators $\tilde{L}^s_{\xi',V}$ and $\mathcal{B}_s$ and then becomes
\begin{align}
\label{eq:Carl22}
\begin{split}
&|\xi'|^{s} \|u\|_{L^2(\Sigma_1)} 
+ |\xi'| \|x_{n+1}^{\frac{1-2s}{2}} u\|_{L^2(\Omega )} + \| x_{n+1}^{\frac{1-2s}{2}} \nabla {u}\|_{L^2(\Omega)}\\
&\leq C (|\xi'|\|x_{n+1}^{\frac{1-2s}{2}} \bar{F} \|_{L^2(\Omega+)} + \| x_{n+1}^{\frac{1-2s}{2}} F_0 \|_{L^2(\Omega)} + |\xi'|^{1-s}\|g\|_{L^2(\Sigma_1)})\\
&\leq  C (|\xi'|\|\tilde{L}^s_{\xi',V} u\|_{(H^1(\Omega,x_{n+1}^{1-2s})^{\ast})} + |\xi'|^{1-s}\|\mathcal{B}_s(u)\|_{L^2(\Sigma_1)}).
\end{split}
\end{align}
As a result, we infer that the map \eqref{eq:map} is injective.

Building on this observation, we obtain that the linear functional 
 $$ T: \tilde{L}^s_{\xi',V} (\mathcal{C}) \times \mathcal{B}_s(\mathcal{C}) \rightarrow \R, \ (\tilde L^s_{\xi',V} u , \mathcal{B}_{s} u)  \mapsto (u,f)_{L^2(\Omega)} +  (u,g)_{L^2(\Sigma_1)} $$
 \noindent is well defined.

Moreover, using \eqref{eq:Carl22}, the bound
 \begin{align*} 
 \begin{split}
&| (u,f)_{L^2(\Omega)} +  (u,g)_{L^2(\Sigma_1)}  |  \leq \| u \|_{L^2(\Omega,x_{n+1}^{1-2s})} \| f \|_{L^2(\Omega,x_{n+1}^{2s-1})} + \| u \|_{L^2(\Sigma_1)}  \| g \|_{L^2(\Sigma_1)} \\ 
& \leq C_{\Omega,V,k,s}  \| u \|_{L^2(\Omega,x_{n+1}^{1-2s})}  + C_{q,k,\Sigma_1} \| u  \|_{L^2(\Sigma_1)} \\ & \leq ( C_{\Omega,V,k,s}|\xi'|^{-1} +  C_{q,k,\Sigma_1}|\xi'|^{-s})(\||\xi'| \tilde{F} \|_{L^2(\Omega,x_{n+1}^{1-2s})} + \|F_0\|_{L^2(\Omega, x_{n+1}^{1-2s})}  + \||\xi'|^{1-s} \mathcal{B}_s(u)\|_{L^2(\Sigma_1)} ) \\  
& \leq c (|\xi'|^{-1} + |\xi'|^{-s}) (\| \tilde{L}^s_{\xi',V} u \|_{(H^1_{sc}(\Omega,x_{n+1}^{1-2s}))^{\ast}}+ \|\mathcal{B}_s u\|_{L^2_{sc}(\Sigma_1)}).
\end{split}\end{align*}
\noindent holds for a constant $c= c_{\Omega,\Sigma_1,k,V,q}$. Here $\tilde{L}^s_{\xi',V} u = \nabla\cdot \tilde{F} + F_0$ in the sense of distributions. The subscript denotes the use of semiclassical norms with $|\xi'|^{-1}$ as a small parameter, i.e.
\begin{align*}
\| \tilde{L}^s_{\xi',V} u \|_{(H^1_{sc}(\Omega,x_{n+1}^{1-2s}))^{\ast}} & := \||\xi'| \tilde{F} \|_{L^2(\Omega,x_{n+1}^{1-2s})} + \|F_0\|_{L^2(\Omega, x_{n+1}^{1-2s})},\\
\|\mathcal{B}_s u\|_{L^2_{sc}(\Sigma_1)}&:= \||\xi'|^{1-s} \mathcal{B}_s u\|_{L^2(\Sigma_1)} .
\end{align*}
As a consequence, as a functional on a subset of $(H^{1}_{sc}(\Omega, x_{n+1}^{1-2s}))^{\ast} \times L^2_{sc}(\Sigma_1)$, we have  $\|T\| = O(|\xi'|^{-s})$ for $|\xi'|\rightarrow \infty$. Since for $s \in [\frac{1}{2},1)$ the vector space
$\tilde{L}^s_{\xi', V}(\mathcal{C}) \times \mathcal{B}_s(\mathcal{C})$ is a subvector space of $(H^1_{sc}(\Omega, x_{n+1}^{1-2s}))^{\ast} \times L^2_{sc}(\Sigma_1)$,
by the Hahn-Banach theorem, the functional $T$ can be extended to act on all of $(H^1_{sc}(\Omega, x_{n+1}^{1-2s}))^{\ast}\times L^2_{sc}(\Sigma_1)$ while maintaining the same norm. 

Making use of the Riesz representation theorem, we find some $\tilde r_1 \in (H^1_{sc}(\Omega,x_{n+1}^{1-2s}))^*$ and $\tilde r_2 \in L^2_{sc}(\Sigma_1)$ such that for every choice of $v=(v_1,v_2)\in (H^1_{sc}(\Omega,x_{n+1}^{1-2s}))^{\ast}\times L^2_{sc}(\Sigma_1)$ it holds that
\begin{align*}
T(v_1,v_2) = (v_1,\tilde r_1)_{(H^1_{sc}(\Omega,x_{n+1}^{1-2s}))^*} &+ (v_2,\tilde r_2)_{L^2_{sc}(\Sigma_1)}\;,\\
\|\tilde r_1\|_{(H^{1}_{sc}(\Omega,x_{n+1}^{1-2s}))^*} + \|\tilde r_2\|_{L^2_{sc}(\Sigma_1)} =& \|T\|=O(|\xi'|^{-s}) .
\end{align*}

However, if we let $r_1$ be the Riesz representative of $\tilde r_1$ in $H^1_{sc}(\Omega,x_{n+1}^{1-2s})$ and define $r_2:= |\xi'|^{2-2s}\tilde r_2$, we can compute 
\begin{align*}
T(v_1,v_2) = (v_1,\tilde r_1)_{(H^1_{sc}(\Omega,x_{n+1}^{1-2s}))^*} + (v_2,|\xi'|^{2-2s}\tilde r_2)&_{L^2(\Sigma_1)} = \langle v_1,r_1 \rangle + (v_2,r_2)_{L^2(\Sigma_1)}\;, \\ 
|\xi'|^{s-1} \|r_2\|_{L^2(\Sigma_1)} = |\xi'|^{s-1} \||\xi'|^{2-2s}\tilde r_2\|_{L^2(\Sigma_1)}& =  \||\xi'|^{1-s}\tilde r_2\|_{L^2(\Sigma_1)} = \|\tilde r_2\|_{L^2_{sc}(\Sigma_1)}\;, \\ 
\|r_1\|_{H^{1}_{sc}(\Omega,x_{n+1}^{1-2s})} =& \|\tilde r_1\|_{(H^{1}_{sc}(\Omega,x_{n+1}^{1-2s}))^*}\;,
\end{align*}
where $\langle \cdot , \cdot \rangle$ denotes the $(H^1_{sc}(\Omega, x_{n+1}^{1-2s}))^{\ast}$, $H^1_{sc}(\Omega, x_{n+1}^{1-2s})$ duality pairing.
This eventually gives
\begin{align}
\label{eq:CGO_estimates}
\begin{split}
T(v_1,v_2) &= \langle v_1,r_1 \rangle + (v_2,r_2)_{L^2(\Sigma_1)}\,, \\
\|r_1\|_{L^2(\Omega, x_{n+1}^{1-2s})} &+ |\xi'|^{-1}\|\nabla r_1\|_{L^2(\Omega, x_{n+1}^{1-2s})}
+ |\xi'|^{s-1}\|r_2\|_{L^2(\Sigma_1)} = \\ & = \|r_1\|_{H^{1}_{sc}(\Omega,x_{n+1}^{1-2s})} + |\xi'|^{s-1}\|r_2\|_{L^2(\Sigma_1)}   \\ & =  \|\tilde r_1\|_{(H^{1}_{sc}(\Omega,x_{n+1}^{1-2s}))^*} + \|\tilde r_2\|_{L^2_{sc}(\Sigma_1)} = O(|\xi'|^{-s})\;.
\end{split}
\end{align}  
Using that $ L^2_{sc}(\Omega, x_{n+1}^{2s-1}) \subset (H^1_{sc}(\Omega, x_{n+1}^{1-2s}))^{\ast}$ with the identification that the functional $\ell_{v_1}$ associated with $v_1 \in L^2_{sc}(\Omega, x_{n+1}^{2s-1})$ is given by 
\begin{align*}
\ell_{v_1}(f):=(v_1, f)_{L^2(\Omega)} \mbox{ for } f\in L^2_{sc}(\Omega,x_{n+1}^{1-2s}),
\end{align*}
we have that for $v_1 \in L^2_{sc}(\Omega, x_{n+1}^{2s-1})$
\begin{align}
\label{eq:id}
\langle v_1, r_1 \rangle 
:= \langle \ell_{v_1}, r_1 \rangle
= (v_1, r_1)_{L^2(\Omega)}.
\end{align}

Integrating by parts, we next deduce the equations satisfied by $r_1, r_2$. Formally this follows by integrating the equations by parts twice and then inserting suitable test functions. Since a priori no weighted second derivatives of $r_1, r_2$ are given, we need to argue more carefully. To this end, recalling \eqref{eq:id}, we compute for $u \in \mathcal{C}$ with $u=\p_\nu u=0$ on $\Sigma_2$
\begin{align} \label{eq:CGO_3}
\begin{split}
(u,f)&_{L^2(\Omega)} +   (u,g)_{L^2(\Sigma_1)}  = T(\tilde{L}^s_{\xi',V} u, \mathcal{B}_s(u)) \\
 & =  (\tilde{L}^s_{\xi'} u, r_1)_{L^2(\Omega)} +  (x_{n+1}^{1-2s} V u, r_1)_{L^2(\Omega)} + (\mathcal{B}_s(u), r_2)_{L^2(\Sigma_1)}\\
 & = (x_{n+1}^{1-2s} \nabla u,  \nabla r_1)_{L^2(\Omega)}
 - 2(x_{n+1}^{1-2s} \xi'\cdot \nabla' u, r_1)_{L^2(\Omega)} +  (x_{n+1}^{1-2s} V u, r_1)_{L^2(\Omega)}\\
 & \quad + (\mathcal{B}_s(u), r_2-r_1)_{L^2(\Sigma_1)} + (qu, r_1)_{L^2(\Sigma_1)}.
\end{split}
\end{align}

As a consequence, considering $u \in C_c^{\infty}(\Omega)$ we infer that the function $r_1$ is a weak solution to the bulk equation
\begin{align*}
\tilde{L}^s_{\xi',V} r_1 = f \mbox{ in } \Omega
\end{align*}
and 
\begin{align}
\label{eq:bulk_one}
\begin{split}
(u,f)&_{L^2(\Omega)}
= (x_{n+1}^{1-2s} \nabla u,  \nabla r_1)_{L^2(\Omega)}
 - 2(x_{n+1}^{1-2s} \xi'\cdot \nabla' u, r_1)_{L^2(\Omega)} +  (x_{n+1}^{1-2s} V u, r_1)_{L^2(\Omega)}
\end{split}
\end{align}
for all $u \in C_c^{\infty}(\Omega)$.
Next, by an approximation result which uses the fact that $s\geq \frac{1}{2}$, we obtain that the identity \eqref{eq:bulk_one},
which a priori only holds for $u\in C_c^{\infty}(\Omega)$, also remains true for $u \in x_{n+1}^{2s} C_c^{\infty}(\overline{\Omega})$. Combining this with \eqref{eq:CGO_3}, thus implies in turn that for $u\in x_{n+1}^{2s} C_c^{\infty}(\overline{\Omega})$ we have the following boundary equation
\begin{align} 
\label{eq:CGO_boundary_one}
\begin{split}
(u,g)_{L^2(\Sigma_1)} 
 & =   (\mathcal{B}_s(u), r_2-r_1)_{L^2(\Sigma_1)} + (qu, r_1)_{L^2(\Sigma_1)}.
\end{split}
\end{align}

Using this observation, we now consider a suitable test function to deduce further information from \eqref{eq:CGO_boundary_one}:
Let $h\in C^\infty_c (\Sigma_1)$ and consider an open set $\overline{\Sigma}$ such that supp$(h)\subset \overline\Sigma\subset \Sigma_1$. Let $\epsilon > 0$ be so small that $\overline\Sigma \times (0,\epsilon) \subset\subset \Omega $ and consider $\psi \in C^\infty_c(\bar\Omega)$ such that $\psi(x) = 1$ if $x\in$ supp$(h) \times [0, \epsilon/2)$ and $\psi(x)=0$ if $x\not\in\overline\Sigma \times [0,\epsilon]$. Finally, let $u(x)=x_{n+1}^{2s}\psi(x)h(x')$.
\vspace{2mm}

Observe that since supp$(u) \subset \Sigma_1 \times (0,\epsilon) \subset\subset \Omega$ we have $u=\p_\nu u = 0$ on $\Sigma_2$. Moreover, since $\psi h \in C^\infty_c (\overline\Omega)$, we have $u\in x_{n+1}^{2s}C^\infty_c(\overline\Omega)$. Thus, $u$ is a valid test function. We can compute

$$ \mathcal B_s{u} = \lim_{x_{n+1}\rightarrow 0} x_{n+1}^{1-2s}\p_{n+1}u+qu =  h(x') \lim_{x_{n+1}\rightarrow 0} x_{n+1}^{1-2s}\p_{n+1}(\psi(x) x_{n+1}^{2s}) = 2s\, h(x')$$

\noindent by the properties of $\psi$. Also, $u(x)=0$ if $x\in\Sigma_1$. Thus, \eqref{eq:CGO_boundary_one} is reduced to

$$ 0= ( \mathcal B_s{u},r_2-r_1)_{L^2(\Sigma_1)} = 2s (h, r_2-r_1)_{L^2(\Sigma_1)}\;,  $$

\noindent which implies $r_1=r_2$ in $\Sigma_1$ by the arbitrary choice of $h$. 

\noindent 
As a consequence, this implies that $r_1$ satisfies the equation
\begin{align*}
(u,f)_{L^2(\Omega)} + (u, g)_{L^2(\Sigma_1)}
& = -(x_{n+1}^{1-2s} \nabla u, \nabla r_{1})_{L^2(\Omega)} + 2( x_{n+1}^{1-2s}\xi' \cdot \nabla'u, r_{1})_{L^2(\Omega)}\\
& \quad
+ (x_{n+1}^{1-2s} V  u, r_{1})_{L^2(\Omega)} + (q u, r_1)_{L^2(\Sigma_1)},
\end{align*}
for all $u \in \mathcal{C}$. Now by density of $\mathcal{C}$ in $H^1(\Omega, x_{n+1}^{1-2s})$ (see Proposition \ref{prop:dense}), this exactly corresponds to $r_1$ being a weak solution of the equation
\begin{align*}
\tilde{L}_{-\xi,V}^s r_1 & = f \mbox{ in } \Omega,\\
\lim\limits_{x_{n+1}\rightarrow 0} x_{n+1}^{1-2s} \p_{n+1} r_1 + q r_1 & = g \mbox{ on } \Sigma_1.
\end{align*}
Finally, we recall that since we proved that $r_1=r_2$ in $\Sigma_1$, formula \eqref{eq:CGO_estimates} now reads
$$\|r_1\|_{L^2(\Omega,x_{n+1}^{1-2s})} + |\xi'|^{-1} \|\nabla r_1\|_{L^2(\Omega, x_{n+1}^{1-2s})} + |\xi'|^{s-1}\|r_1\|_{L^2(\Sigma_1)} = \|T\|=O(|\xi'|^{-s})\,,$$
which yields the desired correction function $r:=r_1$ and the claimed estimates.
\end{proof}


With the construction of CGO solutions to \eqref{eq:CGO_model} in hand, we now turn to the associated inverse problem. Arguing as in Section \ref{sec:well_posed1}, it is possible to prove the well-posedness of the weak formulation of the problem \eqref{eq:CGO_model} outside of a discrete set of eigenvalues. More precisely, to obtain this we consider the associated bilinear form 
\begin{align*}
\tilde{B}_{q,V}(u,v):= \int\limits_{\Omega} x_{n+1}^{1-2s} \nabla u \cdot \nabla v dx  + \int\limits_{\Omega} V x_{n+1}^{1-2s} u v dx + \int\limits_{\Sigma_1} qu v dx',
\end{align*}
for $u,v \in H^{1}(\Omega, x_{n+1}^{1-2s})$. Further we investigate the Dirichlet problem \eqref{eq:CGO_model} for data $f$ belonging to the abstract space
\begin{align*}
R:= H^{1}(\Omega, x_{n+1}^{1-2s})/  H^{1}_{\Sigma_2,0}(\Omega, x_{n+1}^{1-2s}),
\end{align*}
\noindent endowed with the usual quotient topology
$$ \| f \|_R := \inf_{u \in f} \left\{ \|u\|_{H^{1}(\Omega, x_{n+1}^{1-2s})} \right\}. $$
\noindent This choice is motivated by the the observation that for all $u,v \in H^{1}(\Omega, x_{n+1}^{1-2s})$ we have for the corresponding remainder classes $[u], [v] \in R$
$$ [u] = [v] \quad \Leftrightarrow \quad u|_{\Sigma_2}=v|_{\Sigma_2}\;, $$
\noindent and thus the equivalence classes of $R$ can be interpreted as restrictions on $\Sigma_2$ of functions belonging to $H^{1}(\Omega, x_{n+1}^{1-2s})$. In view of this interpretation, one can make sense of the assertion $u|_{\Sigma_2} = f$, with $u\in H^1(\Omega,x_{n+1}^{1-2s})$ and $f\in R$, as equivalent to $u\in f$. Moreover, by the properties of the infimum for all $f\in R$ with $\|f\|_R > 0$ and $\epsilon > 0$, we can find $u\in H^1(\Omega, x_{n+1}^{1-2s})$ with $u|_{\Sigma_2} = f$ such that 
$$ \|u\|_{H^{1}(\Omega, x_{n+1}^{1-2s})} \leq \|f\|_{R}+\epsilon.$$
\noindent By just choosing $\epsilon \leq \|f\|_R$ we deduce that for all boundary data $f$ on $\Sigma_2$ there exists an extension $E_s(f) \in {H^{1}(\Omega, x_{n+1}^{1-2s})}$ such that 
$$  \|E_s(f)\|_{H^{1}(\Omega, x_{n+1}^{1-2s})} \leq 2\|f\|_{R}.$$
\noindent This lets us argue similarly as in Section \ref{sec:well_posed1}, and we obtain analogous well-posedness results.

 We denote the dual space of $R$ by $R^{\ast}$. In the following we assume that zero is not a Dirichlet eigenvalue and thus define for $f \in R$ a Dirichlet-to-Neumann operator $\tilde{\Lambda}_{q,V}: R \rightarrow R^{\ast}$ by setting
\begin{align*}
\langle \tilde{\Lambda}_{s,q,V} f , g \rangle_{R^{\ast}, R} = 
B_{q,V}(u_f, E_s g).
\end{align*}
Here $E_s g$ denotes a $H^{1}(\Omega, x_{n+1}^{1-2s})$ extension of the function $g \in R$. 
Relying on similar arguments as for the Dirichlet-to-Neumann maps studied in Section \ref{sec:well_posed1}, the map $\tilde{\Lambda}_{s,q,V}$ is continuous from $R$ into $R^{\ast}$.

With the CGO solutions available, we can now address the proof of Theorem \ref{thm:unique}.
Indeed, with the given special solutions, the solution to our inverse problem now follows from the Alessandrini identity.

\begin{proof}[Proof of Theorem \ref{thm:unique}]
Let $V:=V_1-V_2$ and $q:=q_1-q_2$. The assumption that $\Lambda_1 = \Lambda_2$ and the Alessandrini identity from Lemma \ref{lem:Aless} allow us to write that, for any solutions $u_1, u_2$ to \eqref{eq:Schroedinger},

$$ \int_{\mathbb R^{n+1}} \chi_{\Omega}V u_1 \overline{u_2} x_{n+1}^{1-2s} dx + \int_{\mathbb R^n} \chi_{\Sigma_1} q u_1  \overline{u_2} dx' =0\,.$$

We shall test this identity using our special CGO solutions. Fix $\xi, k$ as in Proposition \ref{CGO_construction} and let
\begin{align*}\begin{split}
u_1(x) &:= e^{\xi'\cdot x'}(e^{(ik'\cdot x' + i k_{n+1}x_{n+1}^{2s})/2}+r_1(x))\;, \\
u_2(x) &:= e^{\tilde{\xi}'\cdot x'}(e^{-(ik'\cdot x' + i k_{n+1}x_{n+1}^{2s})/2}+r_2(x))\;.
\end{split}\end{align*}
Here if $\xi' = \zeta_1 + i \zeta_2$, we set $\tilde{\xi'}:= -\zeta_1 + i \zeta_2$.
Substituting these into the above identity gives rise to
\begin{align*}\begin{split}
0 &= \int_{\mathbb R^{n+1}} \chi_{\Omega}V x_{n+1}^{1-2s} \left( r_1r_2 + (r_1+r_2)e^{(ik'\cdot x' + i k_{n+1}x_{n+1}^{2s})/2} + e^{ik'\cdot x' + i k_{n+1}x_{n+1}^{2s}} \right) dx + \\ 
& \quad +\int_{\mathbb R^n} \chi_{\Sigma_1} q  \left( r_1r_2 + (r_1+r_2)e^{ik'\cdot x'/2} + e^{ik'\cdot x'} \right) dx' \;.
\end{split}\end{align*}

We now aim to estimate the terms involving $r_1$ and $r_2$, showing that they can be dropped in the limit $|\xi'| \rightarrow \infty$. Recall from Proposition  \ref{CGO_construction} that
$$ \|r_j\|_{L^2(\Omega, x_{n+1}^{1-2s})} + \|r_j\|_{L^2(\Sigma_1)} = O(|\xi'|^{1-2s})$$
for $j=1,2$, and thus since $s\in(1/2,1)$ we have both $\|r_j\|_{L^2(\Omega, x_{n+1}^{1-2s})}\rightarrow 0$ and $\|r_j\|_{L^2(\Sigma_1)}\rightarrow 0$ as $|\xi'|\rightarrow \infty$. Therefore,
\begin{align*}\begin{split} 
\int_{\mathbb R^{n+1}} |\chi_{\Omega}V x_{n+1}^{1-2s} r_1r_2| dx & \leq \|V\|_{L^\infty(\Omega)}\|r_1\|_{L^2(\Omega, x_{n+1}^{1-2s})}\|r_2\|_{L^2(\Omega, x_{n+1}^{1-2s})} \rightarrow 0\;, \\
 \int_{\mathbb R^{n+1}} |e^{(ik'\cdot x' + i k_{n+1}x_{n+1}^{2s})/2} \chi_{\Omega}V x_{n+1}^{1-2s} r_1| dx & \leq \|V\|_{L^\infty(\Omega)}\|r_1\|_{L^2(\Omega, x_{n+1}^{1-2s})}\|x_{n+1}^{1/2-s}\|_{L^2(\Omega)} \rightarrow 0\;, \\
 \int_{\mathbb R^n} |\chi_{\Sigma_1} q  r_1r_2| dx' & \leq \|q\|_{L^\infty(\Sigma_1)}\|r_1\|_{L^2(\Sigma_1)}\|r_2\|_{L^2(\Sigma_1)} \rightarrow 0\;, \\
 \int_{\mathbb R^n} |e^{ik'\cdot x'/2}\chi_{\Sigma_1} q  r_1| dx' & \leq \|q\|_{L^\infty(\Sigma_1)}\|r_1\|_{L^2(\Sigma_1)}|\Sigma_1|^{1/2} \rightarrow 0,
\end{split}\end{align*}
as $|\xi'|\rightarrow \infty$, and similarly for the remaining terms. The Alessandrini identity is thus reduced to
$$ \int_{\mathbb R^{n+1}} \chi_{\Omega}V x_{n+1}^{1-2s}  e^{ik'\cdot x' + i k_{n+1}x_{n+1}^{2s}} dx + \int_{\mathbb R^n} \chi_{\Sigma_1} q  e^{ik'\cdot x'} dx' =0\;,$$
which after the change of variables $(y',y_{n+1}) = (x', x_{n+1}^{2s})$ in the first integral takes the form
\begin{equation}\label{Alex_2}
\int_{\mathbb R^{n+1}} \left(\frac{\chi_{\Omega}V}{2s}\right)(y',y_{n+1}^{1/2s}) y_{n+1}^{1/s-2} e^{ik\cdot y}dy + \int_{\mathbb R^n} \chi_{\Sigma_1} q  e^{ik'\cdot x'} dx' =0\;.
\end{equation}
\vspace{3mm}

Let $\mathcal{S} (\mathbb R^{n+1})$ and $\mathcal{S}' (\mathbb R^{n+1})$ respectively be the sets of Schwartz functions and tempered distributions over $\mathbb R^{n+1}$. Consider $\delta_{x_{n+1}}(0) \in \mathcal{S}' (\mathbb R^{n+1})$ defined by $$ \langle\delta_{x_{n+1}}(0), \phi\rangle = \int_{\mathbb R^n} \phi((x',0)) dx' $$  
for all $\phi \in \mathcal{S} (\mathbb R^{n+1})$. Then 
\begin{align*}
f(x) := \left(\frac{\chi_{\Omega}V}{2s}\right)(x',x_{n+1}^{1/2s}) x_{n+1}^{1/s-2} {\chi_{[0,\infty)}(x_{n+1})} + \delta_{x_{n+1}}(0) (\chi_{\Sigma_1} q)(x')
\end{align*}
where $\chi_{[0,\infty)}(x_{n+1})$ denotes the characteristic function of $[0,\infty)$ is also a tempered distribution, since for all $\phi \in \mathcal{S} (\mathbb R^{n+1})$ we have
\begin{align*}\begin{split}
|\langle f,\phi \rangle| & = \left| \int_{\mathbb R^{n+1}} \left(\frac{\chi_{\Omega}V}{2s}\right)(x',x_{n+1}^{1/2s}) x_{n+1}^{1/s-2}\chi_{[0,\infty)}(x_{n+1}) \phi(x) dx  +  \int_{\mathbb R^n} (\chi_{\Sigma_1} q)(x') \phi((x',0)) dx' \right| \\
& \leq \frac{\|V\|_{L^{\infty}(\Omega)}}{2s} \int_{\Omega} x_{n+1}^{1/s-2} |\phi(x)| dx + \|q\|_{L^\infty(\Sigma_1)} \int_{\Sigma_1} |\phi((x',0))| dx' \\
& \leq \|\phi\|_{L^\infty} \left(\frac{\|V\|_{L^{\infty}(\Omega)}}{2s} \int_{\Omega} x_{n+1}^{1/s-2} dx + \|q\|_{L^\infty(\Sigma_1)} |\Sigma_1| \right) < \infty\;.
\end{split}\end{align*}

The Fourier transform of $f$ belongs to $\mathcal{S}' (\mathbb R^{n+1})$ as well, and by definition it is the tempered distribution given by
\begin{align*}\begin{split}
\langle \hat f, \phi \rangle & = \langle \mathcal{F}\left[ \left(\frac{\chi_{\Omega}V}{2s}\right)(x',x_{n+1}^{1/2s}) x_{n+1}^{1/s-2} {\chi_{[0,\infty)}(x_{n+1})} \right](k),\phi(k) \rangle + \langle \delta_{x_{n+1}}(0) (\chi_{\Sigma_1} q)(x') , \hat\phi(x) \rangle \\ 
& =  \langle \mathcal{F}\left[ \left(\frac{\chi_{\Omega}V}{2s}\right)(x',x_{n+1}^{1/2s}) x_{n+1}^{1/s-2} {\chi_{[0,\infty)}(x_{n+1})}  \right](k),\phi(k) \rangle + \int_{\mathbb R^n} (\chi_{\Sigma_1} q)(x') \hat\phi ((x',0)) dx' \\
& = \int_{\mathbb R^{n+1}} \phi(k)\int_{\mathbb R^{n+1}} \left(\frac{\chi_{\Omega}V}{2s}\right)(x',x_{n+1}^{1/2s}) x_{n+1}^{1/s-2} {\chi_{[0,\infty)}(x_{n+1})} e^{ix\cdot k}dx dk \\
& \quad + \int_{\mathbb R^n} (\chi_{\Sigma_1} q)(x') \int_{\mathbb R^{n+1}} \phi(k) e^{ik'\cdot x'} dk dx' \\
& = \int_{\mathbb R^{n+1}} \phi(k) \left( \int_{{\R^{n+1}_+}} \left(\frac{\chi_{\Omega}V}{2s}\right)(x',x_{n+1}^{1/2s}) x_{n+1}^{1/s-2}   e^{ik\cdot x}dx + \int_{\mathbb R^n} (\chi_{\Sigma_1} q)(x')   e^{ik'\cdot x'} dx' \right) dk
\end{split}\end{align*}
for all $\phi \in \mathcal{S} (\mathbb R^{n+1})$, where for convenience of notation, we both use the notation $\hat{f}$ and $\F f$ to denote the Fourier transform. By \eqref{Alex_2} the last expression vanishes, which proves that $\hat f =0$. Now the Fourier inversion theorem for tempered distributions allows us to deduce that $\langle f,\phi \rangle =0$ for every $\phi \in \mathcal{S} (\mathbb R^{n+1})$. Testing this equality with an arbitrary function $\phi\in C^\infty_c(\Omega)$ we get
\begin{align*}
0 & = \langle \left(\frac{\chi_{\Omega}V}{2s}\right)(x',x_{n+1}^{1/2s}) x_{n+1}^{1/s-2} {\chi_{[0,\infty)}(x_{n+1})} + \delta_{x_{n+1}}(0) (\chi_{\Sigma_1} q)(x'),\phi\rangle\\
& = \int_{{\R^{n+1}_+}} \left(\frac{\chi_{\Omega}V}{2s}\right) x_{n+1}^{1/s-2} \phi \,dx\;,
\end{align*}
which by the arbitrary choice of $\phi$ implies $V=0$ in $\Omega$, and we are left with $f(x)= \delta_{x_{n+1}}(0) (\chi_{\Sigma_1} q)(x')$. 

Let now $\psi \in C^\infty_c(\Sigma_1)$, and consider $\eta \in C^\infty(\mathbb R)$ such that $\eta(x)=1$ if $x\in (-1,1)$ and $\eta(x)=0$ if $x\not\in(-2,2)$. Since it belongs to $C^\infty_c(\mathbb R^{n+1})$, the function $\phi(x) := \psi(x')\eta(x_{n+1})$ is a suitable test function for $\langle f,\phi \rangle =0$, and by using it we obtain

$$ 0= \langle f,\phi\rangle = \langle \delta_{x_{n+1}}(0) (\chi_{\Sigma_1} q)(x'), \psi(x')\eta(x_{n+1}) \rangle = \int_{\mathbb R^n} \chi_{\Sigma_1} q \psi \,dx'\;.$$
Eventually, by the arbitrary choice of $\psi$ we conclude that $q=0$ in $\Sigma_1$.
\end{proof}

As a corollary of this argument we remark that while for $s=\frac{1}{2}$ with the described method we cannot simultaneously prove uniqueness for the potentials $q$ and $V$ (due to the lack of the decay of $r$ on the boundary), this method still allows us to prove uniqueness for $V$ given a fixed potential $q$: 

\begin{cor}
\label{cor:unique}
Let $\Omega \subset \mathbb R^{n+1}_+$, $n\geq 2$, be an open, bounded and smooth domain. Assume that $\Sigma_1 := \p\Omega \cap \{x_{n+1}=0\}$ and $\Sigma_2 \subset \p\Omega\setminus\Sigma_1$ are two relatively open, non-empty subsets of the boundary such that $\overline{\Sigma_1\cup\Sigma_2}=\p\Omega$. Let $s= \frac{1}{2}$. If the potentials $q \in L^{\infty}(\Sigma_1)$ and $V_1, V_2 \in L^{\infty}(\Omega)$ relative to problem \eqref{eq:CGO_model} are such that $$\Lambda_1:= \Lambda_{s,V_1, q}=\Lambda_{s,V_2,q}=:\Lambda_2\;,$$ 
\noindent then $V_1 = V_2$.
\end{cor}

\begin{proof}
The proof follows that of Theorem \ref{thm:unique}, but it is significantly easier due to the lack of boundary terms. Again we let $V:=V_1-V_2$, but this time the Alessandrini identity from Lemma \ref{lem:Aless} reduces to simply 
$$ \int_{\mathbb R^{n+1}} \chi_{\Omega}V u_1 \overline{u_2} dx =0\,,$$

\noindent where $u_1, u_2$ solve \eqref{eq:Schroedinger}. Fix $\xi', k \in \R^{n+1}$ as in Lemma \ref{CGO_construction} with the modification from Remark \ref{rmk:CGO}, and for $\xi' = \zeta_1 + i \zeta_2$ set $\tilde{\xi'}:= -\zeta_1 + i \zeta_2$. Testing the equation above with the following CGOs
\begin{align*}
u_1(x) := e^{\xi'\cdot x'}(e^{ik\cdot x/2}+r_1(x))\;, \quad u_2(x) := e^{\tilde{\xi}'\cdot x'}(e^{-ik\cdot x/2}+r_2(x)),\;
\end{align*}
\noindent leads to
\begin{align*}\begin{split}
0 &= \int_{\mathbb R^{n+1}} \chi_{\Omega}V \left( r_1r_2 + (r_1+r_2)e^{ik\cdot x/2} + e^{ik\cdot x} \right) dx\;.
\end{split}\end{align*}
 
In our current case $s=1/2$, Proposition \ref{CGO_construction} does not grant any decay for the correction functions $r_j$ on the boundary; however, we will make use only of their decay estimate in the bulk. Given that $\|r_j\|_{L^2(\Omega)} =O(|\xi'|^{-1/2})$, by Cauchy-Schwarz
\begin{align*}\begin{split} 
\int_{\mathbb R^{n+1}} |\chi_{\Omega}V r_1r_2| dx & \leq \|V\|_{L^\infty(\Omega)}\|r_1\|_{L^2(\Omega)}\|r_2\|_{L^2(\Omega)} =O(|\xi'|^{-1}), \\
 \int_{\mathbb R^{n+1}} |e^{ik\cdot x/2} \chi_{\Omega}V r_j| dx & \leq \|V\|_{L^\infty(\Omega)}|\Omega|^{1/2}\|r_j\|_{L^2(\Omega)} =O(|\xi'|^{-1/2}) \;.
\end{split}\end{align*} 
Therefore, by finding the limit $|\xi'|\rightarrow \infty$ of the tested equation we obtain
$$ 0 = \int_{\mathbb R^{n+1}} \chi_{\Omega}V e^{ik\cdot x} dx = \mathcal{F}[\chi_\Omega V](k) $$
for all $k\in\mathbb R^{n+1}$. It now follows from the Fourier inversion theorem that $V=0$ on $\Omega$, that is, the potentials $V_1$ and $V_2$ must coincide.
\end{proof}

\begin{appendix}
\section{Proof of Proposition \ref{prop:dense}}

In this section, we provide the proof of Proposition \ref{prop:dense}. 
To this end, we begin by showing the following auxiliary result:

\begin{lem}
\label{lem:aux}
The set $C^{\infty}(\overline{\R^{n+1}_+})$ is dense in $H^1(\R^{n+1}_+, x_{n+1}^{1-2s})$.
\end{lem}

\begin{proof}[Proof of Lemma \ref{lem:aux}]
We consider $\varphi : \R^{n+1}\rightarrow \R$ such that $\supp(\varphi) \subset B_1^-(0)$, $\varphi \geq 0$ and $\int_{\R^{n+1}} \varphi dx =1$. Set $\varphi_{\epsilon}(x)= \epsilon^{-n-1} \varphi(\frac{x}{\epsilon})$. Further let $f\in H^1(\R^{n+1}_+, x_{n+1}^{1-2s})$. We construct a smooth  sequence $f_{\epsilon}$ such that $f_{\epsilon} \rightarrow f$ in $H^{1}(\R^{n+1}_{+},x_{n+1}^{1-2s})$.
To this end define $f_{\epsilon}(x):= (f \ast \varphi_{\epsilon})(x)$. Then, since $f\in L^{1}_{loc}(\R^{n+1}_+)$, we obtain that $f_{\epsilon}$ is smooth. Moreover, as a consequence of the maximal function estimate for weights in the Muckenhoupt class (see for instance Theorem 1.2 in \cite{K94} with the difference of working with half-balls instead of balls) $f_{\epsilon}, \ \nabla f_{\epsilon} = (\nabla f )_{\epsilon} \in L^2(\R^{n+1}_+,x_{n+1}^{1-2s})$. In order to prove the convergence, we only show $f_{\epsilon} \rightarrow f$ in $L^2(\R^{n+1}_+, x_{n+1}^{1-2s})$ (the statement for the gradient is then analogous) and begin by  collecting a number of auxiliary observations.
We note that for each $\Omega \subset \overline{\R^{n+1}_+}$, by the maximal function estimates we also have that
\begin{align}
\label{eq:max_func}
\|f_{\epsilon}\|_{L^2(\Omega,x_{n+1}^{1-2s})}
\leq  C \|f\|_{L^2(N(\Omega,\epsilon),x_{n+1}^{1-2s})}.
\end{align}
Here $N(\Omega, \epsilon)$ denotes an $\epsilon$ neighbourhood of $\Omega$ in $\overline{\R^{n+1}_+}$.
Now, since $f \in L^{2}(\R^{n+1}_+,x_{n+1}^{1-2s})$, for each $\delta>0$ there exists $R>1$ such that $\|f\|_{L^2(\R^{n+1}_+ \setminus \overline{B_{R}},x_{n+1}^{1-2s})} \leq \delta$ and thus by \eqref{eq:max_func} also $\|f_{\epsilon}\|_{L^2(\R^{n+1}_+ \setminus \overline{B_{R}},x_{n+1}^{1-2s})} \leq \delta$. 
Moreover, again by the integrability of $f$, there exists $\tilde{\delta}>0$ such that
\begin{align*}
\|f\|_{L^2(B_R \cap \{x_{n+1}\leq \tilde{\delta}\}, x_{n+1}^{1-2s})} + \|f_{\epsilon}\|_{L^2(B_R \cap \{x_{n+1}\leq \tilde{\delta}\}, x_{n+1}^{1-2s})} \leq \delta.
\end{align*}
Finally in $B_{2R} \cap \{x_{n+1}> \tilde{\delta}/2\}$ there exists a sequence $f_k \in C^{\infty}(\overline{B_{2R} \cap \{x_{n+1}> \tilde{\delta}/2\}})$ such that $f_k \rightarrow f$ in $L^2(\R^{n+1}_+, x_{n+1}^{1-2s})$.
Since $(f_k)_{\epsilon}:= f_k \ast \varphi_{\epsilon} \rightarrow f_k$ uniformly on compact sets, we may thus also conclude that
\begin{align*}
\|f_k-(f_k)_{\epsilon}\|_{L^2(B_R \cap \{x_{n+1}\geq \tilde{\delta} \},x_{n+1}^{1-2s})} \leq \delta.
\end{align*}
Also, by \eqref{eq:max_func}
\begin{align*}
\|f_{\epsilon}- (f_k)_{\epsilon}\|_{L^2(B_R \cap \{x_{n+1}\geq \tilde{\delta} \},x_{n+1}^{1-2s})} 
&= \|(f-f_k)_{\epsilon}\|_{L^2(B_R \cap \{x_{n+1}\geq \tilde{\delta} \},x_{n+1}^{1-2s})} \\
&\leq C \|f-f_k\|_{L^2(B_{2R} \cap \{x_{n+1}\geq \tilde{\delta}/2 \},x_{n+1}^{1-2s})} .
\end{align*}
Combining the above observations, infer that
\begin{align*}
\|f-f_{\epsilon}\|_{L^2(\R^{n+1}_+,x_{n+1}^{1-2s})}
&\leq \|f\|_{L^2(\R^{n+1}_+\setminus \overline{B_R},x_{n+1}^{1-2s})}  + \|f_{\epsilon}\|_{L^2(\R^{n+1}_+\setminus \overline{B_R},x_{n+1}^{1-2s})} \\ & \quad+ \|f\|_{L^2(B_R \cap \{ x_{n+1} \leq \tilde \delta \},x_{n+1}^{1-2s})} +  \|f_{\epsilon}\|_{L^2(B_R \cap \{ x_{n+1} \leq \tilde \delta \},x_{n+1}^{1-2s})} \\ & \quad+ \|f-f_{\epsilon}\|_{L^2(B_R \cap \{ x_{n+1} \geq \tilde \delta \},x_{n+1}^{1-2s})}  
\\ & \leq 3\delta + \|f-f_{\epsilon}\|_{L^2(B_R \cap \{x_{n+1}\geq \tilde{\delta} \},x_{n+1}^{1-2s})}\\
&\leq 3\delta + \|f_k-(f_k)_{\epsilon}\|_{L^2(B_R \cap \{x_{n+1}\geq \tilde{\delta} \},x_{n+1}^{1-2s})} + \|f-f_k\|_{L^2(B_R \cap \{x_{n+1}\geq \tilde{\delta} \},x_{n+1}^{1-2s})} \\
&\quad + \|f_{\epsilon}- (f_k)_{\epsilon}\|_{L^2(B_R \cap \{x_{n+1}\geq \tilde{\delta} \},x_{n+1}^{1-2s})} \\
&\leq 6 \delta . 
\end{align*}
Arguing analogously on the level of the derivative implies the claim.
\end{proof}

Next we define the following auxiliary set

\begin{align*}
C^{\infty}_{\Sigma_2}(\Omega):= \{f\in C^{\infty}(\overline{\Omega}): \ \exists \delta>0 \mbox{ s.t. }  f|_{N(\Sigma_{2},\delta)}=0 \}.
\end{align*}

Using this, we turn to the proof of the approximation result. 

\begin{proof}[Proof of Proposition \ref{prop:dense}]
Using Lemma \eqref{lem:aux}, we argue in three steps.\\

\emph{Step 1: Density of $\bigcup\limits_{\delta \in (0,\delta_0)} H^{1}_{N(\R^n \setminus \Sigma_1,\delta),0}(\R^{n+1}_+, x_{n+1}^{1-2s}) \subset H^{1}_{\R^n \setminus \Sigma_1,0 }(\R^{n+1}_+, x_{n+1}^{1-2s})$.}

This follows by rescaling: Indeed, by translation we may assume that $x=0$ is a center of the star-shaped set $\Sigma_1$.
Now let $u \in H^{1}_{\R^n \setminus \Sigma_1,0}(\R^{n+1}_+, x_{n+1}^{1-2s})$. Then, as $C^{\infty}(\overline{\R^{n+1}_+})$ is dense in $H^1(\R^{n+1}_+, x_{n+1}^{1-2s})$, there exists $(u_k)_{k\in \N} \subset C^{\infty}(\overline{\R^{n+1}_+})$ such that $u_k \rightarrow u$ in $H^1(\R^{n+1}_+, x_{n+1}^{1-2s})$. Since $\Sigma_1$ is star-shaped, if we define $d:= \dist(0,\p\Sigma_1)$ and $u_{\delta}(x):= u\left( \frac{d}{d-\delta} x\right)$ for $\delta \in (0,d)$, then we have that $u_{\delta} \in H^{1}_{ N(\R^n \setminus \Sigma_1,\delta),0}(\R^{n+1}_+, x_{n+1}^{1-2s})$ and
\begin{align}
\label{eq:approx_1}
\begin{split}
\|u_{\delta}- u\|_{H^1(\R^{n+1}_+, x_{n+1}^{1-2s})}
& \leq \|u_{\delta}- u_k\left( \frac{d}{d-\delta} \cdot\right)\|_{H^1(\R^{n+1}_+, x_{n+1}^{1-2s})}
+ \|u_k\left( \frac{d}{d-\delta} \cdot\right)-u_k\|_{H^1(\R^{n+1}_+, x_{n+1}^{1-2s})}\\
& \quad + \|u - u_k \|_{H^1(\R^{n+1}_+, x_{n+1}^{1-2s})}.
\end{split}
\end{align}
Now, the first and third contributions in \eqref{eq:approx_1} converge to zero by definition of $u_k$ as approximations to $u$. The middle right hand side contribution converges to zero by the assumed regularity of $u_k$ and a Taylor approximation up to order one.

Using a partition of unity and straightening out the boundary by a suitable diffeomorphism this also implies that 
$\bigcup\limits_{\delta \in (0,\delta_0)} H^{1}_{N(\partial \Omega \setminus \Sigma_1,\delta),0}(\Omega, x_{n+1}^{1-2s}) \subset H^{1}_{\Sigma_2,0}(\Omega, x_{n+1}^{1-2s})$ is dense.
\medskip

\emph{Step 2: Density of $C^{\infty}_{\Sigma_2}(\Omega)  \subset H^{1}_{\partial \Omega \setminus \Sigma_1,0}(\Omega, x_{n+1}^{1-2s})$}.

By Step 1 it suffices to prove that $\bigcup\limits_{\epsilon \in (0,\delta/2)} C^{\infty}_{N(\Sigma_2,\delta+\epsilon )}(\Omega) \subset H^1_{N(\Sigma_2,\delta/2),0}(\Omega, x_{n+1}^{1-2s})$ is dense in $H^1_{N(\Sigma_2,\delta),0}(\Omega, x_{n+1}^{1-2s}) $ for all sufficiently small $\delta>0$.

By virtue of a partition of unity and by straightening out the boundary, it suffices to consider $u \in H^1(\R^{n+1}_+, x_{n+1}^{1-2s})$ satisfying one of the following two cases:
\begin{itemize}
\item[(i)] $u|_{\R^n}$ has compact, but non-trivial support in $\Sigma_1$,
\item[(ii)] $u|_{\R^n}=0$,
\end{itemize}
and to prove a corresponding approximation result in these cases.
The first case arises when working with a patch of the partition of unity which includes $N(\Sigma_1,\epsilon)$, the second occurs for any other patch (we remark that without loss of generality, it is possible to arrange for this).

\emph{Step 2a: Case (i).}
For case (i) we in turn argue in two steps. 

\emph{Step 2a, part 1; constant modification at $x_{n+1}=0$.}
First we define the function $\tilde{u}_{\epsilon}$ such that $\tilde{u}_{\epsilon}(x',x_{n+1}) = u(x',0)$ for $x_{n+1} \in [0,2\epsilon]$ and $\tilde{u}_{\epsilon}(x',x_{n+1})=0$ for $x_{n+1}>2\epsilon$. We observe that $\tilde{u}_{\epsilon} = 0$ in $\R^{n+1}_+ \setminus (\Sigma_1 \times [0,2\epsilon])$.
We further consider $\eta : [0,\infty) \rightarrow [0,1]$ with $\eta \in C^{\infty}([0,\infty))$, $\eta(t) = 1$ on $[0,\frac{1}{2}]$, $\supp(\eta) \subset [0,2]$ and $|\nabla \eta|\leq C$. Based on this we define $\eta_{\epsilon}(t):= \eta(\frac{t}{\epsilon})$ and $u_{\epsilon}:= \eta_{\epsilon}(x_{n+1}) \tilde{u}_{\epsilon}(x) +(1-\eta_{\epsilon}(x_{n+1})) u(x)$. We claim that $u_{\epsilon} \rightarrow u$ in $H^1(\R^{n+1}_+, x_{n+1}^{1-2s})$. 

To this end, we observe that
\begin{align*}
\|u_{\epsilon}-u\|_{L^2(\R^{n+1}_+, x_{n+1}^{1-2s})}
&= \|\eta_{\epsilon} (\tilde{u}_{\epsilon}-u) \|_{L^2(\R^n \times [0,2\epsilon], x_{n+1}^{1-2s})}\\
&\leq \|u(x',0)- u(x)\|_{L^2(\R^n \times [0,2\epsilon], x_{n+1}^{1-2s})} \rightarrow 0,
\end{align*}
by the integrability of $u(x',0)-u(x)$.
For the derivative we note that
\begin{align}
\label{eq:expr_grad}
\nabla u - \nabla u_{\epsilon} =
(u-\tilde{u}_{\epsilon})\nabla \eta_{\epsilon}
+ \eta_{\epsilon} \nabla (u- \tilde{u}_{\epsilon}).
\end{align}
Due to the support conditions for $\eta_{\epsilon}$ and by the fact that $\nabla (u- \tilde{u}_{\epsilon}) \in L^2({\R^n} \times [0,2\epsilon_0),x_{n+1}^{1-2s})$ for some fixed $\epsilon_0>0$, we have that 
$$\|\eta_{\epsilon} \nabla (u- \tilde{u}_{\epsilon})\|_{L^2(\R^{n+1}_+, x_{n+1}^{1-2s})} \rightarrow 0$$ as $\epsilon \rightarrow 0$.

For the first contribution in the expression for the gradient \eqref{eq:expr_grad}, we use the fundamental theorem (which makes use of the approximation statement from Lemma \ref{lem:aux}):
We have that 
\begin{align*}
|(u-\tilde{u}_{\epsilon})(x)|
= |u(x', x_{n+1})- u(x',0)|
\leq \int\limits_{0}^{x_{n+1}} |\p_{n+1} u(x',t)| d t.
\end{align*}
Thus, using Hölder's inequality, we obtain
\begin{align*}
|(u-\tilde{u}_{\epsilon})(x)|^2
\leq x_{n+1}^{2s} \int\limits_{0}^{x_{n+1}} t^{1-2s}|\p_{n+1} u(x',t)|^2 dt.
\end{align*}
As a consequence, an integration yields
\begin{align*}
\|(u-\tilde{u}_{\epsilon})\nabla \eta_{\epsilon}\|_{L^2(\R^{n+1}_+, x_{n+1}^{1-2s})}
&\leq C \epsilon^{-1}\|u-\tilde{u}_{\epsilon}\|_{L^2(\R^n \times [0,\epsilon], x_{n+1}^{1-2s})}
\\ & \leq C \epsilon^{s-1} \left\| \left( \int\limits_{0}^{x_{n+1}} t^{1-2s}|\p_{n+1} u(x',t)|^2 dt \right)^{\frac{1}{2}}\right\|_{L^2(\R^n \times [0,\epsilon], x_{n+1}^{1-2s})}
\\ & \leq C \epsilon^{s-1} \left( \int_0^\epsilon x_{n+1}^{1-2s} dx_{n+1}\right)^{1/2} \left\| \left( \int\limits_{0}^{\epsilon} t^{1-2s}|\p_{n+1} u(x',t)|^2 dt \right)^{\frac{1}{2}}\right\|_{L^2(\R^n)}
\\& \leq C_s \epsilon^{s-1} \epsilon^{1-s} \|\nabla u\|_{L^2(\R^n \times [0,\epsilon], x_{n+1}^{1-2s})}\\
& = C_s \|\nabla u\|_{L^2(\R^n \times [0,\epsilon], x_{n+1}^{1-2s})} \rightarrow 0 \mbox{ as } \epsilon \rightarrow 0,
\end{align*}
since $\nabla u \in H^1(\R^{n+1}_+ , x_{n+1}^{1-2s})$. This proves the claimed convergence $u_{\epsilon } \rightarrow u$.

\medskip

\emph{Step 2a, part 2, mollification.}
As a second step, we start with a function $u_{\epsilon}$ as obtained in Step 2a, part 1 which by a slight abuse of notation (by dropping the index) we denote by $u$. For this function, we now consider $u_{\delta}(x):=u\ast \varphi_{\delta}(x)$, where $\delta \in (0,\epsilon)$ and $\varphi_{\delta}(x):= \delta^{-n-1}\varphi(\frac{x}{\delta})$ with $\int\limits_{\R^{n+1}_+}\varphi(y)dy = 1$, $\varphi \in C^{\infty}(\R^{n+1}_+)$ is a mollifier supported in {$B_1^-$}. By the properties of the function $u$ (in particular, recall that $u= 0$ in $(\R^n \setminus \Sigma_1) \times [0,\epsilon]$), for $\delta>0$ sufficiently small, the function $u_{\delta}\in C^{\infty}(\R^{n+1}_+)\cap H^{1}(\R^{n+1}_+, x_{n+1}^{1-2s})$ then satisfies that $\supp(u_{\delta}|_{\R^{n}}) \subset N(\Sigma_1, \delta)$ and $u_{\delta} \rightarrow u$ in $H^1(\R^{n+1}_+, x_{n+1}^{1-2s})$. 

Combining both steps from Steps 2a by means of a diagonal argument then implies the claim for case (i).

\medskip

\emph{Step 2b: The case (ii).}
Now for case (ii) we argue as in the classical case, but replace the trace inequalities by correspondingly weighted ones; we refer to \cite[Chapter 5.5, Theorem 2]{E10}. We present some of the details for completeness. First by the density of $C^{\infty}(\overline{\R^{n+1}_+}) \subset H^1(\R^{n+1}_+, x_{n+1}^{1-2s})$ there exists a sequence $(u_m)_{m\in \N} \subset C^{\infty}(\overline{\R^{n+1}_+})$ such that $u_m\rightarrow u$ in $H^1(\R^{n+1}_+, x_{n+1}^{1-2s})$. Due to trace estimates similarly as in Lemma \ref{lem:boundary_trace_s} and the fact that $u|_{\R^n}=0$ we have $u_m|_{\R^n} \rightarrow 0$. Now by the fundamental theorem we obtain
\begin{align*}
|u_m(x',x_{n+1})| \leq |u_m(x',0)| + \int\limits_{0}^{x_{n+1}} |D u_m(x',t)| dt.
\end{align*}
Integrating and applying Hölder's inequality implies that
\begin{align*}
\|u_m(\cdot, x_{n+1})\|_{L^2(\R^n)}^2 
\leq C(\|u_m(\cdot, 0)\|_{L^2(\R^n)}^2 + x_{n+1}^{2s} \int\limits_{0}^{x_{n+1}} t^{1-2s} \|\nabla u_m(\cdot, t)\|^2_{L^2(\R^n)} d t).
\end{align*}
In particular, for $m \rightarrow \infty$, by the vanishing of the trace of $u$, we arrive at
\begin{align}
\label{eq:trace}
\|u(\cdot, x_{n+1})\|_{L^2(\R^n)}^2 
\leq C x_{n+1}^{2s} \int\limits_{0}^{x_{n+1}} t^{1-2s} \|\nabla u(\cdot, t)\|^2_{L^2(\R^n)} d t.
\end{align}
We now define 
\begin{align*}
w_m := u (1-\zeta_m),
\end{align*}
where $\zeta_m(x) := \zeta(m x_{n+1})$ and $\zeta\in C^\infty(\mathbb R)$ is such that  $\zeta(t) = 1$ for $t\in [0,1]$ and $\zeta = 0$ on $(2,\infty)$ and $0\leq \zeta\leq 1$.
We obtain
\begin{align*}
\p_{n+1} w_m &= (1-\zeta_m) \p_{n+1} u - m u \zeta'|_{mx_{n+1}},\\
\p_j w_m & = (1-\zeta_m) \p_j u \mbox{ for all } j\in \{1,\dots,n\}.
\end{align*}
Thus, 
\begin{align*}
\|\nabla (w_m-u_m)\|_{L^2(\R^{n+1}_+, x_{n+1}^{1-2s})}^2
\leq C \|\zeta_m \nabla u\|_{L^2(\R^{n+1},x_{n+1}^{1-2s})}^2
+ C m^2 \int\limits_{0}^{2/m} \int\limits_{\R^n} x_{n+1}^{1-2s} |u|^2 dx' dx_{n+1}.
\end{align*}
By construction, the first term converges to zero, as $\zeta_m \neq 0$ only for $x_{n+1}\in (0,2/m)$. For the second contribution we use \eqref{eq:trace}. This yields

\begin{align*}
m^2 \int\limits_{0}^{2/m} \int\limits_{\R^n} x_{n+1}^{1-2s} |u|^2 dx' dx_{n+1} & =  m^2 \int\limits_{0}^{2/m}  x_{n+1}^{1-2s} \|u(\cdot, x_{n+1})\|_{L^2(\R^n)}^2   dx_{n+1} 
\\ & \leq C m^2 \int\limits_{0}^{2/m}  x_{n+1} \int\limits_{0}^{x_{n+1}} t^{1-2s} \|\nabla u(\cdot, t)\|^2_{L^2(\R^n)} d t  dx_{n+1} 
\\ & \leq C m^2 \left(\int\limits_{0}^{2/m}  x_{n+1} dx_{n+1}\right)\int\limits_{0}^{2/m} t^{1-2s} \|\nabla u(\cdot, t)\|^2_{L^2(\R^n)} d t  
\\ & \leq C \|\nabla u\|_{L^2(\R^n \times [0,2/m], x_{n+1}^{1-2s})} \rightarrow 0 \mbox{ as } m \rightarrow \infty.
\end{align*}

\medskip

\emph{Step 3: Density of $\tilde{\mathcal{C}} \subset H^{1}_{\Sigma_1,0}(\Omega, x_{n+1}^{1-2s})$.} 

Let $u \in C^{\infty}_{\Sigma_2}(\Omega)$. We now approximate this function by a function of the desired structure.
Working in boundary normal coordinates $x = x' + t \nu(x')$ we define $\tilde{u}_{\epsilon}(x):= u(x')$ for $x \in  \partial \Omega_{2\epsilon}$. Let now $\eta_{\epsilon}$ be a smooth cut-off function which is equal to one in $\partial \Omega_{\epsilon}$ supported in $\partial \Omega_{2\epsilon}$ with $|\nabla' \eta_{\epsilon}| \leq C $ and $|\p_{\nu}\eta_{\epsilon}| \leq \frac{C}{\epsilon}$.
We then set $u_{\epsilon}(x):= \eta_{\epsilon}(x) \tilde{u}_{\epsilon}(x) + (1-\eta_{\epsilon})u(x)$.
Then,
\begin{align*}
\|u-u_{\epsilon}\|_{L^2(\Omega, x_{n+1}^{1-2s})}
= \|\eta_{\epsilon} (u-\tilde{u}_{\epsilon})\|_{L^2(\Omega, x_{n+1}^{1-2s})}.
\end{align*}
Since $u \in C^{\infty}(\Omega)$, we have $|\eta_{\epsilon}(x)||u(x)-\tilde{u}_{\epsilon}(x)| \leq C \sup\limits_{x\in \supp(\eta_{\epsilon})} |\p_t u(x)| t \leq C \epsilon $. Thus, 
\begin{align*}
\|u-u_{\epsilon}\|_{L^2(\Omega, x_{n+1}^{1-2s})}
\leq C_s \epsilon \epsilon^{1-s}.
\end{align*}

For the derivative we note that 
\begin{align*}
\|\nabla(u-u_{\epsilon})\|_{L^2(\Omega, x_{n+1}^{1-2s})}
& = \|\nabla[\eta_{\epsilon} (u-\tilde{u}_{\epsilon})]\|_{L^2(\Omega, x_{n+1}^{1-2s})}\\
& \leq  \|(u-\tilde{u}_{\epsilon})(\nabla \eta_{\epsilon}) \|_{L^2(\Omega, x_{n+1}^{1-2s})}
+ \| \eta_{\epsilon} \nabla (u-\tilde{u}_{\epsilon}) \|_{L^2(\Omega, x_{n+1}^{1-2s})}.
\end{align*}
Now using that $|\nabla \eta_{\epsilon}| \leq C \epsilon^{-1}$, $|u-\tilde{u}_{\epsilon}| \leq C \epsilon$ and $|\nabla (u- \tilde{u}_{\epsilon})| \leq C$, we obtain
\begin{align*}
\|\nabla(u-u_{\epsilon})\|_{L^2(\Omega, x_{n+1}^{1-2s})}
& \leq  \|(u-\tilde{u}_{\epsilon})(\nabla \eta_{\epsilon}) \|_{L^2(\Omega, x_{n+1}^{1-2s})}
+ \| \eta_{\epsilon} \nabla (u-\tilde{u}_{\epsilon}) \|_{L^2(\Omega, x_{n+1}^{1-2s})}\\
& \leq C \omega_s(\supp(\eta_{\epsilon}))^{\frac{1}{2}} \leq C_s \epsilon^{1-s},
\end{align*}
where for $\Omega' \subset \R^{n+1}_+$ measurable $\omega_s(\Omega'):= \int\limits_{\Omega'} x_{n+1}^{1-2s} dx$.
We note that the function $u_{\epsilon}$ has the desired property defining $\tilde{\mathcal{C}}$. Indeed, by the construction of $\tilde{u}_{\epsilon}$ we have $u_{\epsilon}= 0$ on $\Sigma_2$ and by construction of $\tilde{u}_{\epsilon}$ and of $\eta_{\epsilon}$ we also have $\p_{\nu} \tilde{u}_{\epsilon} = 0$ on $\partial \Omega$. 
It remains to argue that $\p_{n+1} u_{\epsilon}=0$ in $N(\Sigma_1, \delta) \times [0,\delta)$ for some $\delta>0$ small. This on the one hand follows from the fact that in $\Sigma_1 \times [0,\epsilon/2]$ the boundary normal coordinates are simply Euclidean coordinates $x=(x', x_{n+1})$ and that the function $\tilde{u}_{\epsilon}$ does not depend on the $x_{n+1}$ variable there by definition. On the other hand, we also have that in $N(\Sigma_2, \tilde{\epsilon})$ for some $\tilde{\epsilon}>0$ the function $u \in C^{\infty}_{\Sigma_2}(\Omega)$ satisfies $u = 0$. As a consequence, the function $\tilde{u}_{\epsilon}(x)=0$ in a set $\{x\in \Omega: x = x' + t \nu, \ x' \in N(\Sigma_1,\delta)\setminus\Sigma_1, \ t \in [0,2\delta]\}$ for some small $\delta>0$. This however implies that $\nabla \tilde{u}_{\epsilon}=0$ on this set, which entails that $\p_{n+1} u_{\epsilon}=0$ also in a set $N(\Sigma_1,\delta) \times [0,\delta)$.

Combining all the steps from above by a diagonal argument  concludes the proof.
\end{proof}

\end{appendix}

\section*{Acknowledgements}
This project was started during a visit of G. Covi to the MPI MIS. Both authors would like to thank the MPI MIS for the great working environment. Both authors would also like to thank Mikko Salo for pointing out the articles \cite{C14, C15} and some of the literature on the inverse Robin problem to them. G. Covi was partially supported by the European Research Council under Horizon 2020 (ERC CoG 770924).

\bibliographystyle{alpha}
\bibliography{citations_Lipschitz_updated_6}

\end{document}